\documentclass[11pt]{amsart}
\usepackage{times}
\usepackage[T1]{fontenc}
\usepackage{amssymb, amsthm, amsmath}
\usepackage{mathrsfs}

\usepackage[active]{srcltx}

\usepackage{todonotes}

\usepackage{setspace}
\usepackage{geometry}

\usepackage{hyperref}
\DeclareMathOperator{\dM}{DM}
\DeclareMathOperator{\acl}{acl}
\DeclareMathOperator{\dcl}{dcl} 
 
\DeclareMathOperator{\aut}{Aut} \DeclareMathOperator{\id}{Id}

 \DeclareMathOperator{\dom}{dom}

\DeclareMathOperator{\tp}{tp}

\DeclareMathOperator{\mr}{RM}

\DeclareMathOperator{\rv}{rv}

\newtheorem{theorem}{Theorem}[section]

\newtheorem{claim}{Claim}[theorem]

\newtheorem{corollary}[theorem]{Corollary}

\newtheorem{fact}[theorem]{Fact}
\newtheorem{lemma}[theorem]{Lemma}

\newtheorem{proposition}[theorem]{Proposition}
\newtheorem*{gen-dif}{\fbox{{\large A}} \hypertarget{Agen-dif}{Gen-Dif}}

\newtheorem*{min-balln}{\fbox{{\large A}} \hypertarget{Amin-ball}{Cballs}}

\theoremstyle{definition}
\newtheorem{definition}[theorem]{Definition}
\newtheorem{example}[theorem]{Example}
\newtheorem{remark}[theorem]{Remark}
\newtheorem{question}[theorem]{Question}

\newcommand{\Rr}{{\mathbb{R}}}
\newcommand{\Nn}{{\mathbb{N}}}
\newcommand{\Qq}{{\mathbb{Q}}}

\newcommand{\Zz}{{\mathbb {Z}}}

\newcommand{\m}{\textbf{m}}
\newcommand{\bk}{\textbf{k}}

\newcommand{\CB}{{\mathcal B}}

\newcommand{\CL}{{\mathcal L}}
\newcommand{\CK}{{\mathcal K}}
\newcommand{\CN}{{\mathcal N}}

\newcommand{\CR}{{\mathcal R}}
\newcommand{\CM}{{\mathcal M}}

\newcommand{\CO}{{\mathcal O}}

\newcommand{\0}{\emptyset}

\renewcommand{\phi}{\varphi}

\newcommand{\RV}{\mathrm{RV}}

\def\la{\langle}
\def\ra{\rangle}
\def\qp{\mathbb Q_p}

\def\dpr{\mathrm{dp\text{-}rk}}
\def\sub{\subseteq}

\newenvironment{claimproof}[1][\proofname]
  {%
    \proof[#1]%
  }
  {%
    \endproof%
  }

\date{January 2022}

\title{On groups  interpretable in various valued fields}
\author{Yatir Halevi}
\address{Department of Mathematics\\ University of Haifa\\ 199 Abba Khoushy Avenue \\ Haifa \\Israel}
\address{Department of Mathematics, Ben Gurion University of the Negev, Be'er-Sheva 84105, Israel}
\email{ybenarih@campus.haifa.ac.il}

\author{Assaf Hasson}
\address{Department of Mathematics, Ben Gurion University of the Negev, Be'er-Sheva 84105, Israel}
\email{hassonas@math.bgu.ac.il}

\author{Ya'acov Peterzil}
\address{Department of Mathematics, University of Haifa, Haifa, Israel}
\email{kobi@math.haifa.ac.il}

\date{March 2024}

\begin{document}

\thanks{The first author was partially supported by ISF grant No. 555/21 and 290/19. The second author was supported by ISF grant No. 555/21. The third author was supported by ISF grant No. 290/19.}
\maketitle

\begin{abstract}
    We study infinite groups interpretable in three families of valued fields:  $V$-minimal, power bounded $T$-convex, and  $p$-adically closed fields. We show that every such  group $G$ has unbounded exponent and that if  $G$ is dp-minimal then it is abelian-by-finite. 
    
    Along the way, we associate with any infinite interpretable  group an infinite type-definable subgroup which is definably isomorphic to a group in one of four distinguished sorts: the underlying valued field $K$, its residue field $\bk$ (when infinite), its value group $\Gamma$,  or $K/\CO$, where $\CO$ is the valuation ring.
    
    Our work uses and extends techniques developed in \cite{HaHaPeVF} to circumvent elimination of imaginaries.  
\end{abstract}

\tableofcontents

\section{Introduction}

We continue our work from \cite{HaHaPeVF}, where we studied fields interpretable in a variety of valued fields, and extend our investigation to interpretable groups. To recall, in \cite{HaHaPeVF} we considered interpretable objects, namely quotients of definable sets by definable equivalence relations, in valued fields, with a focus on interpretable fields in  three families of valued fields: (i) V-minimal (i.e. certain expansions of algebraically closed valued fields of residue characteristic $0$), (ii)  $T$-convex expansions of power bounded o-minimal structures (i.e. certain expansions of  real closed valued fields) and (iii)  certain expansions of $p$-adically closed fields, namely of fields that are elementarily equivalent to finite extensions of $\mathbb{Q}_p$).

The key idea from \cite{HaHaPeVF} was to bypass results on elimination of imaginaries  and replace them with a reduction to four {\em distinguished} sorts: the valued field $K$ itself, the residue field $\bk$ (when infinite), the value group $\Gamma$, and the quotient $K/\CO$, where $\CO$ is the valuation ring. Applying techniques similar to those developed in \cite{HaHaPeVF} we introduce here the following notion: We say that $G$ is {\em locally  strongly internal to a set $D$} if there exists a definable infinite $X\sub G$ and a definable injective $f:X\to D^n$, for some $n$.

Using this notion we state our  main theorem\footnote{For brevity, we shall use ``power bounded $T$-convex expansions'' instead of ``$T$-convex expansions of power bounded o-minimal structures''}:

% \begin{theorem}[Section \ref{S: final}] \label{intro-2}  Let $K$ be a valued field  of characteristic $0$ and assume that either (i) $\CK$ is a $V$-minimal expansion of $K$ (in particular, $K$ is an algebraically closed valued field of residue characteristic $0$), (ii) $\CK$ is a power bounded $T$-convex expansion of $K$ (so $K$ is a real closed field with a convex valuation ring), or (iii) $\CK=K$ is a $p$-adically closed field.

% Let $G$ be an infinite group interpretable in $\CK$. Then, after possibly replacing $G$ with  a quotient by a finite normal subgroup,  there is an infinite type-definable subgroup $\nu\le G$ of unbounded exponent such that one of the following holds:
% \begin{enumerate}
%     \item $\nu$ is definably isomorphic to a group type definable in either $K$, or $\bk$, or
% \item $\nu$ is definably isomorphic to a type definable subgroup of $\la \Gamma^n,+\ra$, or 

% \item there is a definable group $H$, $\nu\vdash H\sub G$, such that $H$ is definably isomorphic to a  subgroup of $\la (K/\CO)^n,+\ra$. 
% \end{enumerate}
% \end{theorem}

 \begin{theorem}[Section \ref{S: final}] \label{intro-2}  Let $K$ be a valued field  of characteristic $0$ and assume that either (i) $\CK$ is a $V$-minimal expansion of $K$ (ii) $\CK$ is a power bounded $T$-convex expansion of  $K$, or (iii) $K$ is a $p$-adically closed field and $\CK=K$.

Let $G$ be an infinite group interpretable in $\CK$. Then,  after possibly replacing $G$ with a quotient by a finite normal subgroup, there exists a distinguished sort $D\in \{K,\bk,\Gamma,K/\CO\}$, such that $G$ is locally strongly internal to $D$ and for every such $D$   there is an infinite type-definable  subgroup $\nu_D$ of $G$,  such that
\begin{enumerate}
    \item If $D=K$ or $D=\bk$  then $\nu_D$ is definably isomorphic to a type-definable group  $\nu_D'\vdash K^n$, or $\nu_D'\vdash \bk^n$, respectively, for some $n$. 
%    \item If $D=\bk$ then $\nu_D$ is definably isomorphic to a type definable group  $\nu_D'\vdash \bk^n$, for some $n$.
\item If $D=\Gamma$ then $\nu_{D}$ is definably isomorphic to a type-definable subgroup of $\la \Gamma^n,+\ra$, for some $n$.

\item If  $D=K/\CO$  then there is a definable subgroup $H$ of G,  $\nu_{D}\vdash H$, such that $H$ is definably isomorphic to a  subgroup of $\la (K/\CO)^n,+\ra$, for some $n$. 
\end{enumerate}
\end{theorem}

\begin{remark}\label{R: intro thm}
\begin{enumerate}
    \item By a ``type-definable'' subgroup of $G$ we mean a partial type over $K$ (consisting possibly of infinitely many formulas) whose realization in every elementary extension $\widehat \CK$ of $\CK$ is a subgroup of $G(\widehat \CK)$.  
    A type-definable group $\nu$ is \textit{infinite } if $\nu(\widehat \CK)$ is infinite for some $\widehat K\succ \CK$. 
We write $\nu\vdash S$ to mean that the  formula defining $S$ belongs to the type $\nu$.     
    %We note that $\nu$ is infinite if and only if $\dpr(\nu)>0$.   

   \item As we show in Proposition \ref{P: nu}, the group $\nu_D$  above  has various invariance properties. In addition, its dp-rank equals $\max \dpr(X)$, as $X\sub G$ varies over all definable subsets of $G$ for which there is some finite-to-one $f:X\to D^n$, for some $n$, 
    %\item For $D$ as in the previosu clause, after passing to a quotient by a finite normal subgroup, if needed, the partial type $\nu_D$ is invariant under conjugation by elements of $G(\CK)$ (even though its realization in  $\widehat \CK$ need not be a normal subgroup of $G(\widehat \CK)$, for an elementary extension $\widehat \CK$). 
\end{enumerate}

    % Moreover, as we show in Proposition \ref{P: nu}, after modding out by the finite normal subgroup, the partial type $\nu$ is invariant under conjugation by elements of $G(\CK)$ (even though its realization in  $\widehat \CK$ need not be a normal subgroup of $G(\widehat \CK)$).

\end{remark}

% In practice, we first construct $\nu$ satisfying conditions (1) or (2) of the theorem, and exploit this information (as well as some properties of the construction) to deduce that $\nu$ has unbounded exponent. 

% \textcolor{red}{K:The next sentence seems technical here. shall we move, or omit?} In the appendix we show that, in fact, when $\CK$ is power-bounded $T$-convex and $G$ is dp-minimal (namely, $\dpr(G)=1$),  $\nu$ is a type-definable ordered group. 

Though the above  theorem may seem technical, it gives a new set of invariants for studying interpretable groups. In a subsequent paper, \cite{HaHaPeSSG}, we use this result to classify definably simple non-abelian and definably semisimple groups  interpretable in the same classes of fields. In the present paper, we give more direct global applications of the local data provided by our main result: 

\begin{theorem}[Section \ref{S: final}] \label{intro-1}
 Let $K$ be a valued field  of characteristic $0$ assume that either (i) $\CK$ is a $V$-minimal expansion of $K$, (ii) $\CK$ is a power bounded $T$-convex expansion of $K$, or (iii) $\CK=K$ is $p$-adically closed. Let $G$ be an infinite group interpretable in $\CK$. Then (1) $G$ has unbounded exponent and (2) If $G$ is   dp-minimal  then it  is abelian-by-finite. 
\end{theorem}

\begin{remark} Theorem \ref{intro-1}  implies, in particular, that every 1-dimensional group $G\sub K^n$  definable in any of the above structures $\CK$ is abelian-by-finite. In $p$-adically closed fields, this was proved by  Pillay and Yao, \cite{PiNi}.  

Since the value group in  a $p$-adically closed field is a $\Zz$-group (i.e., a model of  Presburger Arithmetic) the above theorem also covers the work of Onshuus and Vicaria \cite[Theorem 1.1]{OnVi} on dp-minimal groups interpretable in $\Zz$-groups. 

Finally,  the above theorem complements  the work of Simonetta, \cite{Simonetta}, who gave an example of a dp-minimal  group interpretable in an algebraically closed valued field of characteristic $p$ which is nilpotent  of class $2$ (so not abelian-by-finite).
\end{remark}

Finding a framework that would allow us to avoid  repetition of proofs, as we move across the different settings and the different distinguished sorts, turned out to be one of the challenges of the present work. In \cite{HaHaPeVF} much of the work was carried out in SW-uniformities,  a framework of dp-minimal uniform structures introduced by Simon and Walsberg in \cite{SimWal}  (see below the precise definition). SW-uniformities generalise o-minimal expansions of groups and share several well known tameness properties of the latter setting. In particular,  the topology associated with these structures is Hausdorff and non-discrete.  This framework fits  the sorts $K$ in all settings, the sorts $\Gamma$ and $K/\CO$ in the non-$p$-adically closed setting and the sort $\bk$ in the power bounded $T$-convex setting. The remaining cases are divided into the unstable sorts of $K/\CO$ and $\Gamma$ in the $p$-adically close setting, and the algebraically closed (hence stable) sort $\bk$ in the $V$-minimal setting.

In order to uniformly treat all the unstable distinguished sorts (namely, all sorts except $\bk$ in the $V$-minimal setting), we introduce in 
 Section \ref{S: asi to D} the notion of \emph{a vicinity}, generalising the notion of a topological neighbourhood, and the associated \emph{vicinic sort}, generalising here the  notion of an SW-uniform sort. \\

Without getting into technical details, let us give a brief overview of the strategy of the proof and through it also some of the intuition underlying these new definitions. 
As in \cite{HaHaPeVF}, our method is local, in the sense that the subgroup $\nu$ in the conclusion of Theorem \ref{intro-2} appears as a generalised infinitesimal group (as in, e.g., \cite{PeStmu} or \cite{JohnDpJML}).

The first step in the proof is to find,  after possibly passing to a quotient by a finite normal subgroup,  a definable infinite $X\sub G$ with a definable  injection from $X$ into $D^n$, for $D$ one of the four distinguished sorts. Roughly, for a fixed such $D$, if $X$ has maximal dp-rank with respect to this property, then $X$ is called a $D$-set (see Definition \ref{D: D-sets} for a precise formulation).  

Most of the work in the paper is devoted to showing that each $D$-set gives rise to a type definable group $\nu_D$ as in Theorem \ref{intro-2} (as it turns out, this $\nu_D$ does not depend on the set $X$).

When $D$ is one of the SW-uniform sorts and $X$ is a $D$-set, the latter inherits a topology from $D$. Using ideas going back to Pillay, \cite{Pi5} (in a version due to  Ma\v{r}\'{\i}kov\'{a}, \cite{MarikovaGps}), we show  that the topology can be extended to a group topology on $G$. The subgroup $\nu$ in the conclusion of Theorem \ref{intro-2} is precisely the infinitesimal subgroup of $G$ with respect to this topology, namely the intersection of all definable open neighbourhoods of $e$. 

More generally, in all the unstable distinguished sorts, the notion of a vicinity is a good generalization of that of a neighbourhood, and it allows us to define the subgroup $\nu$ of Theorem \ref{intro-2} similarly to the above definition of the infinitesimal group. 

The remaining case, of the strongly minimal $\bk$ in the V-minimal setting, is different and is handled using methods of stable groups. \\

As can be seen in the statement of Theorem \ref{intro-2}, when $D$ is either the value group $\Gamma$ or the sort of closed balls $K/\CO$, we obtain a stronger result, asserting that $\nu$ can be identified with a type-definable subgroup of $(D^n,+)$ (here $n$ can be taken to be, in the above notation, $\dpr(X)$). This follows, essentially, from the fact that in ordered vector spaces and in $\Zz$-groups definable functions are piece-wise affine with respect to the additive structure. For the sort $K/\CO$ this follows from an analogous result  proved  in  \cite{HaHaPeVF} when $K/\CO$ is an SW-uniformity, and in Section \ref{ss: functions in K/O} below in the $p$-adic case.

On the technical level, there are several themes around which the paper is focused: 
\begin{itemize}
\item Properties of definable subsets and functions in $K/\CO$, in the setting of $p$-adically closed fields  (see Section \ref{S:P-minimal}).
    \item The definition of vicinic sorts and their basic properties (see Section \ref{S: asi to D}), and  
    the verification that all the unstable distinguished sorts, in all our settings, are vicinic. In SW-uniformities and in $\Zz$-groups these properties are known (or easy), and most of the work is dedicated to settling them for the sort of closed balls,     $K/\CO$, in the $p$-adically closed setting.
    % \textcolor{red}{K: I find the explanation below too technical for this stage as they dont know yet what vicinic mean. Shall we drop it?}
    % (for SW-uniformities the main work was carried out in \cite{HaHaPeVF}, for PrA this is easy, because $\acl(\cdot)$ has the Exchange Property, so the focus is on $K/\CO$ in the $p$-adically closed case),
    \item  To carry out the above strategy, we have to work with definable injections from (definable subsets of) $G$ into the relevant distinguished sorts.  In general,  we can only assure the existence of finite-to-one such functions. We show that this can be resolved by replacing $G$ with a quotient by a finite normal subgroup (see Section \ref{ss:from almost to strong}). 
    % \textcolor{red}{K: Not sure about the next sentence, especially the reference to Hrushovski here is vague} This can be resolved, using an old argument of Hrushovski's, replacing $G$ with a quotient by a finite normal subgroup, but the adaptation of the argument requires some technical preparation. 
\end{itemize}   

%\textcolor{red}{K: Maybe add this, and if we do then let's omit the sentence about unbounbded exponent which appears earlier}

To deduce the main part of Theorem \ref{intro-1} we first show that the type definable group $\nu_D$ obtained in Theorem \ref{intro-2}  has unbounded exponent, namely for every $\widehat \CK\succ \CK$, the group $\nu_D(\widehat \CK)$ has unbounded exponent. This easily implies that so if $G$.

By a result of Simon, \cite[Proposition 3.1]{SimDPMinOrd}, if $G$ is  dp-minimal then it has a definable abelian normal subgroup $H$,  such that $G/H$ has bounded exponent. Since $G/H$ is itself an interpretable group, the above implies that it must be finite. \\

%Finally, to deduce Theorem \ref{intro-1} from Theorem \ref{intro-2}, we prove that whenever $\nu_D$ is infinite then it has unbounded exponent (hence, so does $G$). In the dp-minimal case we apply a theorem of Simon (ref?) to conclude that $G$ must be abelian-by-finite}

\medskip

We conclude the introduction with a word on the setting: Case (iii) of  Theorem \ref{intro-2} is stated for $p$-adically closed fields (by which we mean models elementarily equivalent, in the MacIntyre language,  to a finite field extension of $\qp$). However, the proofs work without any changes in an expansion of $K$, a finite field extension of $\mathbb Q_p$, by analytic functions on $\CO^n$ (in the sense of \cite{DenvdDr}).
 In fact, our proofs work in $P$-minimal expansions of $p$-adically closed fields that are $1$-h-minimal and have  definable Skolem functions. However, for the sake of clarity of exposition, we state the results in the $p$-adically closed setting. As some of the results may be of independent interest and for ease of future reference, we will specify throughout the level of generality at which the same proof works. \\
 
\subsection{Previous work on groups in valued fields}
     % \textcolor{red}{K: I did not read the previous work. I imagine that by now there are many more, but perhaps they are not ``previous''?}
We list briefly additional work on groups in valued fields, beyond what has been mentioned above.
 Definable groups in $\qp$ were first studied by Pillay, \cite{PilQp}. Hrushovski and Pillay, \cite{HruPil}, discuss connections between definable groups in local fields and algebraic groups.
 In \cite{AcostaQp} Acosta gives an exhaustive list of all 1-dimensional groups definable in $\qp$. In \cite{AcostaACVF}  he extends this result to 1-dimensional \emph{commutative} groups definable in algebraically closed valued fields\footnote{By Theorem \ref{intro-1} the commutativity assumption is redundant in algebraically closed valued fields of equi-characteristic $0$.}
 Johnson and Yao, \cite{JohnYao}, study definable, definably compact, groups in $\qp$.
 Montenegro, Onshuus and Simon, \cite[Theorem 2.19]{MonOnSim}, work under general assumptions,  applicable in several examples of our setting (e.g., ACVF$_{0,0}$, RCVF, $p$-adically closed fields).
 A classification of definably simple groups definable in henselian fields of characteristic $0$ (or, more generally, in 1-h-minimal fields of characteristic 0) is obtained in an unpublished work of Gismatullin, Halupczok and Macpherson.
 
 While the above results  mostly study definable groups, the work of Hrushovski and Rideau-Kikuchi, \cite{HrRid}, covers interpretable, stably dominated groups in algebraically closed valued fields,
 using the results on elimination of imaginaries in such fields. In a recent pre-print, Johnson \cite{JohnTopQp} studies topology on groups interpretable in $p$-adically closed fields.
 
 Finally, although we do not make use of these here, we note the seminal papers of  \cite{HaHrMac1}, \cite{MelRCVFEOI}, and \cite{HrMarRid}, where elimination of imaginaries with appropriate sorts is proved for  algebraically closed valued fields, real closed valued fields and $p$-adically closed fields, respectively.

\subsection{Structure of the paper}
        In Section \ref{S: background} we review the notion of SW-uniformities and the distinguished sorts introduced in \cite{HaHaPeVF} and  collect some of their useful properties.  In Section \ref{S: local K/O}, we study  the definable sets and functions in $K/\CO$ in $p$-adically closed fields, isolating (Section \ref{ss:key geo properties})  key geometric properties which they share with SW-uniformities.
        
       Based on these properties, we introduce in Section \ref{ss: vicinic}  the framework of \emph{vicinic structures} where most of our work takes place. It is a generalisation of SW-uniformities, encompassing also the distinguished sorts of $p$-adically closed fields.

        The rest of Section \ref{S: asi to D} is dedicated to collecting the tools needed for the construction in Section \ref{S:infint groups} of infinitesimal subgroups of interpretable groups. Much of the work in Section \ref{S: asi to D} is devoted to dealing with a technical issue not arising in the study of fields: in \cite{HaHaPeVF} we have shown that every interpretable field can be locally injected into one of the distinguished sorts. In the present setting,  we are required to work with finite-to-one functions, and this requires some additional work.
        Also, in the V-minimal case, the sort $D=\bk$ is a pure algebraically field, and thus has a different nature. It is treated in Section \ref{S: lsi to k}.

        In Section \ref{S: final} we collect all the results of previous sections to prove Theorem \ref{intro-2} and Theorem \ref{intro-1}. Section \ref{S: examples} is dedicated to the study of  several natural examples of interpretable groups in light of the results of this paper. In Section \ref{S: foreig}  
        we apply the techniques developed thus far to show that, in our setting, there are no definable finite-to-finite correspondences between two distinct distinguished sorts. In the Appendix, we study the special case of a dp-minimal group in a power bounded T-convex setting.\\

\noindent{\bf Acknowledgements} We thank Pablo Cubides Kovacsics, Amnon Besser and Dugald Macpherson for several discussions during the preparation of the article. We thank Yair Glasner for pointing out a mistake in an earlier draft, and the anonymous referee for the thorough reading and useful suggestions.

\section{Background and preliminaries}\label{S: background}
\subsection{Notation,  terminology and some preliminaries}
Throughout, structures will be denoted by calligraphic letters, $\CM$, $\CN$, $\CK$ etc., and their respective universes by the corresponding Latin letters, $M$, $N$ and $K$. We reserve $\CK$ to denote expansions (possibly trivial) of valued fields, and $K$ will always be a valued field.  All structures may be multi-sorted. A valued field $\CK=(K,v,\ldots)$  is always  considered with a single home sort (for the ground field $K$) with all other sorts coming from $\CK^{eq}$. All sets are definable using parameters unless specifically mentioned otherwise.

Tuples from a structure $\CM$ are denoted by small Roman characters $a,b,c,\dots$ and are always assumed to be finite. We apply the standard model theoretic abuse of notation writing $a\in M$ for $a\in M^{|a|}$. Variables will be denoted $x,y,z,\dots$ with the same conventions as above. We do not distinguish notationally tuples and variables belonging to different sort, unless some ambiguity can arise. Capital Roman letters $A,B,C,\ldots$ usually denote small subsets of parameters from $M$. As is standard in model theory, we write $Ab$ as a shorthand for $A\cup \{b\}$. In the context of definable groups we will, whenever confusion can arise, distinguish between, e.g., $Agh:=A\cup\{g,h\}$ and $Ag\cdot h:=A\cup \{g\cdot h\}$. 

By a partial type we mean a consistent collection of formulas. Two partial types $\rho_1, \rho_2$ are equal, denoted $\rho_1=\rho_2$, if they are logically equivalent, i.e., if they have the same realizations in some sufficiently saturated elementary extension.

We use standard model theoretic terminology and notation, providing references as we proceed.  The introductory sections of \cite{HrKa} provide a quick, more or less self-contained, guide to much of the model theory of valued fields used in the present paper, as well as many relevant examples. Below we give some basic terminology and notation.\\

\noindent\textbf{Valued fields.} When $(K,v,\ldots)$ is an expansion of a valued field, we let $\CO_K$ (or just $\CO$, if the context is clear) denote its valuation ring. Its maximal ideal is $\m_K$ (or $\m$) and $\bk:=\CO/\m$ its residue field. The value group is  $\Gamma_K$ (or just $\Gamma$). As in \cite{HaHaPeVF}, $K$, $\bk$, $\Gamma$ and $K/\CO$ are the \emph{distinguished sorts}. We shall occasionally (especially in Section \ref{S: examples}) also use the sorts $\mathrm{RV_\gamma}:=K^\times /(1+\m_{\gamma})\cup \{0\}$ where, for a non-negative $\gamma\in \Gamma$ we denote $\m_\gamma:=\{x\in K: v(x)>\gamma\}$. We also let $\mathrm{rv}_\gamma: K\to \mathrm{RV}_\gamma$ denote the quotient map extended by $0$ at $0$. 

A closed ball in $K$ is a set of the form $B_{\geq \gamma}(a):=\{x\in K: v(x-a)\ge \gamma\}$ and $\gamma$ is its (valuative) radius, denoted $r(B)$. Open balls are defined similarly and denoted by $B_{>\gamma}(a)$. Throughout by "ball" we mean either an open or a  closed ball. An open (closed) ball in $K^n$ is a product of $n$ open (closed) balls of \textbf{equal radius}, i.e., $B_1\times \dots \times B_n$ where $B_i:=B_{>r}(a_i)$ (resp. $B_i:=B_{\geq r}(a_i)$) for some $r\in \Gamma$ and $a_i\in K$. A ball in $K^n$ is either a closed or an open ball, it has radius $r$ if it is a product of $n$ closed (or $n$ open) balls in $K$ each of radius $r$.

{\bf Throughout Section \ref{S: background} $\CK$ is either a $V$-minimal field,  a power bounded $T$-convex expansion of a real closed field or a $p$-adically closed field} (see \cite[Section 2.3]{HaHaPeVF}  and references therein for the definitions and basic properties of such fields). All these fields are 1-h-minimal, and though this notion is used, throughout, as a black box, for the sake of completeness we recall the definition (in a formulation suitable for mixed characteristic as well). We first recall the following definition from \cite{hensel-min}.

\begin{definition}\cite[Definition 2.1.6]{hensel-min}\label{D: m-next to}
	For an integer $m\in\mathbb{N}$, a ball $B\sub K$  is {\em $m$-next to $c\in K$} if $B=\{x\in K: \rv_m(x-c)=t\}$ for some non-zero $t\in \RV_m$. 
	
	A ball $B$ is {\em $m$-next to} a finite non-empty set $C\subseteq K$ if $B=\bigcap_{c\in C} B_c$ where  $B_c$ is a ball $m$-next to $c$ for all $c\in C$.
\end{definition}

\begin{remark}\label{rem:next}
    Note that if a ball B is $m$-next to some $c\in K$ then it is an open ball. Furthermore, an open ball $B$ is $m$-next to $c$ if and only if $c\notin B$ and the radius of $B$ is equal to $v(x-c)+m$ for every (equivalently, some) $x\in B$. Thus, an open ball $B$ is $m$-next to a finite set $C$ if and only if $B\cap C=\emptyset$ and $r(B)=\max\{v(x-c)+m:c\in C\}$, for every (equivalently, some) $x\in B$.
\end{remark}

\begin{definition}\label{D: 1-h-min.}
    A theory $T$ of valued fields of characteristic $0$ is \emph{1-h-minimal} if for any set of parameters $A$,
    %$A\sub K\cup \mathrm{RV}_n$ \textcolor{red}{K: which $A$ are excluded? for some positive integer $n$ 
    and any $A$-definable $f: K\to K$ the following hold: 
    \begin{enumerate}
        \item There exists a finite $A$-definable $C$ and a positive integer $m$ such that for any ball $B$, $m$-next to $C$, there exists $\gamma_B\in \Gamma$ such that or all $x,y\in B$ we have 
        \[
            \gamma_B-v(m)\le v(f(x)-f(y))-v(x-y)\le \gamma_B +v(m).
        \]
        \item The set $\{y\in K: |f^{-1}(y)|=\infty\}$ is finite.
    \end{enumerate}
\end{definition}

\begin{remark}
\begin{enumerate}
\item The original formulation from \cite{hensel-minII} has some restrictions on the parameter set $A$, that is $A\subseteq K\cup \mathrm{RV}_n$ for some $n$. In our setting, since all the sorts are from $K^{eq}$ the definitions coincide. 

\item The above geometric definition is equivalent to the model theoretic definition, by \cite[Theorem 2.2.7]{hensel-minII}. 

\end{enumerate}
\end{remark}

It is known that in all the settings we are working in, the sorts $\bk$ and $\Gamma$ are stably embedded (for $p$-adically closed fields, or more generally $P$-minimal fields this is \cite[Theorem 6, Lemma 2]{ClucPresburger}, for $V$-minimal field, or more generally $C$-minimal fields this is \cite[Lemma 3.30]{HrKa} and for the power bounded $T$-convex setting it is \cite[Theorem A, Theorem B]{vdDries-Tconvex}). It is, therefore, clear what  the induced structure is on any of these sorts: $\Gamma$ is a pure ordered vector space when dense, and a $\Zz$-group in the $p$-adic setting, $\bk$ is a model of $T$ in the $T$-convex setting (for our needs it suffices that it is an o-minimal expansion of a real closed field) and a pure algebraically closed field of characteristic $0$ in the $V$-minimal case.  We note, however,  that the sort $K/\CO$ is not stably embedded in any of the structures we are studying, 

We assume the reader is familiar with the notion of dp-rank, see \cite[\S 2.1]{HaHaPeVF} for the definition and a quick survey of dp-rank and its main properties. Let us point out that in the field $K$, the value group $\Gamma$ and the residue field $\bk$, dp-rank coincides with the classical notions of dimension associated with each of these structures. It may also help the reader to note (and this will be proved in some cases below) that in every distinguished sort $D$ in our setting  the dp-rank of $a\in D^n$ equals its $\acl$-dimension, even if $\acl$ does not satisfy exchange.
For most applications below it suffices to know that dp-rank is finite and sub-additive; 
\[
    \dpr(ab/A)\le \dpr(a/A)+\dpr(b/Aa).
\]

Finally, when saying that a valued field has \emph{definable Skolem functions}, we mean that the valued field sort $K$  has definable Skolem functions. Power bounded $T$-convex valued fields have definable Skolem function by \cite[Remark 2.4]{vdDries-Tconvex} (after adding a constant). So do $p$-adically closed fields  \cite[Theorem 3.2]{vdDriesSkolem} and their expansion by all sub-analytic sets \cite[Theorem 3.6]{DenvdDr}.

\subsection{Almost strong internality}
The starting point of the present work is the following result that can be deduced from \cite[Proposition 5.5]{HaHaPeVF} (see also \S 7,  \textit{loc. cit.} for the details). 
By {\em a finite-to-finite correspondence between $X_1$ and $X_2$} we mean a  definable relation $C\sub X_1\times X_2$ such that both projections $\pi_i: C\to X_i$ ($i=1,2$) are surjective with finite fibres. 
\begin{fact}\label{F: 5.5}
    %Let $\CK$ be an either $P$-minimal, power bounded $T$-convex or $C$-minimal expansion of a valued field. 
    If $X$ is infinite and interpretable in $\CK$ then there exist  a distinguished sort $D$, infinite definable $T\sub X, S\sub D$ (possibly defined over additional parameters) and a definable finite-to-finite correspondence $C\sub T\times S$. 
\end{fact}

When -- in the above notation -- $X$ was a field, we were able to eliminate the discrete sorts $K/\CO$ and $\Gamma$ in the $p$-adically closed case (in fact, there this was done for any P-minimal field), and the remaining distinguished  sorts turned out to be rather tame topological structures (that we called SW-uniformities -- to be described in more detail below). In that setting (and for $X$ a field) we were able to strengthen Fact \ref{F: 5.5} to obtain a definable bijection between $T$ and a  subset of one of the distinguished sorts.  In the present setting  some extra work is needed to  obtain a weaker result, i.e. a finite-to-one function from an infinite definable subset of our group $G$ into one of the distinguished sorts. Extending our terminology from \cite{HaHaPeVF} we define: 

\begin{definition}\label{Def: internal}
    A definable set $X$ is \emph{$A$-almost strongly internal to a  definable set $D$} if there exists an $A$-definable finite-to-one function $f: X\to D^k$, for some $k\in \Nn$. The set $X$ is \emph{locally almost strongly internal to $D$} if there exists a parameter set $A$ and an infinite $A$-definable $X'\sub X$ that is $A$-almost strongly internal to $D$.  
\end{definition}

Recall that $X$ is \emph{strongly internal} to $D$ if we can find an injective $f: X\to D^k$ and \emph{locally strongly internal} to $D$ if some definable infinite $X'\sub X$ is strongly internal to $D$. As we shall see in Example  \ref{L: not lsi}, a set which is locally almost strongly internal to $D$ need not be locally strongly internal to $D$. The upshot of this situation is that we have to develop the theory, in parallel, for subsets of $G$ almost strongly internal to a distinguished sort $D$, and for subsets of $G$ strongly internal (possibly to the same sort $D$). Since the statements and the proofs are, as a rule, similar in both cases we usually state both results simultaneously; e.g. 
\begin{quote}
``If $X$ is locally (almost)  strongly internal to $D$ ... then $Y$ is (almost) strongly internal...''
\end{quote}
with the convention that either all brackets are included or all brackets are omitted. I.e., the two statements included in the above formulation are ``If $X$ is locally  strongly internal to $D$ ... then $Y$ is  strongly internal...'' and ``If $X$ is locally almost  strongly internal to $D$ ... then $Y$ is almost strongly internal...''.

\subsection{Simon-Walsberg uniformities and group topologies}\label{ss: SW-uniformities}
Though the setting of the present paper requires that we take into account the discrete sorts $\Gamma$ and $K/\CO$, in $p$-adically closed fields,  SW-uniformities still have a significant role to play, and even in these two cases some  arguments are modelled after the analogous arguments in the topological setting. So we remind:
\begin{definition}
    A definable set $D$ in a multi-sorted structure $\CM$ is an SW-uniformity if: 
    \begin{enumerate}
        \item $D$ is dp-minimal. 
        \item $D$ has a definable uniform structure (or, uniformity) giving rise to a Hausdorff topology (see \cite[IX.11]{james1966topology}).  I.e., there exists a formula $\varphi(x,y, z)$ and a definable set $S$ in $\CM$ such that $\{\phi(D^2, s): s\in S\}$ is a uniform structure on $D$, and the intersection of all $\phi(D^2,s)$, for $s\in S$, is the diagonal.
        \item Every definable subset of $D$ has non-empty interior with respect to the uniform topology. 
        \item $D$ has no isolated points in the uniform topology. 
    \end{enumerate}
\end{definition} 
As we already noted, the sorts $K$ in all settings,  $\Gamma$ and  $K/\CO$ in the power bounded T-convex and V-minimal settings and $\bk$ in the power bounded $T$-convex setting, are all SW uniformities (the fact that  $K/\CO$ is not stably embedded is not a problem since the definition allows for that).

One of the most important tools in \cite{HaHaPeVF} was the  technical fact below. We shall prove here an a analogue of this result in other settings as well, and eventually include a variant of it as one of the axioms of our \emph{ad hoc} setting of vicinic sorts.

\begin{fact}[{\cite[Corollary 3.13]{HaHaPeVF}}]   \label{Gen-Os in SW}
  Let $D$ be an SW-uniformity and  $X \sub  D^n$, $Y \sub  X$ be definable sets. Let $b \in  D^k$ be any element,  $a\in Y$ in the relative interior of $Y$ in $X$  and $A$ a small set of parameters. Then there exist $B\supseteq A$  and a $B$-definable open subset $U = U_1 \times \dots  \times  U_n \sub  D^n$ such that $a \in  U \cap  X \sub  Y$ and $\dpr(a, b/B) = \dpr(a, b/A)$.
\end{fact}

We also require the following result (implicit in \cite{HaHaPeVF}) that plays a key role in  our definition of $D$-groups, the main object of study of Section \ref{S: asi to D} and Section \ref{S:infint groups} :

\begin{lemma}\label{L: asi from surjective} 
	Let $\CM$ be some structure and let $D$ be an SW-uniformity definable in $\CM$, $W, Z$ definable sets, and $f:Z\to W$ a definable finite-to-one surjection.
	
	\begin{enumerate}
		\item If $Z$ is strongly internal to $D$ then there exists 
		a definable $W_1\sub W$, with $\dpr(W_1)=\dpr(W)$ such that $W_1$ is strongly internal to $D$.
		\item If $Z$ is almost strongly internal to $D$ then 
		there exists a definable $W_1\sub W$, $\dpr(W_1)=\dpr(W)$, such that $W_1$ is almost strongly internal to $D$.
	\end{enumerate}
\end{lemma}
\begin{proof}
Assume everything is definable over some parameter set $A$. 

	(1) Since $Z$ is strongly internal to $D$ we may assume that $Z\subseteq D^k$ for some $k$. We fix $z_0\in Z$ such that $\dpr(z_0/A)=\dpr(Z)$. We first find a definable open set $V\sub D^k$ such that $|[z_0]_f\cap V|=1$, where $[z_0]_f=f^{-1}(f(z_0))$.
	By Fact \ref{Gen-Os in SW} there exists an open set $U$, $z_0\in U\sub V$, definable over some $B\supseteq A$, such that $\dpr(z_0/A)=\dpr(Z)$, and $[z_0]_f\cap U=\{z_0\}$. But now, the set  $Z_1=\{z\in Z:[z]_f\cap U=\{z\}\}$ is defined over $B$ and contains $z_0$ hence $\dpr(Z_1)=\dpr(Z)=\dpr(W)$. Now $f\restriction Z_1$ is injective, and its image $W_1$ is our desired set.
	
	(2) Since  $Z$ is almost strongly internal to $D$ there exists $\sigma:Z\to D^k$ with finite fibres. Let 
	\[C=\{(w,y)\in W\times D^k: y\in \sigma(f^{-1}(w))\};\]
	it is a finite-to-finite correspondence. Choose $y_0\in \pi_2(C)$ (the projection of $C$ into $D^k$) with $\dpr(y_0/A)=\dpr(\pi_2(C))$. Since $C$ is finite-to-finite the set of all $y\in D^k$ such that $C^y\cap C^{y_0}\neq \0$ is finite, so there exists an open $V\ni y_0$ such that for all $y\in V$, $C^y\cap C^{y_0}=\0$.
	
	As above, we may replace $V$ with an open $U$ defined over $B$ such that $\dpr(y_0/B)=\dpr(y_0/A)$. Then $C\cap (W\times U)$ is a graph of a finite-to-one function, $F$. Since $y_0$ is in the range of that function and $\dpr(y_0/A)=\dpr(W)$, necessarily $\dpr(\dom(F))\ge \dpr(W)$, so equality must hold, as required. 
\end{proof}

Much of Section \ref{S: local K/O} is dedicated to making sense of the above results for $K/\CO$ in  $p$-adically closed fields and proving an appropriate analogue. Similar results for $\Zz$-groups are easier and follow from known properties of such groups. \\
% The case where $G$ is locally strongly internal to  $\bk$, when $\kb$ is  strongly minimal, is treated separately, in Section \ref{S: lasi to sm}. 

When $G$ is locally  strongly internal to an SW-uniformity,  the group $\nu$ we obtain in Theorem \ref{intro-2} admits a group topology. This   will allow us to topologise $G$, using the following two easy observations. %It is easy and probably well known, so we omit the proof: 

\begin{lemma}\label{L:from H a uniformity on G}
Let $G$ be a group and $H$ a subgroup endowed with a Hausdorff group topology, with $\mathcal B_H$ a basis for the $e$ neighbourhoods in $H$. For $V\in \mathcal B_H$, let $U_V = \{(x,y) \in  G \times G : x^{-1}y \in  V \}$.

Then the collection $\mathcal U^G=\{U_V:V\in \mathcal B_H\}$ is a left-invariant uniformity on $G$ extending the associated uniformity on $H$.
\end{lemma}

The uniformity $\mathcal U^G$ induces a topology on $G$, call it $\tau^G$,  whose basis is the collection of all sets $U_V(h,G)$, as $h$ varies in $G$ and $V$ varies in $\mathcal B_H$. Though this need not in general be a group topology,  for every $g\in G$ the map $x\mapsto gx$ is continuous. We can show, however: 

\begin{lemma} \label{L: top from H to G}
Let $G$ be a group and  $H\le G$  a topological group such that,
 \begin{enumerate}
    \item[(i)] For every $g\in G$ there is an open $V\sub H$ such that $V\sub H^g\cap H $.
    \item[(ii)] For every $g\in G$, the function $x\mapsto x^g=gxg^{-1}$, from $H$ into $H^g$, is continuous at $e$ with respect to the topology on $H$ (this makes sense due to (i)).
\end{enumerate}
Then the uniformity $\mathcal U^G$ defined above makes $G$  a topological group.
\end{lemma}
\begin{proof}
Let $\CB_H$ denote a neighbourhood base for $H$ at $e$. We may assume that each $W\in \CB_H$ is symmetric.  Note also that by Lemma \ref{L:from H a uniformity on G} for each $g\in G$, the family $\{g V:V\in \CB_H\}$ is a basis for the $\tau^G$-neighbourhoods at $g$.
 
 We prove first that group inverse is continuous.
Assume that $g^{-1}V$ is an open neighbourhood of $g^{-1}$. We need to find an open neighbourhood $gW$ of $g$ such that $W^{-1}g^{-1}\subseteq g^{-1} V$. By our assumptions, the map $x\mapsto gxg^{-1}$ is continuous at $e$ and hence there is $W\in \CB_H$ (in particular $W^{-1}=W$), such that $gWg^{-1}\subseteq V$. It follows that $Wg^{-1}\subseteq g^{-1}V$.

To prove continuity of  multiplication let  $g_1,g_2\in G$, and assume that $g=g_1g_2$, and $gV$ is a basic neighbourhood of $g$ and fix some $V'\in \CB_H$ such that $V'\subseteq H^g \cap H$, as provided by (i). Using the continuity at $e$ of $x\mapsto x^g$, we can shrink $V'$ so that $V'\subseteq V^g$. Using the continuity of $x\mapsto x^{g_1}$ and the fact that $H$ is a topological group, we can find $W_1, W_2\in \CB_H$, such that
$$ (g_1W_1g_1^{-1}) (gW_2g^{-1})\subseteq V'.$$

It follows that 
$g_1W_1 g_2W_2g^{-1}\subseteq gVg^{-1}$, and hence $g_1W_1g_2W_2\subseteq gV$, as needed. Thus $\mathcal U^G$ gives rise to a group topology.
\end{proof}

\section{Some results on $p$-adically closed fields}\label{S:P-minimal}
An important  tool in \cite{HaHaPeVF} is the analysis of interpretable fields via subsets that are strongly internal to SW-uniform sorts. Such an analysis is not available to us when studying interpretable groups  in a $p$-adically closed field $K$, as neither $\Gamma$ nor $K/\CO$ are SW-uniform sorts. In the present section we endow $K/\CO$ with a topology -- neither definable nor Hausdorff -- where, nonetheless, many of the tools available in SW-uniformities can still be applied. 

We prove two geometric properties that play a crucial role in the sequel. The first, comprising the heart of Section \ref{S: local K/O}, is an analogue of Fact \ref{Gen-Os in SW}, that can be interpreted as asserting the existence of generic neighbourhoods (of generic points) with respect to the above topology.  This is formulated precisely in Proposition \ref{P:sgenos for K/O}(3). Using this result we then proceed to proving an analogue of Lemma \ref{L: asi from surjective}.  Section \ref{ss: functions in K/O} is dedicated to the second main property that we need:  local linearity of definable functions in $K/\CO$, culminating in Corollary \ref{C: KO affine}. The strategy of proof is similar to that carried out in \cite[\S 6.2]{HaHaPeVF}, using results on 1-h-minimal fields, but working in mixed characteristic requires a little more care.  

While the results of Section \ref{ss: functions in K/O} may be of independent interest, the results of Section \ref{S: local K/O} are more technical. Readers unfamiliar with the analogous results from \cite{HaHaPeVF} may wish to first read Section \ref{ss:key geo properties} for a more detailed discussion of the underlying geometric motivation.   \\

% We isolate three of the geometric properties that suffice to put the machinery in gear, and verify that they hold in models of Presburger Arithmetic. This will allow, in later sections, a uniform treatment of a wide variety of examples. 

\textbf{From this point  up to Section \ref{S: asi to D}, $\CK:=(K,+,\cdot,v)$ is a $p$-adically closed field}. We note, however, that the results of Section \ref{S: local K/O} hold in any $P$-minimal field. For Section \ref{ss: functions in K/O} we need, in addition, 1-h-minimality and definable Skolem functions (as discussed in the introduction). 

\newcommand{\bF}{\mathbb F}
We remind that $p$-adically closed fields are elementarily equivalent (in the Macintyre language) to a finite extension, denoted  $\bF$, of $\mathbb{Q}_p$. We work in a large saturated $p$-adically closed  field, $\CK=(K,v)$, and so we may assume that $(K,v)$ is an elementary extension
%\footnote{For expansions of $p$-adically closed fields, as above, the notation means  that  only the reducts to the pure language of valued fields  are elementarily equivalent.}  
of $(\bF,v)$. We add constants for all elements of $\bF$; since its value group $\Gamma_\bF$ is isomorphic to $\mathbb{Z}$ as an ordered abelian group for simplicity of notation we identify $\Gamma_\bF$ with $\mathbb{Z}$, and note that $\bk_K=\bk_\bF$. It is well known that, as an abelian group, $\mathbb{F}/\mathcal{O}_{\mathbb{F}}$ is isomorphic to $\bigoplus_{i=1}^n\mathbb{Z}(p^\infty)$, where $\mathbb{Z}(p^\infty)$ is the Pr\"ufer $p$-group and $n=[\mathbb{F}:\mathbb{Q}_p]$. Throughout, we will use without further reference the fact that $v$ is well-defined on $K/\CO\setminus \{0\}$ and denote $v(0)=\infty$ (for $0\in K/\CO$). As usual $v$ extends to $(K/O)^n$ by setting $v(a_1,\dots, a_n)=\min\{v(a_i): i=1,\dots, n\}.$ 
%The following is standard (and well known): 

\begin{fact}\label{F:F is dense in Z-radius balls}
	\begin{enumerate}
		\item For every  $a\in K$ with $v(a)\in \mathbb{Z}$, there is an element $r\in \bF$ with $v(a-r)>\mathbb{Z}$.
		
		\item For every $b\in K/\CO$, if $v(b)\in \mathbb{Z}\cup\{\infty\}$ then $b\in \bF/\CO_\bF\leq K/\CO$.
		\item For all $n\in \Nn$, there are only finitely many $b\in K/\CO$ such that $nb=0$; in fact every such $b$ is in $\mathbb{F}/\mathcal{O}_{\mathbb{F}}$.  
	\end{enumerate}
\end{fact}
\begin{proof}
	(1) This follows from the completeness of $\bF$ as a valued field. Indeed, for $t\in \bF$ with $v(at)=0$ we can find $c\in \bF$ with $v(at-c)>0$ thus $v(a-ct^{-1})>-v(t)$; continuing inductively we construct a pseudo-Cauchy sequence in $\bF$, converging, by completeness,
	to a limit in $\bF$. This limit has the desired properties. 
	
	(2) Every $b\in K/\CO$ with $v(b)\in \mathbb{Z}\cup\{\infty\}$ is of the form $b=B_{\geq 0}(a)$ for some $a\in K$ with $v(a)\in \mathbb Z\cup \{\infty\}$. By (1) there is $r\in \bF$ with $v(a-r)>\mathbb{Z}\geq 0$, so $r+\CO=a+\CO$ as required.

    (3) It suffices to show that if $nb=0$ then $b\in \bF/\CO_\bF$. If $nb=0$ and $b=b'+\CO$, for some $b'
    in K$,  then $v(nb')=v(n)+v(b')\ge 0$. So $v(b')\in \mathbb{Z}$, i.e., $b\in \bF/\CO_\bF$.   
\end{proof}

% We note that Clause (3) above is not needed for what we prove in the present section. It will, however, play a role later in the proof of the main results of the paper (Section \ref{S: final}). 

\subsection{Local analysis in $K/\CO$}\label{S: local K/O}
As noted above, the valuation on $K$ is well-defined on $K/\CO$, but the topology it induces on $K/\CO$ is discrete. For that reason we introduce a coarser (non-Hausdorff) topology retaining some properties of SW-uniformities. The main technical property we need regarding this topology is the existence of generic neighbourhoods (Proposition \ref{P:sgenos for K/O}), but we also take the opportunity to single out several other properties that will be of use in the sequel. As an application we obtain Corollary \ref{L:Tor of ball} on the torsion of definable subgroups of $K/\CO$, that is used in the sequel, but may be  of independent interest. We first need some preliminaries:  
\begin{definition}
A ball $B\subseteq K^n$ is \emph{large} if $r(B)$, the valuative radius of $B$, satisfies $r(B)<\mathbb{Z}$.
\end{definition}

Let $\pi:K\to K/\CO$ be the quotient map, we also write $\pi:K^n\to (K/\CO)^n$ for any integer $n$.

\begin{definition}
	A \emph{ball} $U\subseteq (K/\CO)^n$ is the image under $\pi$ of a  large ball in $K$. 
\end{definition}

 We consider the {\em ball topology } on $(K/\CO)^n$ whose basis is the family of balls. Thus, for $n>1$, this is the product topology on $(K/\CO)^n$.

\begin{remark}
	The ball topology on $K/\CO$ is not Hausdorff, in fact, not even $T_0$:  for $b_1\neq b_2\in K/\CO$, if $v(b_1-b_2)\in \mathbb{Z}$ then every open set containing $b_1$ also contains $b_2$. In addition, the ball topology has no definable basis (since the set of balls of infinite radius in $K$ is not definable),
\end{remark}

Below, all topological notions in $K/\CO$ and $(K/\CO)^n$ refer to the ball topology.
Despite the above remark, some amount of local analysis in $K/\CO$ is possible using the lemma below. For $a=(a_1,\ldots, a_n)\in (K/\CO)^n$, we let $\dim_{\acl}(a)$ be the maximal $k$ such that some sub-tuple $a'$ of $a$ of length $k$  is algebraically independent and $a\in \acl(a')$ (this is well-defined even when $\acl$ does not satisfy exchange).

\begin{lemma}\label{L:generic over some gamma in K/O, K}
	Let $a=(a_1,\dots,a_m)$ be a tuple of elements in $K$ or $K/\CO$ and $A$ an arbitrary parameter set such that $\dim_{\acl}(a/A)=n$. Then for any $\gamma_0<\gamma_1<\mathbb{Z}$ with $\gamma_0-\gamma_1<\mathbb{Z}$ there exists $\gamma\in \Gamma$ such that  $\gamma_0<\gamma<\gamma_1<\mathbb{Z}$ and $\dim_{\acl}(a/A\gamma)=n$.
\end{lemma}
\begin{proof}
	By passing to a sub-tuple of $a$ which is  $\acl$-independent over $A$, we may assume that $|a|=n$.
	
	By replacing $\gamma_0$ with $\gamma_0-\gamma_1<\mathbb{Z}$  it is enough to find $\gamma\in \Gamma$, $\gamma_0<\gamma<\mathbb{Z}$ such that $\dim_{\acl}(a/A\gamma)=n$.
	
	For $1\le i \le n$, let $\hat a_i:= (a_1,\dots, a_{i-1}, a_{i+1}, \dots, a_n)$. Assume towards a contradiction that for every infinite $\gamma>\gamma_0$, there is $1\le i\le n$ such that $a_i\in \acl(\hat a_i, A, \gamma)$.  We may assume,  without loss of generality, that $a_1\in \acl(\hat a_1,A,\gamma)$ for arbitrarily large   $\gamma$ below $\mathbb Z$. We now consider all formulas $\varphi(x,y)$ over $\hat a_1,A$ with $y$ a $\Gamma$-variable and $x$ a $K/\CO$-variable.
	
	By compactness, there are formulas $\varphi_1(x,y),\dots, \varphi_k(x,y)$ over $\hat a_1,A$, and some $k\in \mathbb N$, such that  $|\varphi_i(K/\CO,\gamma)|\le k_i$ for every $i\leq k$, and for every $\gamma_1\in \Gamma$, if $\gamma_0<\gamma_1<\mathbb{Z}$ then there exists $\gamma_1<\gamma<\Zz$, such that $\bigvee_{i=1}^k \varphi_i(a_1,\gamma)$. 
	
	Without loss of generality we may assume that for $i=1$ the set of $\gamma>\gamma_0$  such that $\models \varphi_1(a_1,\gamma)$,  is cofinal below $\mathbb Z$. Since $\Gamma$ is a pure $\Zz$-group, we may further assume that $\varphi_1(a_1,\Gamma)$ is a definable set of the form $\{t\in \Gamma:\alpha<t<\beta\wedge x\equiv_N c\}$ for some $N\in\mathbb{N}$ and $0\leq c<N$. By the cofinality assumption, we cannot have $\beta < \mathbb{Z}$, so there exists $m\in \mathbb Z$, such that $\varphi_1(a_1,m)$, contradicting the assumption that $a_1\notin \acl(A,a_2,\dots,a_n)$.
\end{proof}

In terms of the ball topology on $K/\CO$, the above lemma says that if $a\in (K/\CO)^n$ is acl-generic then any neighbourhood  of $a$ contains a ball of ``generic radius''. In Proposition \ref{P:sgenos for K/O} we strengthen this to provide a ``generic neighbourhood'' of $a$ (inside any neighbourhood of $a$). But we first need some preliminary results. 

\begin{lemma}\label{L:full dpr is interior in K/O}
	Assume that $b\in (K/\CO)^n$ is such that $\dim_{\acl}(b/A)=n$, for some parameter set $A$.
	\begin{enumerate}
		\item For every $A$-definable $X\subseteq (K/\CO)^n$, if $b\in X$ then $b$ is in the interior of $X$.
		\item If $p=\tp(b/A)$ then $p(\CK)$ is open in $(K/\CO)^n$ 
		\item $\dpr(b/A)=n$.
	\end{enumerate}
\end{lemma}
\begin{proof}
	Let $a\in K^n$ be such that $\pi(a)=b$.
	
	(1) We proceed by induction on $n$. For $n=1$, we first prove that every $A$-definable set $Y\subseteq K$ containing $a$ must contain a large ball. Indeed, since $b\in \pi(Y)$ is non-algebraic  $Y$ intersects infinitely many $0$-balls. Since $K$ is P-minimal, by \cite[Proposition 5.8, Lemma 6.26]{HaHaPeVF}, $Y$ contains a large  ball.
	
	Let $q=\tp(a/A)$; by compactness we conclude that $q(\CK)$ contains a large ball, but then, as $\mathbb{Z}$ is automorphism invariant, every $\alpha\models q$ belongs to a large ball $B\subseteq q(\CK)$. In particular, $a\in B\subseteq Y$ for some large  ball $B$.
	
	Now assume that $b\in X$ for an $A$-definable set $X\sub K/\CO$ and apply the above to $a\in Y=\pi^{-1}(X)$. We conclude that there is a ball $U\subseteq K/\CO$ with $b\in U\subseteq X$.
	
	For $n>1$, write $b=(b',b_n)\in X$ for $b'\in (K/\CO)^{n-1}$. Let $a=(a',a_n)\in K^n$ be such that $\pi(a)=b$. Then $\dim_{\acl}(a)=n$ and $\dim_{\acl}(a'/a_nA)=n-1$. Let $Y=\pi^{-1}(X)$.
	
	Because $b_n\notin \acl(b'A)$, it follows that $X_{b'}=\{(b',t)\in X:t\in K/\CO\}$ is infinite, so by the case $n=1$, $b_n$ is in the interior of $X_{b'}$. It follows that $a_n$ is in the  interior of $Y_{a'}=\{(a',t)\in Y: t\in K\}$ in the large ball topology. Since $\dim_{\acl}(a'/a_nA)=n-1$, by Lemma \ref{L:generic over some gamma in K/O, K} we can find $\gamma<\mathbb{Z}$ such that $\dim_{\acl}(a'/a_nA\gamma)=n-1$ and $B_{\geq \gamma} (a_n)\subseteq Y_{a'}$.
	
	We consider the $a_nA\gamma$-definable set $W=\{x\in K^{n-1}: B_\gamma(a_n)\subseteq Y_{x}\}$. It contains $a'$, so by induction it contains a large ball $B_{\geq \gamma'}(a')$. Denoting $\widehat\gamma=\max\{\gamma,\gamma'\}$ the ball $B_{\geq \widehat \gamma}(a')\times B_{\geq \widehat \gamma}(a_n)$ is centred at $a$ and contained in $Y$. It follows that $b=\pi(a)$ is in the interior of $X$.
	
	(2) Follows by compactness from (1).
	
	(3) By (2), for $p=\tp(b/A)$, $p(\CK)$ contains a ball $U$. Such a ball is a Cartesian product of $n$ balls, each of which has dp-rank $1$. So $\dpr(p)\ge n$, and equality must hold. 
\end{proof}

We conclude:
\begin{corollary}\label{C: dpr=acl-dim in K/O}
	For $b\in (K/\CO)^n$ and any parameter set $A$ we have $\dpr(b/A)=\dim_{\acl}(b/A)$.
\end{corollary}
\begin{proof}
	By sub-additivity of the dp-rank we just need to show that $\dpr(b/A)\geq \dim_{\acl}(b/A)$. If $\dim_{\acl}(b/A)=k$ then by passing to a sub-tuple we may assume that $\dim_{\acl}(b_1,\dots,b_k/A)=k$. By Lemma \ref{L:full dpr is interior in K/O}(3) we have that $\dpr(b_1,\dots,b_k/A)=k$ but then $\dpr(b/A)\geq k$, as needed.
\end{proof}

We can now prove the desired generalisation of Lemma \ref{L:generic over some gamma in K/O, K}:
\begin{proposition}\label{P:sgenos for K/O}
	Let $b\in (K/\CO)^n$, $A$ any set of parameters and $c\in (K/\CO)^m$.
	\begin{enumerate}
		\item Assume that $\dpr(b/A)=n$ and that $b\in W$ for some ball $W\subseteq (K/\CO)^n$. Then there is a ball $V\subseteq W$, defined over some $B\supseteq A$ such that $b\in V$ and $\dpr(b/B)=\dpr(b/A)=n$.
		\item For any $\gamma<\mathbb{Z}$ there exists $B\supseteq A$ and a $B$-definable ball $V\subseteq B_{\geq \gamma}(b)$ such that $b\in B$ and $\dpr(b,c/A)=\dpr(b,c/B)$.
		
		\item Assume that $b\in X\subseteq (K/\CO)^n$ for some $B$-definable $X$ with $\dpr(b/B)=n$ then there exists $C\supseteq A$ and a $C$-definable ball $V\subseteq X$ such that $b\in V$ and $\dpr(b,c/C)=\dpr(b,c/A)$.
	\end{enumerate}
\end{proposition}
\begin{proof}
	In the proof, we repeatedly use -- without further mention -- Corollary \ref{C: dpr=acl-dim in K/O}, i.e.,  that $\dim_{\acl}=\dpr$.
	
	(1) By Lemma \ref{L:full dpr is interior in K/O}(2) $\tp(b/A)$ is open and thus contains a ball, $Z$, centred at $b$. Without loss of generality, $Z=W$. By Lemma  \ref{L:generic over some gamma in K/O, K}  we can find $\gamma$, $r(W)<\gamma<\mathbb{Z}$, such that $B_{\geq \gamma}(b)\sub W\vdash\tp(b/A)$ and $\dpr(b/A\gamma)=n$. We let $s=B_{\geq \gamma}(b)$, $r(s)=\gamma$ its radius and $[s]$ its code.

	\begin{claim}
	    For any parameter set $A$, any $b\in (K/\CO)^n$ and any $K/\CO$-ball $s$ containing $b$, if  $r(s)\in\dcl(A)$ and $\dpr(b/A)=n$ then $\dpr(b/A[s])=n$. 
	\end{claim}
	\begin{claimproof}
	    Fix $A,b, s$ as in the statement and let $\gamma=r(s)$. By Lemma \ref{L:full dpr is interior in K/O}(2), $q:=\tp(b/A)$ is open in $(K/\CO)^n$. It follows that $q$ contains a large ball $U$, that we may assume to be contained in $s$.  
		
	The proof is based on the case $n=1$ which we prove first:
	Since  $U$ is large, there are $b_m\in U\sub s$, $m\in \mathbb{N}$, such that $b_{m'}\neq b_m$ for $m'\neq m$. Obviously, $b_m\models q$ for all $m$. Thus, for every $b_m$, there is $\sigma_m\in \aut(\CK/A)$ mapping $b$ to $b_m$. Because $b_m\in s=B_{\geq \gamma}(b)$ and $\gamma\in \dcl(A)$ also $\sigma_m(B_{\geq \gamma}(b))=B_{\geq \gamma}(b_m)=B_{\geq \gamma}(b)$; so $\sigma_m(s)=s$. As a result, we get $b\notin \acl(A[s])$; i.e. $\dpr(b/A[s])=1$. 
	
	Assume now $b=(b_1,
	\ldots, b_n)$, and $s\sub (K/\CO)^n$ is a ball containing $b$ as above. We shall see that for every $i=1,\ldots, n$, $b_i$ is not in the algebraic closure of $A[s]b'$, where $b'=(b_1,\dots,\widehat{b_i},\dots,b_n)$. By Corollary \ref{C: dpr=acl-dim in K/O}, this implies $\dpr(b/A[s])=n$.
	For simplicity of notation, we treat the case $i=n$, so $b'=(b_1,\ldots, b_{n-1})$.
	
	We write $s'=B_{\geq \gamma}(b')$ and $s_n=B_{\geq \gamma}(b_n)$.
	By applying the case $n=1$ to $\tp(b_n/Ab')$, we get that $\dpr(b_n/Ab'[s_n])=1$.
	 Because $r(s')=r(s)\in \dcl(A)$ (by assumption) we get $[s']\in \dcl(Ab')$ and as $s=s'\times s$ also $[s]\in \dcl(Ab'[s_n])$. Therefore $\dpr(b_n/Ab'[s])=1$, as required.
	\end{claimproof}
	
	To conclude, recall that we have fixed some $\gamma<\mathbb{Z}$ such that $\dpr(b/A\gamma)=n$ and such that $B_{\geq \gamma}(b)\sub W$. Applying the claim to $b$ over $A\gamma$ we obtain the desired conclusion. 
	
	(2) To simplify notation, assume that $A=\0$. We start with the case where $m=0$, i.e. there is no $c$.   Let $b=(b_1,\dots,b_n)$ and $k=\dpr(b)$; without loss of generality, assume that $\dpr(b_1,\dots,b_k)=k$. Fix any $\gamma< \gamma'<\Zz$ such that $\gamma-\gamma'<\Zz$.  Choose $\gamma_1$ be as provided by Lemma \ref{L:generic over some gamma in K/O, K} applied to $\gamma'$, so $\gamma<\gamma_1$,  $\dpr(b/\gamma_1)=k$ and $\gamma-\gamma_1<\mathbb{Z}$. Let $U_1=B_{\geq \gamma_1}(b)$ and let $b'=(b_1,\dots,b_k,b_{k+1}',\dots,b_n')\in U_1$ be such that $\dpr(b'/\gamma_1)=n$.
	
	By (1), there exists $B\supseteq A$ and a $B$-definable ball $V\subseteq U_1$ with $V\ni b'$ and $\dpr(b'/B)=n$.  Since $\dpr(b'/B)=n$ and $\gamma-\gamma_1<\mathbb{Z}$ we get by Lemma \ref{L:generic over some gamma in K/O, K} that there exists $\gamma_2$,  $\gamma<\gamma_2<\gamma_1$ with $\dpr(b'/B\gamma_2)=n$. Consider the ball 
	\[V'=\{x\in (K/\CO)^n: v(x-y)\geq \gamma_2 \text{ for some $y\in V$}\}.\] 
	Observe that $b'\in V'$ as  $v(b-b')\geq \gamma_1>\gamma_2$. Our choice of $b', \gamma_2$ assures that $\dpr(b'/B\gamma_2)=n$ and that $\dpr(b/B\gamma_2)=k$. Since  $V'\sub B_{\geq \gamma}(b)$ it satisfies the requirements.
	
	Now, given $c\in (K/\CO)^m$ and a ball $U\ni b$, we apply the above result to the tuple $(b,c)$ and the open set $U\times (K/\CO)^m$ to obtain the desired conclusion.

	(3) By Lemma \ref{L:full dpr is interior in K/O}(1), $b$ is in the interior of $X$. Hence, there exists some $B_{\geq \gamma}(b)\subseteq X$ for $\gamma<\mathbb{Z}$. By (2), there exists $C\supseteq A$ and a $C$-definable ball $V\sub (K/\CO)^n$, $b\in V\subseteq B_{\geq \gamma}(b)$, with $\dpr(b,c/C)=\dpr(b,c/A)$.
\end{proof}

We conclude this section with a couple of applications of the tools developed thus far.  The first result is known for SW-uniformities (see \cite[\S 6.2]{HaHaPeVF}). The present proof is similar, but some extra care is needed, since the ball topology on $K/\CO$ is not Hausdorff.

\begin{lemma}\label{L:asi diagaram K/O}
	\begin{enumerate}
		\item Let $\{X_t:t\in T\}$ be a definable family of finite subsets of $(K/\CO)^n$, such that for all $b_1\neq b_2\in X_t$ we have $v(b_1-b_2)\in\mathbb{Z}$ and every $b\in (K/\CO)^n$ belongs to  finitely many $X_t$. Then there is a finite-to-one function from $T$ into$(K/\CO)^n$.
		
		\item Let $f:X\to T$ be a definable finite-to-one surjective map, and assume that $X$ is almost strongly internal to $K/\CO$. Then there exists a definable subset $T_1\subseteq T$ with $\dpr(T_1)=\dpr(T)$, such that 
 $T_1$ is almost strongly internal to $K/\CO$.
	\end{enumerate}
\end{lemma}
\begin{proof}  
	(1) By saturation, 
	\[
	Z:=\{v(b_1-b_2): b_1,b_2\in X_t, t\in T\}
	\]
	is finite. Let $m_0=\min\{Z\}\in \mathbb Z$; then for every $t\in T$ the set $\bigcup X_t$  is contained in a single ball of radius $m_0$.
	
	Let $U\subseteq K^n$ be a closed ball of valuational radius $m_0$ centred at $0$. Since $(K,v)$ is $p$-adically closed,  for each $a\in K^n$, $a+U$ contains only finitely many balls of radius $0$. Hence, the map  sending $t$ to the unique coset of $U$ containing $X_t$ is finite-to-one, so we have  constructed a finite-to-one definable map from $T$ into $K^n/U$. Since $(K/\CO)^n$ and $K^n/U$ are in definable bijection, we are done.
	
	(2) Let $g:X\to (K/\CO)^n$ be a finite-to-one definable map. For any $t\in T$ let $X_t=g(f^{-1}(t))$. Then $\{X_t:t\in T\}$ is a definable family of finite subsets of $(K/\CO)^n$ such that each $b\in (K/\CO)^n$ belongs to only finitely many $X_t$.
	
	By saturation, there is a uniform bound $m$ on $|X_t|$ for $t\in T$. We proceed by induction on $m$. If $m=1$, we are done; so assume otherwise. For simplicity, assume everything is defined over $\emptyset$.
	
	Let $t_0\in T$ be with $\dpr(t_0)=\dpr(T)$ and assume first that there are distinct $b, b'$ in $X_t$ with $v(b-b')<\mathbb{Z}$; so there exists a ball $U\ni b$ containing $b$ and not $b'$. By Proposition \ref{P:sgenos for K/O}(2), there exists a ball $V\subseteq U$ containing $b$ defined over parameters $B$ satisfying that $\dpr(b/B)=\dpr(b)$. As $b$ and $t_0$ are inter-algebraic over $\emptyset$, we also have $\dpr(t_0/B)=\dpr(t_0)=\dpr(T)$.   As $b'\notin V$, clearly $|X_{t_0}\cap V|<|X_{t_0}|$.

	The set $T_1=\{t\in T: |X_t\cap V|<|X_t|\}$ is defined over $B$ and contains $t_0$, thus $\dpr(T_1)=\dpr(T)$. We many now replace each $X_t$ with $X_t\cap V$ and proceed by induction. This completes the proof in the case where $v(b-b')<\mathbb{Z}$ for some  $b,b'\in X_{t_0}$. 
	
	So we assume now that $v(b-b')\in \mathbb Z$ for all distinct $b, b'$ in $X_{t_0}$. By saturation there exists $m\in \mathbb Z$ such that 
	$v(b-b')\geq m$ for all distinct $b, b'$ in $X_{t_0}$.
	Thus, the set \[T_1=\{t\in T: \mbox{for all distinct } b, b'\in X_t,\,\, v(b-b')\geq m\}\] contains $t_0$ so $\dpr(T_1)=\dpr(T)$, and we conclude by (1).
\end{proof}

We end this section with a small observation that will be used later on.

\begin{lemma}\label{L:Tor of ball}
  \begin{enumerate}
      \item For any definable subgroup $H\subseteq (K/\CO)^n$ of full dp-rank, $\mathrm{Tor}(H)=(\mathbb{F}/\mathcal{O}_{\mathbb{F}})^n$.
      
      \item Every definable subgroup $H\subseteq (K/\CO)^r$ has non-trivial torsion.
  \end{enumerate}
        
\end{lemma}
\begin{proof}
    (1) The fact that every torsion element of $H$ is included in $(\mathbb{F}/\mathcal{O}_{\mathbb{F}})^n$ is Fact \ref{F:F is dense in Z-radius balls}(3). For the other direction, assume that $H$ is $A$-definable and let $b\in H$ with $\dpr(b/A)=1$. By Lemma \ref{L:full dpr is interior in K/O}, there exists some $\gamma<\mathbb{Z}$ with $B_{>\gamma}(b)\subseteq H$; but then $B_{>\gamma}(0)=B_{>\gamma}(b)-b\subseteq H$ as well. Now obviously, $\mathbb{F}/\mathcal{O}_{\mathbb{F}}\subseteq B_{>\gamma}(0)\subseteq H$.
    
    (2) Let $H\subseteq (K/\CO)^r$ be a definable subgroup, which we may assume to be infinite, and let $n=\dpr(H)$. Since  by Lemma \ref{C: dpr=acl-dim in K/O} the dp-rank is given by the $\acl$-dimension, there exists a finite-to-one coordinate projection $\tau:H\to (K/\CO)^n$, with $\dpr(\tau(H))=n$. By (1) above, $\tau(H)$ has non-trivial torsion, and since $\ker (\tau)$ is a finite group $H$ has non-trivial torsion as well.
\end{proof}

\subsection{Definable functions in $K/\CO$}\label{ss: functions in K/O}
% In this subsection, \textbf{we assume that $\CK$ is  a sufficiently saturated 1-h-minimal P-minimal field,} e.g. a $p$-adically closed field.

We show that definable functions in $K/\CO$ are locally affine at generic points, with respect to the additive structure. With the tools developed in the previous subsection, the proof is quite similar to the one in \cite[\S 6.2]{HaHaPeVF}.  All results remain true in $\qp^{an}:=(\Qq_p,\CL_{an})$, the expansion of $\Qq_p$ by all convergent power series  $f:\CO^n\to K$, since those are $P$-minimal, 1-h-minimal with definable Skolem functions \cite{vdDriesHasMac} and \cite[Theorem 3.6]{DenvdDr}, respectively. 

\begin{definition}
	A partial type $P\sub K^n$ over a set of parameters $A$ is \emph{large} if $\dpr(\pi_*P)=n$, where $\pi: K\to K/\CO$ is the quotient map, and $\pi_*P$ the push forward of $P$.
\end{definition}

By Lemma \ref{L:full dpr is interior in K/O} (and Corollary \ref{C: dpr=acl-dim in K/O}), a partial type $P$ is large if and only if $(\pi_*P)(\CK)$ has non-empty interior, if and only if $P$ contains a large ball. Note that to a ball in $K$ we can associate two notions of largeness: as a partial type and as a ball. The above shows that the two notions coincide, i.e., that a ball $B$ is large as a partial type  if and only if $r(B)<\mathbb{Z}$. 

For the following, recall that by \cite[Example 6.2]{HaHaPeVF} $K/\CO$ is opaque,\footnote{ Opacity is needed only in order to invoke \cite[Lemma 6.4]{HaHaPeVF}, so we omit the definition. See \cite[\S 6.1]{HaHaPeVF} for details.} thus by \cite[Lemma 6.4]{HaHaPeVF} for every complete type $q$ concentrated om $(K/\CO)^n$ with $\dpr(q)=n$ there exists a unique complete type $p$ concentrated on $K^n$ with $\pi_*p=q$; moreover $\dpr(p)=n$ as well.

\begin{lemma}\label{L:basic prop large ball}
	Let $\CK=(K,v)$ be a $p$-adically closed field and $A$ an arbitrary set of parameters.
	\begin{enumerate}
		\item A partial type $P\vdash K^n$ over $A$ is large if and only if there is a  completion $p$ of $P$ (over $A$) that is large.
		\item If a partial type $P\vdash K^n$ is large, then $\dpr(P)=n$.
		\item If $\tp(a_1,\dots,a_n/A)$ is large, then so is $\tp(a_1/Aa_2,\dots,a_n)$.
		\item Let $B\subseteq K^n$ be a large open ball. Then $B+\CO^n= B$.
	\end{enumerate}
\end{lemma}
\begin{proof}
	(1) We only need to show that if $P$ is large, it has a large completion. Because $\dpr(\pi_*P)=n$ it has a completion $q$ of full dp-rank, and since $K/\CO$ is opaque, by \cite[Lemma 6.4]{HaHaPeVF}  there is a unique complete type $p$ such that $\pi_{*}(p)=q$. Since dp-rank can only decrease under definable maps, necessarily $\dpr(p)=n$. 
	
	The rest is as in \cite[Lemma 6.9]{HaHaPeVF}.
\end{proof}

We need some results from the theory of 1-h-minimal fields in mixed characteristic, as developed in \cite{hensel-minII}. Since, as in \cite{HaHaPeVF}, we apply 1-h-minimality as a black box, we will not dwell on it beyond the definition already given (Definitions \ref{D: 1-h-min.} and Definition \ref{D: m-next to}). It suffices for us that  $p$-adically closed fields, as well as the above-mentioned sub analytic expansions, are P-minimal and 1-h-minimal (\cite[Lemma 6.2.7]{hensel-min}).

The main result concerning 1-h-minimality (Definition \ref{D: 1-h-min.}), that we need is: 
\begin{fact}\label{F: VJP}\cite[Corollary 3.1.3]{hensel-minII}
	Let $T$ be a $1$-h-minimal theory, $K\models T$ and $f: K\to K$ an $A$-definable function. Then there exists an $A$-definable finite set $C$, and $m\in \Nn$ such that for any ball $B$ $m$-next to $C$, $f$ is differentiable on $B$ and $v(f')$ is constant on $B$. Moreover: 
	\begin{enumerate}
		\item For all $x,x'\in B$, 
		\[
		v(f(x)-f(x'))=v(f'(x))+v(x-x')
		\]
		\item If $f'\neq 0$ on $B$ then for any open ball $B'\sub B$ the image  $f(B')$ is an open ball of radius $v(f')+r(B')$ where $r(B')$ is the valuative radius of $B'$. 
	\end{enumerate}
\end{fact}

To apply this result, we need the following observation: 

\begin{lemma}\label{L: m-next}
	Let $p\in S_n(A)$ be a large type in $K^n$,  $C\subseteq K^n$ a finite $A$-definable set. Then for any fixed $m\in \Nn$ and any $b\models p$ there exists a large ball $B$, $b\in B\subseteq p(\CK)$, and a large ball $B'\supseteq B$ that is $m$-next to $C$. 
\end{lemma}
\begin{proof}
	Fix $b\models p$, $m\in \mathbb{N}$ and a finite $A$-definable set $C$; by Lemma \ref{L:full dpr is interior in K/O}, there exists an open large ball $B_{>r_0}(b)\subseteq p(\CK)$. Since $C$ is $A$-definable, $p(\CK)\cap C=\emptyset$ so also $B_{>r_0+m}(b)\cap C=\emptyset$.

 Now, if $r=\max\{v(b-c)+m:c\in C\}$ then $r<r_0+m$, and by Remark \ref{rem:next}, the ball $B_{>r}(b)$ is $m$-next-to $C$ and contains $B_{r_0+m}(b)\sub p(\CK)$. \end{proof}

Recall the following from \cite{HaHaPeVF}.

\begin{definition}
	A (partial) function $f:K^n\to K$ \emph{descends} to $K/\CO$ if  $\dom(f)+\CO^n=\dom(f)$ and for every $a,b\in \dom(f)$, if  $a-b\in \CO^n$, then $f(a)-f(b)\in \CO$. The function  $f$ {\em descends to $K/\CO$ on some (partial) type $P\vdash \dom(f)$} if $f\restriction P$ descends to $K/\CO$.
	
	Conversely, a  function $F:(K/\CO)^n\to K/\CO$ \emph{lifts} to $K$, if it is the image under the natural quotient map of a definable function $f: K^n\to K$ descending to $K/\CO$.   
\end{definition}

Recall that a structure $\CM$ has \emph{definable Skolem functions} (for the sort $S$) if for any definable formula $\phi(x,z)$  with $x$ ranging over $S$, there exists a definable function $f_\phi$ such that $\exists y \phi(x,z)\to \phi(f_\phi(z), z)$. Thus, if $\CK$ has definable Skolem functions (for the valued field sort $K$) then every definable function $f:(K/\CO)^n\to K/\CO$ lifts to $K$. This is the only application of the existence of definable Skolem functions in the present section. \\

The following basic example will play an important role: 

\begin{example}\label{example-descent}
	When $a\in \CO$  the linear function $x\mapsto a\cdot x$ descends to an endomorphism $\lambda_a:(K/\CO,+)\to (K/\CO,+)$. If $a\in\m$, then $ \lambda_a$  has an infinite kernel.
\end{example}

The next technical lemma is an analogue of \cite[Lemma 6.14]{HaHaPeVF}. The proof is similar with minor adaptations. For $f:K^n\to K$ we let $f_{x_i}$ denote the partial derivative with respect to $x_i$. 
\begin{lemma}\label{L:derivtives in O}
	Let $p\in S_n(A)$ be a large type, $p\vdash \dom(f)$ for some $A$-definable $f: K^n\to K$. Then: 
	\begin{enumerate}
		\item $f$ is differentiable on $p$. 
		\item If $f$ descends to $K/\CO$ on $p$ then $ (f_{x_1},\dots, f_{x_n})(a)\in \CO^n$ for all $a\models p$. 
		\item Assume that $\mathrm{Im}(f)\sub \CO$. Then for every $a\models p$ there exists a large ball $B\ni a$ contained in $p(\CK)$ such that for all $b\in B$  and $1\le i\le n$ 
		\[
		v(f_{x_i}(b))+2r(B))>0. 
		\]
		In particular $f_{x_i}(a)\in \mathfrak m$ for all $a\models p$. 
	\end{enumerate}
\end{lemma}
\begin{proof}
	(1) By \cite[Proposition 3.1.1]{hensel-minII}, every definable function is generically differentiable and $p$ is a generic type in $K^n$, so the result follows. 
	
	(2) For simplicity, assume that $A=\0$. Fix $a=(a_1,\dots, a_n)\models p$. We show that $f_{x_1}(a)\in \CO$. Set $g(t)=f(t,a_2, \dots, a_n)$. By Fact \ref{F: VJP} there is a finite $(a_2,\dots, a_n)$-definable set, $C$, and a natural number $m$ such that $g'$ is constant on any ball $B$ $m$-next to $C$. Since $p_1:=\tp(a_1/a_2,\dots, a_n)$ is large, it follows from Lemma \ref{L: m-next} that there is a large ball $B\sub p_1(\CK)$ containing $a_1$ where $g'$ is constant. 	By Lemma \ref{L:basic prop large ball}, $B+1=B$ and therefore by Fact \ref{F: VJP} we have: 
	\[
	v(g(a_1+1)-g(a_1))=v(g'(a_1))=v(f_{x_1}(a)). 
	\]
	Since $f$ descends to $K/\CO$ on $p$ we get that $v(f(a_1+1,a_2,\dots, a_n)-f(a))\geq 0$, so that $v(f_{x_1}(a))\geq 0$, as required. 
	
	(3) Let $B_0\sub p(\CK)$ be a large ball containing $a$, $r_0:=r(B_0)$. By $\aleph_1$-saturation of $\CK$ there exists $r_0<r<\mathbb{Z}$ such that $r_0-r<\mathbb{Z}$. For $B:=B_{>r}(a)$, let $b\in B$; we claim that $v(f_{x_i}(b))+r\ge 0$. 
	
	To simplify the notation, we show it for $i=1$. 
	Let $g(t):=f(t,b_2,\dots, b_n)$. Note that since $b\in B_{>r_0}(a)$, for all $r'>r_0$ we have that $B_{>r'}(b_1)\sub p_1(\CK)$ where $p_1:=\tp(b_1/b_2,\dots, b_n)$. In particular, if $C$ is the $(b_2, \dots, b_n)$-definable finite set provided by Fact \ref{F: VJP} (applied to $g$) and $m\in \Nn$ is the corresponding natural number, then $B_{>r+m}(b_1)\sub p_1(\CK)$ and therefore $v(b_1-c)<r_0+m$ for all $c\in C$. In particular, $B_{>r}(b_1)$ is contained in a ball $m$-next to $C$. So $v(g')$ is constant on $B_{>r}(b_1)$. If $g'(t)\equiv 0$, the claim holds trivially. Otherwise, by Fact \ref{F: VJP} $g(B_{>r}(b_1))$ is an open ball of radius $v(g'(b_1))+r$. Since $g(B_{>r}(b_1))\sub \CO$, the claim follows. 
	
	Replacing $r$ with $r-1$, if needed, we may assume that $r$ is even. Applying the claim to $r/2$ the result follows. 
\end{proof}

\begin{lemma}\label{L:near-affine in K/O}
	Let $f:K\to K$ be an $A$-definable partial function and $p\vdash \dom(f)$ a complete large type over $A$. If $f$ descends to $K/\CO$ on $p$ then for every $a\models p$ there is a large  ball $B$, $a\in B\sub p(\CK)$, such that for all $x\in B$.
	\[
	f(x)-f(a)-f'(a)(x-a)\in \m.
	\]
\end{lemma}
\begin{proof}
	By Lemma \ref{L:derivtives in O}(2), $f'(c)\in \CO$ for all $c\models p$. We may thus assume that $f'(c)\in \CO$ for all $c\in \dom(f)$.
	
	For every finite $A$-definable set $C\subseteq K$, $p(\CK)\cap C=\emptyset$. Let $a\models p$ and $B_0\subseteq p(\CK)$ be a large ball containing $a$. Fix $r<\mathbb{Z}$ such that $r(B_{0})-r<\mathbb{Z}$. We let $B=B_{>r}(a)$, then for any $m\in \Nn$ we see (Remark \ref{rem:next}) that $B$ is contained in a ball $m$-next to $C$.  By Taylor's theorem \cite[Theorem 3.1.2]{hensel-minII}, for every $x\in B$. 
	
	\begin{equation}
	v(f(x)-f(a)-f'(a)(x-a))= v(\frac{1}{2}f''(a)(x-a)^2)
	\end{equation}
	
	As $f'(B)\sub \CO$,  Lemma \ref{L:derivtives in O}(3), applied to $f'$,  gives a large open ball $B'$, $a\in B'\subseteq p(\CK)$, such that for all $b\in B'$,
	$$v(f''(b))+2r(B')>0.$$

	Thus for $B_1=B\cap B'$, $v(f''(a))+2r(B_1)> 0$ and for every $x\in B_1$, $v((x-a)^2)>2r(B_1)$. Since all of the above remains true if we replace $r$ with $r+1$ we can ignore the contribution of $v(\frac{1}{2})$. Thus, it follows from the above equation that  $v(f(x)-f(a)-f'(a)(x-a))>0$, as required. 
\end{proof}

To conveniently apply the result of the last lemma,  we make the following definition:

\begin{definition}
	A definable  function $\lambda: K/\CO\to K/\CO$ is a {\em scalar-endomorphism} of $K/\CO$  if there exists $a\in \CO$ such that  $\lambda(x+\CO)=\lambda_a(x+\CO):=\pi(ax)$ where $\pi:K\to K/\CO$ is the quotient map. More generally, $\lambda:(K/\CO)^n\to (K/\CO)$ is a scalar endomorphism if $\lambda(x_1,\dots, x_n)=\sum \lambda_i(x_i)$ where $\lambda_i$ are scalar-endomorphisms in one variable. 
\end{definition}

We now turn to proving that definable functions on $(K/\CO)^n$ are generically given by translates of scalar endomorphisms. We start by observing that the germs (at 0) of definable scalar-endomorphism are $\0$-definable (recall that $\bF\subseteq \dcl(\emptyset)$): 
\begin{lemma}\label{L:germs of endo in K/O}
	For any scalar  endomorphism $\lambda:K/\CO\to K/\CO$,  there exists $r\in \CO_\bF$ such that $\lambda(x)=\lambda_r(x)$ for all $x$ in some open ball around $0$.
\end{lemma}
\begin{proof}
	By assumption $\lambda=\lambda_a$ for some $a\in \CO$. By Fact \ref{F:F is dense in Z-radius balls}, there exists $a'\in \CO_\bF$ such that $v(a-a')>\mathbb{Z}$ thus,  $\lambda_{a-a'}(b)=0$ for any $b\in K/\CO$ with $v(b)\in \mathbb{Z}$.  By compactness, $\lambda $ and $\lambda_{a'}$ agree on some open ball around $0$.
\end{proof}

The proof of \cite[Proposition 6.16]{HaHaPeVF} (for unary functions) goes through unaltered using Lemmas \ref{L:derivtives in O} and \ref{L:near-affine in K/O} giving the following. 
\begin{proposition}\label{P: K/O-affine}
	Let $f:K/\CO\to K/\CO$ be an $A$-definable partial function with $\dom(f)$ open in the  ball topology. If $f$ lifts to $K$ then for every $a\in \dom(f)$ with $\dpr(a/A)=1$ there exist a ball $U\sub \dom(f)$ with $a\in U$ and a scalar endomorphism $L:(K/\CO)\to (K/\CO)$ such that $f(x)=L(x-a)+f(a)$ for all $x\in U$. 
\end{proposition}

\begin{corollary}
	% Let $f:K/\CO\to K/\CO$ be a partial $A$-definable function lifting to $K$ on some non-algebraic $p\in S_1(A)$. Then there is  $t\in \CO_{\bF}$  such that $f(x)-\lambda_t(x)$   is locally constant on  $p$.
 	Let $f:K/\CO\to K/\CO$ be a partial $A$-definable function, $p\in S_1(A)$ a non-algebraic type concentrated on $\dom(f)$. Then there is  $t\in \CO_{\bF}$  such that $f(x)-\lambda_t(x)$   is locally constant on  $p$, with respect to the ball topology.
\end{corollary}
\begin{proof}
	Since $\CK$ has definable Skolem functions, $f$ lifts to $K$ at $p$. By Proposition \ref{P: K/O-affine}, if $a\models p$ then there is some scalar endomorphism $\lambda$ (that may depend on $a$) such that $f(x)=f(a)+\lambda (x-a)$ for all $x$ in some ball $B\sub p(\CK)$. By  Lemma \ref{L:germs of endo in K/O}, we may assume that $\lambda=\lambda_t$ for some $t\in \CO_{\bF}$. Hence, $f(x)=\lambda_t(x)+d$ on some large sub-ball of $B$, for some $d\in K/\CO$ depending on $a$. As $t\in \CO_{\bF}$ and $\bF\subseteq \dcl(\emptyset)$, for every $\alpha\models p$, $f(x)-\lambda_t(a)$ is locally constant at $\alpha$. 
\end{proof}

Summing up all of the above, we get:
\begin{proposition}\label{P: n-dim affine}
	% Let $f:(K/O)^n\to K/\CO$ be a partial $A$-definable function lifting to $K$ on some  $p\in S_n(A)$ with $\dpr(p)=n$. Then there exists a scalar-endomorphism $\lambda: (K/\CO)^n\to (K/O)$ definable over $\CO_{\bF}$ such that $f-\lambda$ is locally constant on $p(\CK)$. 
 Let $f:(K/O)^n\to K/\CO$ be a partial $A$-definable function,  $p\in S_n(A)$ concentrated on $\dom(f)$ with $\dpr(p)=n$. Then there exists a scalar-endomorphism $\lambda: (K/\CO)^n\to (K/O)$ definable over $\CO_{\bF}$ such that $f-\lambda$ is locally constant on $p(\CK)$. 
\end{proposition}
\begin{proof}
	Without loss of generality $A=\0$. Let $a=(a_1,\dots, a_n)\models p$ and consider the function $g_1(t):=f(t,a_2,\dots, a_n)$. By what we have shown, there is $r_1\in \CO_{\bF}$ such that $g_1(t)-\lambda_{r_1}$ is locally constant on $\tp(a_1/a_2,\dots, a_n)$. Similarly, for all $1\leq i\leq n$ we can find $\lambda_{r_i}$ such that $g_i(t)=f(a_1,\dots a_{i-1}, t,  a_{i+1}\dots, a_n)$ is locally constant on $\tp(a_i/a_1,\dots, \hat{a_i}, \dots. , a_n)$.  The result follows. 
\end{proof}

Our next goal is to extend the above result to types not necessarily of full dp-rank in $(K/\CO)^r$. To do so, we introduce a notion that will be of importance in the sequel as well: 

\begin{definition}
	A set $S\sub (K/\CO)^n$ has \emph{minimal fibres} if $\dpr(S)=k$ and there exists a coordinate projection $\pi:S\to (K/\CO)^k$ and some $m\in \mathbb{N}$ such that for every $y\in (K/\CO)^k$, $|\pi^{-1}(y)|\leq m$ and there is no definable (possibly over additional parameters) $S_1\subseteq S$  such that $\dpr(S_1)=\dpr(S)$ and for every $y\in (K/\CO)^k$, $|f^{-1}(y)\cap X_1|<m$.
\end{definition}

\begin{remark}
   It is easy to see that  a set $S\sub (K/\CO)^n$ has minimal fibres if and only if  $\dpr(S)=k$ and there exists a coordinate projection $\pi:S\to (K/\CO)^k$ with finite fibres such that for any definable  $S'\sub S$,  $\dpr(\pi^{-1}(\pi(S')))\setminus S')<\dpr(S)$.
\end{remark}

Notice that if some $S\sub (K/\CO)^n$ of rank $k$ projects finite-to-one into $(K/\CO)^k$ then there exists $S'\sub S$,  $\dpr(S')=\dpr(S)$ (possibly defined over additional parameters) such that $S'$ has minimal fibres with respect to the same projection. Indeed, we just choose $S'\sub S$ with the property that the fibres of $\pi\restriction S'$ are of minimal size among all definable subsets of $S$ of maximal dp-rank. The following is a local version of this observation: 
\begin{lemma}\label{L:wma has finite fibres around a generic in K/O}
	Let $X\subseteq (K/\CO)^r$ be an $A$-definable subset. For any $a\in X$ with $\dpr(a/A)=\dpr(X)$ there exists  $X'\subseteq X$ definable over some set $C\supseteq A$ such that $a\in X'$, $\dpr(a/C)=\dpr(a/A)$ and $X'$ has minimal fibres.
\end{lemma}
\begin{proof}
	Let $p=\tp(a/A)$ and set $\dpr(a/A)=\dim_{\acl}(a/A)=\dpr(X)=n$. Let $\tau:(K/\CO)^r\to (K/\CO)^n$ be a coordinate projection such that $\dim_{\acl}(\tau(a)/A)=\dim_{\acl}(a/A)$. Letting $q=\tau_*p$ it follows that $\dpr(q)=n$.   
	
	By Proposition \ref{P:sgenos for K/O}, there exists a large ball $B\subseteq q(\CK)$ containing $\tau(a)$; so without loss of generality $X=\tau^{-1}(B)\cap p(\CK)$, as it is a definable set of dp-rank $n$. Note that $\tau \restriction X$ has finite fibres and as $\dpr(X)=n$, there exists $X'\subseteq X$ satisfying that $\dpr(X')=\dpr(X)$ with minimal fibres for $\tau$; assume that $X'$ is definable over some parameter set $C'$.
	
	Let $a'\in X'$ be an element with $\dpr(a'/C')=\dpr(X')$. Since $a'\models p$ there exists an automorphism $\sigma$ over $A$ mapping $a'$ to $a$. Then $\sigma(X')\subseteq X$ contains $a$ and has minimal  fibres for $\tau$ and is definable over $C:=\sigma(C)$.
\end{proof}

For any  function $f:X\to Y$, we write $[x]_f=f^{-1}(f(x))$.
For the next result, recall that we view $\bF/\CO_{\bF}$ as embedded inside $K/\CO$.

\begin{lemma}\label{L:fibres is a coset in K/O}
	Let $X\sub (K/\CO)^r$ be an $A$-definable set with minimal fibres.  Assume that $\dpr(X)=n$ and let $a\in X$ be such that $\dpr(a/A)=n$. Then there exists an $A$-definable subset $X_1\sub X$ with $a\in X_1$, a finite subgroup $G_a\sub (\bF/\CO_{\bF})^{r}$  and a coordinate projection $\tau:(K/\CO)^r\to (K/\CO)^n $ such that for every $b\in X_1$, $[b]_\tau=G_a+b$.
\end{lemma}
\begin{proof}
	Let $\tau:X\to (K/\CO)^n$ be a coordinate  projection witnessing that $X$ has minimal  fibres.  For $b\in X$, set  $G_b=\{x-b:x\in [b]_\tau\}$. 
	%For $b\in (K/\CO)^r$ let $v(b)=\min\{v(b_i):1\leq i\leq r\}$. 
	
	Let $a\in X$, $\dpr(a/A)=n$. We claim that for all $b\in [a]_{\tau}$, $v(b-a)\in \mathbb Z$. Indeed,  assume towards a contradiction  that there is $b\in [a]_\tau$ with $v(b-a)<\mathbb{Z}$. By Proposition \ref{P:sgenos for K/O}  we can find a ball $U$ containing $a$, defined over some parameters $C$ such that $\dpr(b/C)=n$ and $b\notin U$. Thus $|[a]_\tau\cap U|<|[a]_\tau|$. Now the set $S'=\{x\in X:|[x]_\tau\cap U|<|[a]_\tau]|\}$ contradicts the  assumption that $X$ has minimal fibres. 
	
	Therefore, $v(G_a)\subseteq \mathbb{Z}$ and by Fact \ref{F:F is dense in Z-radius balls}(2), $G_a\subseteq \bF/\CO_{\bF}$ and so it is $\emptyset$-definable. Thus, replacing $X$ by a subset $X_1$ of the same dp-rank, we may assume that $G_a=G_b$ for all $b\in X$. In particular, $G_a=G_b$ for any $b\in [a]_\tau$. It easily follows that $G_a$ is a subgroup and $[a]_\tau$ its coset.
\end{proof}

\begin{definition}
	
	For $a\in (K/\CO)^r$, let 
	\[
	Z(a):=\{x\in (K/\CO)^r: v(x-a)\in \mathbb{Z}\cup \{\infty\}\}. 
	\]
	For $X\sub (K/\CO)^r$ and $a\in X$ we let $Z_X(a):=Z(a)\cap X$. 
\end{definition}

Notice that $Z_X(a)$ is a $\bigvee$-definable set. The next lemma shows that $Z_X(a)$ does not depend on the choice of $X$: 

\begin{lemma}
	Let $X_1\subseteq X\subseteq (K/\CO)^r$ be $A$-definable subsets with minimal fibres and $a\in X_1$ with $\dpr(a/A)=n=\dpr(X)$. Then $Z_{X}(a)=Z_{X_1}(a)$.
\end{lemma}
\begin{proof}
	Let $\tau:X\to (K/\CO)^n$ be a coordinate projection witnessing minimal fibres. It follows that $[a]_\tau\cap X_1=[a]_\tau$. By replacing $X_1$ with a subset of minimal fibres, still of full dp-rank we may assume that for all $x\in X_1$ we have $[x]_\tau\cap X_1=[x]_\tau$.
	
	By Lemma \ref{L:full dpr is interior in K/O}(1), we have that $\tau(a)$ is in the interior of $\tau(X_1)$. Thus, there exists a ball $V\subseteq (K/\CO)^n$, with $\tau(a)\in V\subseteq \tau(X_1)$. By our choice of $X_1$, we have $\tau^{-1}(V)\subseteq X_1$; the result follows.
\end{proof}

% \textbf{From now on, we further assume that definable functions on $K/\CO$ lift to definable functions on $K$.} This holds, e.g., if  $\CK$ has definable Skolem functions.

The next lemma is the main result of this section. It states that definable sets are locally affine at all generic points: 

\begin{lemma}\label{L:sets are locally cosets in K/O}
	 Let $X\sub (K/\CO)^r$ be $A$-definable with minimal fibres, $\dpr(X)=n$. Let $a\in X$ be such that $\dpr(a/A)=n$. Then there exists $C\supseteq A$, with $\dpr(a/C)=n$, a $C$-definable ball  $U\sub (K/\CO)^r$, $0\in U$, (hence $U$ is a subgroup), a $\0$-definable subgroup $H\sub (K/\CO)^r$, and a $C$-definable $X_1\sub X$ containing $a$,   such that $X_1=a+(H\cap U)$. In particular,  $Z_{X_1}(a)=a+H(\bF/\CO_\bF)$. 
\end{lemma} 
\begin{proof}
	By Lemma \ref{L:fibres is a coset in K/O}, after reducing $X$ and permuting the coordinates, we may assume that the fibres  of the projection $\tau$ on the first $n$-coordinates are all cosets of the same finite subgroup $G\sub (\bF/\CO_\bF)^r$.
	Since, by definition, $\tau(G)=0$ we identify $G$ with a subgroup of $(\bF/\CO_\bF)^{r-n}$. We also let $\sigma: (K/\CO)^r\to (K/\CO)^{r-n}$ be the projection onto the last $r-n$ coordinates. 
	
	Let $k=|G|$, $(x,x')\in X$ such that $x\in \tau(X)$, denote 
	\[f(x):=\sigma\left( \sum_{b\in \tau^{-1}(x)}b \right)\in (K/\CO)^{r-n},\]
	and note that $f(x)=kx'$, as the sum over all elements of $G$ is $0$.  Let $p=\tp(\tau(a)/A)\vdash (K/\CO)^n$. Since in $\CK$ every definable function $f:(K/\CO)^n\to K/\CO$ lifts to $K$, we may apply Proposition \ref{P: n-dim affine} to find some ball $V\subseteq p(\CK)$ with $\tau(a)\in V$, such that $f$ is a translate of a scalar-endomorphism, on $V$. By Proposition \ref{P:sgenos for K/O}(1), we may assume that $V$ is defined over some $C\supseteq A$ with $\dpr(\tau(a)/C)=n$. Let $Y:=\mathrm{Graph}(f\restriction V)$. The graph of $f$ is a coset of a $\0$-definable subgroup $H_1\sub (K/\CO)^r$, and thus 
	\[
	Y=(b+H_1)\cap (V\times (K/\CO)^{r-n}),
	\] 
	for some $b\in (K/\CO)^r$. Let $X_1:=\{(x,y)\in X: (x, ky)\in Y\}$. Note that $X_1$ is definable over $C$ and contains $a$, and hence $\dpr(X_1)=n$.
	
	Because the map $(x,y)\mapsto (x,ky)$ is a $\0$-definable endomorphism of $(K/\CO)^n\times (K/\CO)^{n-r}$, the pre-image of $b+H_1$ under the map $(x,y)\to (x,ky)$ is of the form $a+H$ for some $\0$-definable subgroup $H$, and the pre-image of $V$ is a ball $V'\ni a$ in $(K/\CO)^r$, which we may write $V'=a+U$, for some ball $U\ni 0.$ Thus,  $X_1=a+H\cap a+U=a+(H\cap U)$. By Fact \ref{F:F is dense in Z-radius balls}(2), we now get that $Z_{X_1}(a)=a+H(\bF/\CO_\bF)$.
\end{proof}    

\begin{remark}
% 	\begin{enumerate}
		In the statement of the previous lemma, we cannot require that the ball $U$ is $\0$-definable. For example,  if $X$ itself
		was a ball around $0$, which is smaller than all $\0$-definable balls, then we cannot find a $\0$-definable such $U$. 
\end{remark}

\begin{corollary}\label{C: KO affine}
	If $f:(K/\CO)^r\to K/\CO$ is an $A$-definable partial function and $a\in \dom(f)$ is such that $\dpr(a/A)=\dpr(\dom(f))$ then there exists $C\supseteq A$, a $C$-definable coset $X:=H+d$ and a $\0$-definable scalar endomorphism $\ell:(K/\CO)^r\to K/\CO$ such that
	\begin{enumerate}
	    \item $\dpr(a/C)=\dpr(a/A)$. 
	    \item $a\in X$. 
	    \item $f\restriction X=\ell(x-a)+\ell(a)$. 
	\end{enumerate}
	
\end{corollary}
\begin{proof}
	By Lemma \ref{L:wma has finite fibres around a generic in K/O} we may assume that the graph of $f$ has minimal fibres. We can now apply Lemma \ref{L:sets are locally cosets in K/O} to the graph of $f$.
\end{proof}

%The next example will help us clarify an important distinction arising in Section \ref{S: asi to D}.
The next example shows that a group almost locally strongly internal to a distinguished sort $D$ need not be strongly internal to it. This distinction, that apparently can not be avoided, will have to be taken into account in our work here: 
\begin{example}\label{L: not lsi}
     Let $\mathbb F$ be a residual quadratic extension of $\qp$ and let $\CK\succ \mathbb{F}$ be a sufficiently saturated extension. Let $C_p\le K/\CO$ be a cyclic subgroup of order $p$, and $G:=(K/\CO)/C_p$. Then $G$ is locally almost strongly internal to $K/\CO$ but not locally strongly internal to $K/\CO$. 
     
     On the other hand, letting $H=\{x\in G: px=0\}$ the group $G/H$ is strongly internal to $K/\CO$.
\end{example}
\begin{proof}
    Since the quotient map $\pi: K/\CO\to G$ is finite-to-one it follows directly from  Lemma \ref{L:asi diagaram K/O}(2) that $G$ is almost strongly internal to $K/\CO$. 
    
    We now verify that $G$ is not locally strongly internal to $K/\CO$. Assume towards a contradiction that there exists a definable injection $f:X\to (K/\CO)^r$ with $X$ infinite. Let $V=\pi^{-1}(X)$ and $\widehat f:V\to (K/\CO)^r$ the (clearly, definable) lifting of $f$ to $K/\CO$. As $\dpr(V)=1$, we may assume -- shrinking $V$ if needed -- that $V$ is an open ball. Assume everything is definable over $\emptyset$. Let $a\in V$ with $\dpr(a)=1$. Shrinking $V$ further, we may apply Proposition \ref{P: K/O-affine}, to deduce that  $\widehat f:=(\widehat f_1,\dots, \widehat f_r)$ where each $\widehat f_i$ is of the form $L_i(x-a)+\widehat f_i(a)$ on $V$, for some  scalar endomorphism,  $L_i$.
    By Fact \ref{F:F is dense in Z-radius balls}(3) and Lemma \ref{L:Tor of ball}, $V+C_p=V$.  Because $\widehat f$ factors through $C_p$ on $V$, it follows that the scalar endomorphisms $L_i$ are all invariant under $C_p$, and therefore $C_p\sub \ker(L_i)$ for $i=1,\ldots, r$.
    
    Since $\mathbb F$ is a residual extension of $\qp$, the subgroup $H=\{x\in K/\CO: v(x)\geq -1\}$, where $1\in \Gamma$ is the minimal positive element, is readily seen to contain $C_p$,  and $|H|=|\bk_K|=p^2$. It is also easy to verify that for every scalar endomorphism $\lambda$ of $K/\CO$ with non-trivial kernel, we have $H\sub \ker(\lambda)$, so $C_p\subsetneq H\sub \ker(L_i)$ for every $L_i$. 
    
   Thus, the function $\widehat f$ is invariant under the group $H\supseteq C_p$,
    contradicting our assumption that $f$ was injective.
    
    For the final statement, since $\mathbb{F}/\CO_\mathbb{F}$ is isomorphic to $\mathbb{Z}(p^\infty)\oplus \mathbb{Z}(p^\infty)$, $G/H$ is definably isomorphic to $K/\CO$ as witnessed by the map $x\mapsto px$.
\end{proof}

\begin{remark}
It can be shown that if $\mathbb F$ is a residual quadratic extension of $\mathbb Q_p$, as above, and $H\sub \mathbb F/\CO_{\mathbb F}$ is a finite subgroup, then $(\mathbb F/\CO_{\mathbb F})/H$ is abstractly isomorphic to $\mathbb F/\CO_{\mathbb F}$. However, the above example shows that a definable isomorphism need not exist.

     In contrast, it can be shown that   for every finite subgroup $H\sub \qp/\mathbb Z_p$, the group $(\qp/\mathbb Z_p)/H$ is definably isomorphic to $\qp/\mathbb Z_p$. Furthermore, this remains true in elementary extensions of $\mathbb Q_p$.
% \end{enumerate}
\end{remark}

\section{Vicinic sorts and their properties}\label{S: asi to D}

In the present section we  develop an  \emph{ad hoc} axiomatic setting allowing us to unify the treatment of all unstable distinguished sorts considered in the sequel. We then apply the technical tools we develop to show that, given a group $G$  almost strongly internal to such a sort $D$, there exists a finite normal subgroup $H\trianglelefteq G$ such that  $G/H$ is strongly internal to $D$ and furthermore, it is ``a $D$-group'' (to be defined below). This suffices for the construction of an infinitesimal subgroup in $G/H$.  \\

\noindent \textbf{Throughout Section \ref{S: asi to D} we let $\CM$ be a multi-sorted, $|T|^+$-saturated structure and $D$ a dp-minimal definable set.}

\subsection{The key geometric properties}\label{ss:key geo properties}

To motivate our axiomatic setup, we isolate key geometric properties shared by SW-uniformities, $\Zz$-groups and $K/\CO$, for $\CK$ a $p$-adically closed field. Those are, in the latter setting,  Corollary \ref{C: dpr=acl-dim in K/O} (asserting that dp-rank is the same as  $\acl(\cdot)$-dimension) and Proposition \ref{P:sgenos for K/O} (asserting the existence of generic neighbourhoods). These two properties  constitute our axiomatisation of vicinic sorts (Section \ref{ss: vicinic}).

An important property shared by stable groups, groups definable in o-minimal structures and other tame settings is that the product of two independent generic elements of the group is itself generic. It turns out that a variant of this property can be deduced, in the $p$-adically closed case, from Lemma \ref{L:asi diagaram K/O}(2). Statements  analogous to Lemma \ref{L:asi diagaram K/O}(2) (Lemma \ref{L: asi from surjective} for SW-uniformities, Lemma \ref{L:asi diagaram with EI} for $\Zz$-groups) imply the same property of groups in the remaining cases covered by our setting. This  geometric property of groups  is encapsulated in our notion of $D$-groups (Section  \ref{ss: Dgps})\\

In all of our settings, the induced structure on $K/\CO$ does not satisfy Exchange (there are definable locally constant functions with infinite image). Nevertheless, the fact below, together with  its analogue, Corollary \ref{C: dpr=acl-dim in K/O}, for $K/\CO$ in the $p$-adic case, give dp-rank enough geometric flavour to get us going in all the unstable distinguished sorts: 

\begin{fact}\label{Gamma:A1}
	If $D$ is an SW-uniformity (in some structure) or a stably embedded $\Zz$-group  then $\dim_{\acl}(a/A)=\dpr(a/A)$ for all $a\in D^n$ and any parameter set $A$. 
\end{fact}
\begin{proof}
    For SW-uniformities this follows from \cite[Propsition 2.4]{SimWal} noting that by Simon's work dp-rank is local in dp-minimal theories,  \cite[Theorem 0.3]{Simdp}. If $D$ is a $\Zz$-group, then this follows from \cite[Theorem 0.3]{Simdp} since $D$ satisfies exchange.
\end{proof}

The existence of generic neighbourhoods, originally proved for SW-uniformities (Fact \ref{Gen-Os in SW}),  and extended above to $K/\CO$ in the $p$-adically closed setting (Proposition \ref{P:sgenos for K/O}),  is crucial  for carrying out local analysis at a generic  point.  The following may be viewed as the $\Zz$-group analogue:   

\begin{lemma} \label{Gamma: A2}
	Let $(D,+,<)$ be a $\Zz$-group. Let  $A$ be any set of parameters, $g\in D^n$, $h\in D^m$. For any set of parameters $B$ and $B$-definable $X\subseteq D^n$, if $g\in X$ and $\dpr(g/B)=n$ then there exists $C\supseteq A$ and a $C$-definable set $X_1\sub X$ such that $\dpr(X_1)=n$, $g\in X_1$ and $\dpr(g,h/A)=\dpr(g,h/C)$. 
\end{lemma}
\begin{proof}
	By \cite[Lemma 3.14(2)]{HaHaPeVF}, we can find a small model $L\supseteq B$ for which $\dpr(g/L)=n$. By \cite[Lemma 3.4]{OnVi}, there exists an $L$-definable subset  $X_0\subseteq X$ containing $g=(g_1,\dots, g_n)$ of the form $I_1\times\dots \times I_n$, where for each $i\leq n$, $g_i\in I_i=\{\alpha_i \leq x\leq \beta_i: x\equiv_{N_i} c_i\} $, for some $N_i\in \mathbb N$,  where both intervals $[\alpha_i,g_i]$ and $[g_i,\beta_i]$ are infinite and $0\leq c_i<N_i$.
	
	For simplicity, we prove the result for $n=1$. The general case follows by induction. We use compactness to find $\alpha_1<\alpha_1'<g_1$ such that $\dpr(\alpha_1'/g_1,hL)=1$ and $[\alpha_1',g_1]$ infinite and then $g_1<\beta_1'<\beta_1$ such that $\dpr(\beta_1'/g_1,h,\alpha_1'L)=1$ and $[g_1,\beta_1']$ infinite. We have $g_1\in I_1'=\{\alpha_1
	\leq x\leq \beta_1': x\equiv_{N_i} c_i\}\subseteq I_1$ and by exchange $\dpr(g_1,h/\alpha_1',\beta_1',L)=\dpr(g_1,h/L)$.
\end{proof}

The next lemma is a rather simple first approximation to more technical results which follow:

\begin{lemma}\label{L:asi diagaram with EI}
	Let $\CM$ be any structure, and $S$ a stably embedded definable set. If $S$ has elimination of imaginaries, $X$ is (almost) strongly internal to $S$ and $g:X\to Y$ is a definable surjection then $Y$ is (almost) strongly internal to $S$.
\end{lemma}
\begin{proof}
	We consider the almost strongly internal case only, as the strongly internal case is similar and easier. Let $f:X\to S^n$ be a finite-to-one definable function. For any $y\in Y$ consider the definable set $W_y=f(g^{-1}(y))\subseteq S^n$. By elimination of imaginaries in $S$, there exists a definable function $h:Y\to S^m$ such that $h(y_1)=h(y_2)$ if and only if $W_{y_1}=W_{y_2}$. It will suffice to show that  $h$ is finite-to-one. So assume towards a contradiction that this is not the case. I.e., there exists $b\in S$ with $H:=h^{-1}(b)$ infinite. Fix $y_1\in H$. Then for any other $y_2\in H$, obviously, $g^{-1}(y_1)\cap g^{-1}(y_2)=\0$ (but $g^{-1}(y_2)\neq \0$ because $g$ is onto). Since $f$ is finite to one and $H$ is infinite, for $w\in W_{y_1}$ there is $y_2\in H$ such that $f^{-1}(w)\cap g^{-1}(y_2)=\0$, contradicting the assumption that $W_{y_1}=W_{y_2}$. 
\end{proof}

Since Presburger Arithmetic eliminates imaginaries, this gives the analogue of Lemma \ref{L:asi diagaram K/O}(2) for stably embedded $\Zz$-groups. For SW-uniformities, the analogous statement is Lemma \ref{L: asi from surjective}.

\subsection{Vicinic dp-minimal sets}\label{ss: vicinic}
We now introduce the  axiomatic setting we are interested in, but first:

\begin{definition}
    Let $X$ be an $A$-definable set of finite dp-rank, $a\in X$ and $B\supseteq A$ a set of parameters.
    \begin{enumerate}
        \item The point $a$ is $B$-\emph{generic}  in $X$ (or, {\em generic in $X$ over $B$}) if $\dpr(a/B)=\dpr(X)$.
        \item For an  $A$-generic  $a\in X$, a set $U\sub X$ is \emph{a $B$-generic vicinity of $a$ in $X$} if $a\in U$, $U$ is $B$-definable, and $\dpr(a/B)=\dpr(X)$ (in particular, $\dpr(U)=\dpr(X)$).
    \end{enumerate}
\end{definition}

Note that if $D$ is an SW-uniformity, then a generic vicinity of a generic point $a$ in a set $X\sub D^k$ is, in fact, a neighbourhood of $a$ in the relative topology of $X$. The existence of generic neighbourhoods in SW-uniformities is given by Fact \ref{Gen-Os in SW}. 

\begin{definition}
    A dp-minimal set $D$ is \emph{vicinic} if it satisfies the following axioms: 
\begin{enumerate}
    \item[(A1)] $\dim_{\acl}=\dpr$; i.e. for any tuple $a\in D^n$ and set $A$, $\dim_{\acl}(a/A)=\dpr(a/A)$.
    % \item[(A2)] Let $f:X\to Y$ be a definable finite-to-one surjective map and assume that $X$ is almost strongly internal to $D$ then there exists a definable subset $Y_1\subseteq Y$ with $\dpr(Y_1)=\dpr(Y)$ almost strongly internal to $D$.
    \item[(A2)] For any sets of parameters  $A$ and $B$,  for every $A$-generic elements $b\in D^n$, $c\in D^m$ and  any $B$-generic vicinity $X$  of $b$ in $D^n$,  there exists $C\supseteq A$ and a $C$-generic vicinity of $b$ in $X$ such that $\dpr(b,c/A)=\dpr(b,c/C)$.
\end{enumerate}
\end{definition}

Note that in Axiom (A2) it is crucial that the parameter set $B$ need not contain $A$. The topological intuition is that if $b$ is in the interior of a $B$-definable set $X$ then we can find a smaller neighbourhood of  $b$ defined over a new parameter set $C$, that is generic with respect to all the initial data.

\begin{remark}\label{R: after vicinic}
     By Axiom (A1), for every parameter set $A$, every $c\in D^m$ is inter-algebraic over $A$  with an $A$-generic $c'\in D^k$, for $k=\dpr(c/A)$.  Thus, (A2) remains true if we drop the requirement that $c$ is generic.
\end{remark}

Let us note that indeed all the unstable distinguished sorts in our various settings are vicinic:

\begin{fact}\label{E:dist sorts are vicinic}
    \begin{enumerate}
        \item Every SW-uniformity is vicinic.
        \item If $\CK=(K,+,\cdot,v,\dots)$ is  either V-minimal,  power-bounded T-convex, or a $p$-adically closed valued field, then all the distinguished sorts, except $\bk$ in the V-minimal setting, are vicinic. 
    \end{enumerate}
    
    %If $\CK$ is not P-minimal then every such sort is an SW-uniformity.
    \end{fact}
    \begin{proof} 
    When $D$ is an SW-uniformity, Axiom (A1) holds by \cite[Proposition 2.4]{SimWal} and Axiom (A2) is Fact \ref{Gen-Os in SW}. This proves (1) and (2) follows in all cases except when  $D=K/\CO$ and $D=\Gamma$ in the $p$-adic setting. 
    
   When $\CK$ is $p$-adically closed, then for $K/\CO$, see Corollary \ref{C: dpr=acl-dim in K/O} for (A1), and Proposition \ref{P:sgenos for K/O} for (A2). For $\Gamma$, see Fact \ref{Gamma:A1}, for (A1) and Lemma \ref{Gamma: A2}, for (A2).
\end{proof}

\begin{question}
 Does  every dp-minimal distal structure satisfy Axiom (A2)?  
 We cannot expect the converse to hold, since by the above fact, the valued field sort of an algebraically closed field of equi-characteristic $0$ satisfies Axioms (A2), but it is not distal.   
\end{question}

\textbf{For the rest of Section \ref{S: asi to D},  we assume that $D$ is a vicinic sort.}\\

We first note  an immediate implication of Axiom (A1):
\begin{lemma}\label{L:acl-dp in asi-abstract}
Let $X$ be a definable set $a\in X$ and assume that $X$ is almost strongly internal to $D$ over $A$.  Then for any $B\supseteq A$ we have $\dpr(a/B)=k$ if and only if there exists $a'\in D^k$ such that $\dpr(a'/B)=k$, $a'\in \dcl(aA)$ and $a\in \acl(a'B)$.
\end{lemma}

We now start developing the technical tools needed for the construction of infinitesimal groups in the setting of vicinic structures.  We first want to show that boxes form vicinity-bases at generic points in the following sense:
\begin{lemma}\label{L:genos implies A4}
      For any $A\subseteq \CM$,  $A$-generic $b=(b_1,\dots, b_n)\in D^n$ and $c\in D^m$, and  any $A$-generic vicinity $X$ of $b$ in $D^n$, there exists $C\supseteq A$ and $C$-generic vicinities $I_i\ni b_i$ in $D$, for $i=1,\dots, n$ such that  $I_1\times\dots\times I_n\subseteq X$  and $\dpr(b,c/C)=\dpr(b,c/A)=n+m$.
\end{lemma}
\begin{proof}
 We use induction on $n$, where the case $n=1$ is an immediate application of (A2) to $b\in D$, and $c\in D^m$.
 
 We now consider $b=(b_1,\ldots, b_n)$ and $X$ an $A$-vicinity of $b$ in $D^n$. Let $b'=(b_1,\ldots, b_{n-1})$, and apply (A2)  to $b_n\in X_{b'}\sub D$.
 
We then find $C'\supseteq A$ and a $C'$-vicinity $I\sub X_{b'}$ of $b_n$ in $D$, such that 
$\dpr(b_n, (b',c)/C')=n+m$.  Let $$X_1:=\{x\in D^{n-1}: (\forall y)(y\in I\to (x,y)\in X)\}.$$  This is a $C'$-definable set, containing $b'$. So it is a $C'$-vicinity of $b'$ in $D^{n-1}$. By induction (now replacing $c$ with $b_n,c$), there exist $I_i\sub D$, $i=1,\ldots, n-1$, defined over $C\supseteq C'$, such that $b'\in \prod_{i=1}^{n-1}I_i\sub X_1$, with $\dpr(b',b_n,c/C)=n+m$. The set  $\prod_{i=1}^n I_i\sub X$ is the desired vicinity of $b$.
\end{proof}

\subsection{Functions with minimal fibres}
The notion of definable sets with minimal fibres, with respect to some finite-to-one projection, appeared already in Section \ref{S:P-minimal}.
We slightly generalize.

\begin{definition}
\begin{enumerate}
\item  For $X$ a definable set of finite dp-rank,  a definable function $f:X\to Y$ has {\em minimal fibres} if there exists some $m\in \mathbb N$ such that for every $y\in Y$, we have $|f^{-1}(y)|\leq m$, and there is no definable $X_1\sub X$ (possibly over additional parameters), with $\dpr(X_1)=\dpr(X)$, such that  for every $y\in Y$, $|f^{-1}(y)\cap X_1|<m.$

\item A set $X\sub D^n$ {\em has minimal fibres (in $D$)} if there exists a coordinate projection $\pi:D^n\to D^m$ such that $\dpr(X)=m$, and $\pi$ has minimal fibres.
\end{enumerate}

\end{definition}

\begin{remark} \label{remark1}
\begin{enumerate}
\item Notice that an $A$-definable finite-to-one $f:X\to Y$  has minimal fibres if and only if for every $B\supseteq A$ and every $B$-generic $a\in X$, all elements of $[a]_f$ satisfy the same type over $f(a)B$.

    \item Every definable set $Z\subseteq D^r$ has a definable subset $Z_1\subseteq Z$ (possibly over additional parameters) such that $\dpr(Z_1)=\dpr(Z)$ and $Z_1$ has minimal fibres. %Thus every (almost) $D$-critical set contains an (almost) $D$-set.
    \item If $Z\subseteq D^r$ has minimal  fibres and $X'\subseteq X$ with $\dpr(X')=\dpr(X)$ then $X'$ has minimal fibres as well.
\end{enumerate}

\end{remark}

Before proceeding to the next lemma we recall that for any  function $f:X\to Y$, we write $[x]_f=f^{-1}(f(x))$.\\

Our next goal is to show that Axiom (A2) can be pulled back via maps with minimal fibres, in the following sense:
\begin{lemma} \label{L:abstract general gen-os}
Let $X$ be definable in $\CM$  with $\dpr(X)=n$  and  $f:X\to D^n$ an $A$-definable function with minimal fibres. Let $b\in X$ be $A$-generic and $c \in \CM$ inter-algebraic over $A$ with some $d\in D^m$.

Then, for every parameter set $B$ and every $B$-generic vicinity $Y\sub X$ of $b$, there exists $C\supseteq A$ and a $C$-generic vicinity $Y_1\sub Y$ of $b$, such that $\dpr(b,c/C)=\dpr(b,c/A)$.
\end{lemma}
\begin{proof}
Let $B$ be any set of parameters and $Y$ a $B$-generic vicinity of $b$ in $X$. Since $f$ has finite fibres (because it has minimal fibres)  our assumptions imply that $f(b)$ is $A$-generic in $D^n$, and $f(Y)$ is a $B$-generic vicinity of $f(b)$ in $D^n$. Thus, by (A2), there exists $C\supseteq A$, and a $C$-generic vicinity $W\sub f(Y)$ of $f(b)$, such that $\dpr(f(b),d/C)=\dpr(f(b),d/A)$. Hence, $\dpr(b,c/C)=\dpr(b,c/A)$ as well.

Let $m$ be the size of maximal $f$-fibres.
Since $f$ has minimal fibres, for every $C$-generic $y\in Y$, we have $|[y]_f|=m$, so $[y]_f\subseteq Y$. Thus, the set $W_1\sub W$, of all $w\in W$ such that $|f^{-1}(w)|=m$ satisfies $\dpr(W_1)=\dpr(W)$, and we have $Y_1:=f^{-1}(W_1)\subseteq Y$. The set $Y_1$ is the desired $C$-generic vicinity of $b$.
\end{proof}

Now we wish to pull back the conclusion of Lemma \ref{L:genos implies A4} via functions with minimal fibres. We first note:

\begin{lemma}\label{L:product of min fibres is such}
Assume that $f_i:X_i\to Y_i$, $i=1,2$, have minimal fibres. Then $(f_1,f_2):X_1\times X_2\to Y_1\times Y_2$ has minimal fibres. In particular, if $X_i\sub D^{m_i}$ has minimal fibres, for $i=1,2$, then so does $X_1\times X_2\sub D^{m_1}\times D^{m_2}$.\end{lemma}
\begin{proof}
Let $f=(f_1,f_2):X_1\times X_2\to Y_1\times Y_2$ and assume for simplicity of notation that it is $\0$-definable. We apply Remark \ref{remark1}(1). Let $(a,b)\in X_1\times X_2$ be generic over some $B$. Then, by sub-additivity of dp-rank, $a$ is $Bb$ generic in $X_1$, so all elements of $[a]_{f_1}$ realize the same type over $f_1(a)bB$, hence also over $f_1(a)f_2(b)bB$. It follows that all elements of $[a]_{f_1}\times \{b\}$ satisfy the same type over $f_1(a)f_2(b)B$. Similarly, all elements in $\{a\}\times [b]_{f_2}$ satisfy the same type over $f_1(a)f_2(b)B$. It easily follows that all elements of $[(a,b)]_f$ realize the same type over $f(a,b)B$, as needed.
\end{proof}

\begin{corollary}\label{C:generic inside product of min fibres}
 For $i=1,2$ let $X_i$  be $A$-definable sets, $\dpr(X_i)=n_i$,  and $f_i:X_i\to D^{n_i}$ $A$-definable functions with minimal fibres. Let  $(d_1,d_2)\in X_1\times X_2$ be $A$-generic. Then for any $A$-generic vicinity $Y\sub X_1\times X_2$ of $(d_1,d_2)$, there exists $B\supseteq A$ and a $B$-generic vicinity of $(d_1,d_2)$ of the form $I_1\times I_2\subseteq Y$, with $I_i\sub X_i$.
\end{corollary}
\begin{proof}
  Let $f=(f_1,f_2):X_1\times X_2\to D^{n_1}\times D^{n_2}$. By  the minimality assumption on $f_1,f_2$, and Lemma \ref{L:product of min fibres is such}, we have $[(d_1,d_2)]_f\cap Y=[(d_1,d_2)]_f$.
Let
 $$W=\{w\in f(Y):f^{-1}(w)\cap Y=f^{-1}(w)\}.$$
 Then, $(f(d_1),f(d_2))\in W$, so $\dpr(W)=\dpr(d_1,d_2)=n_1+n_2$.
 
 By Lemma \ref{L:genos implies A4},  there exists $B\supseteq A$ such that $W$ contains a $B$-generic vicinity of the form $V_1\times V_2$ of $(f_1(d_1),f_2(d_2))$ in $f(X_1\times X_2)$. By definition, $f_1^{-1}(V_1)\times f_2^{-1}(V_2)$ is contained in $Y$ and it is thus a $B$-generic vicinity of $(d_1,d_2)$ satisfying the requirements. 
\end{proof}

\subsection{Critical, almost critical and $D$-sets}
Let $\CM$ and $D$ be as before (so $D$ is still vicinic). Recall Definition \ref{Def: internal} of a set \emph{strongly internal} and \emph{almost strongly internal} to $D$. We remind (and expand) a definition from \cite{HaHaPeVF}:

\begin{definition}\label{D: D-sets}
Let $S$ be a definable set of finite dp-rank.
    \begin{enumerate}
    \item A definable $X\sub S$ is $m$-internal to $D$ if there exists an $m$-to-one  $f:X\to D^n$, 
        \item A definable set $X\subseteq S$ is $D$-critical (for $S$) if $X$ is strongly internal to $D$, and has maximal dp-rank among all such subsets of $S$. Its dp-rank is the {\em $D$-critical rank of $S$}.
        \item A definable set $X\subseteq S$ is \emph{almost $D$-critical (for $S$)} if (i) it is almost-strongly internal to $D$ and has maximal dp-rank among all such subsets of $S$, and (ii) it is $m$-internal for minimal $m$ among all sets satisfying (i). We call $\dpr(X)$ the \emph{almost $D$-critical rank of $S$}.

            The set $X$ is called (almost) $D$-critical {\em over $A$} if the corresponding map of $X$ into $D^n$ is defined over $A$.
            
            \item A definable set $X\sub S$  is an (almost) $D$-set over $A$,  if $X$ is an (almost) $D$-critical set,  witnessed by an $A$-definable $f:X\to D^n$, such that {\em in addition} $f(X)\sub D^n$ has minimal fibres.
        \end{enumerate}
\end{definition}

Notice that the $D$-critical rank of $S$ is always bounded above by the almost $D$-critical rank of $S$, but the ranks need not be equal, as we shall now see. Thus, a $D$-critical set is not necessarily almost $D$-critical.

\begin{example}\label{E: crit dim is not a-crit dim}
    Let $\CK$ be an elementary extension of a quadratic residual extension of $\qp$, as in  Example \ref{L: not lsi}. Let $C_p\le  K/\CO$ be a cyclic subgroup of order $p$  and $G_0=(K/\CO)/C_p$  as in that example. Let $G=K/\CO\times G_0$. Then $G$ is almost strongly internal to $K/\CO$, since both $K/\CO$ and $G_0$ are; so its almost $K/\CO$-critical rank is $2$. 
    
    We claim that the critical $K/\CO$-rank of $G$ is $1$: The definable set $K/\CO\times\{0\}$ is strongly internal to $K/\CO$, so we only need to note that $G$ has no definable subset of dp-rank $2$  which is strongly internal to $K/\CO$. Indeed, if $X\sub G$ has dp-rank $2$, then, just like in Lemma \ref{L:full dpr is interior in K/O}, it contains a definable set of the form $Y_1\times Y_2$, with $Y_2\sub G_0$ of dp-rank $1$. So if $X$ were strongly internal to $K/\CO$ then $G_0$ would be locally strongly internal to $K/\CO$, contradicting Example \ref{L: not lsi}.
\end{example}

\begin{remark} \label{r; D-subset}
Let $X\subseteq S$ be an almost $D$-critical set, witnessed by $f:X\to D^k$. Then it is not hard to see that there exists $X_1\sub X$, defined over some $B\supseteq A$, that is a D-set for the same $S$, as witnessed by $f$. If $f$ is injective then we can find such an $X_1\sub X$ which is a $D$-set.
\end{remark}

We will use the following remark implicitly throughout.

\begin{remark} \label{important remark} 
\begin{enumerate}
    \item If $f:X\to D^n$ witnesses that $X$ is an (almost) $D$-set for $S$ then, since $f(X)$ has minimal fibres, we may compose $f$ with an appropriate coordinate projection to obtain a map with minimal fibres $\pi\circ f:X\to D^n$, such that $n=\dpr(X)$.
Thus, in Lemma \ref{L:abstract general gen-os} the assumption that  $\dpr(X)=n$ can be dropped when $X$ is an (almost) $D$-set. 
Similarly, Lemma \ref{C:generic inside product of min fibres} holds when $X_1,X_2$ are (almost) $D$-sets.

\item If $X_1,X_2\subseteq S$ are (almost) $D$-sets then so is $X_1\times X_2\sub S^2$. Indeed, by Lemma \ref{L:product of min fibres is such} it is sufficient to show that $X_1\times X_2$ is (almost) $D$-critical in $S\times S$: If $Y\sub S\times S$ is (almost) strongly internal to $D$ then so are the fibres $Y_x$ and  $Y^x$ for every $x\in S$. In particular $\dpr(Y_x)\le \dpr(X_1)$ and  $\dpr(Y^x)\le \dpr(X_2)=\dpr(X_1)$ for any $x\in S$. By sub-additivity $\dpr(Y)\le 2\dpr(X_1)$, so $2\dpr(X_1)$ is the (almost) $D$-critical rank of $S\times S$.

\item If $X\subseteq S$ is an (almost) $D$-set and $Y\subseteq X$ with $\dpr(Y)=\dpr(X)$ then $Y$ is also an (almost) $D$-set.
\end{enumerate}
\end{remark}

Below, if the ambient set is clear from the context or immaterial, we will just refer to (almost) $D$-critical sets, without explicit mention of $S$ (though such an $S$ of finite dp-rank is always assumed to exist in the background). \\

We end this section with a result on generic vicinities, which will play an important role in the next sections.
\begin{lemma}\label{filter base}
Assume that $X$ is an (almost) $D$-set over $A$ and $d\in X$ generic over $A$. Let $X_i\sub X$, $i=1,2$, be two $A_i$-generic vicinities of $d$ in $X$ for some $A_i\supseteq A$.

Then there exists $C$ and a $C$-definable $U\sub X_1\cap X_2$ which is a $C$-generic vicinity of $d$ in $X$.\end{lemma}
\begin{proof}
By Remark \ref{important remark} (1) we may apply Lemma \ref{L:abstract general gen-os} as follows: first apply the lemma to $X$ (viewed as $A_2$-definable), $d$ and the $A_1$-generic vicinity $X_1$ (in the role of $Y$ in the lemma) to obtain $U_1\sub X_1$ defined over some $A_2'\supseteq A_2$ such that $d\in U_1$ is $A_2'$-generic. Now $U_1\cap X_2$ is $A_2'$-definable, so an $A_2'$-generic vicinity of $d$ satisfying the requirements. 
\end{proof}

\subsection{$D$-groups}\label{ss: Dgps}

We proceed with a series of technical lemmas ultimately allowing us to construct, inside a group that is (almost)  strongly internal to $D$, a definable subset that is both (almost) strongly internal and sufficiently closed  under the group operation.

\begin{lemma}\label{L:general g/gh-abstract}
Let $G$ be an $A$-definable group. Let $X,Y\subseteq G$ be $A$-definable with $X$ (almost) strongly internal to $D$ over $A$, and fix  an $A$-generic $(g,h)\in X\times Y$.

If $\dpr(g/A,g\cdot h)<\dpr(X)$  then there exists a finite-to-one definable function from a subset of $X\times Y$ onto a set $W\subseteq  X\cdot Y\subseteq G$ satisfying $\dpr(W)>\dpr(Y)$.

\end{lemma}
\begin{proof} 
For simplicity of notation, assume that $A=\emptyset$ and write $k=gh\in G$. Assume  that $d:=\dpr(g/k)< m:=\dpr(X)$, and let $n=\dpr(Y)$.

Because $X$ is (almost) strongly internal to $D$ we may apply Lemma \ref{L:acl-dp in asi-abstract} to obtain  $a\in D^d$, $a\in \dcl(g)$ such that $g\in \acl(a,k)$.

Notice that each two pairs of $g,h,k$ are interdefinable over $\0$. E.g., the map $(x,y)\mapsto (x,xy)$ sends $(g,h)$ to $(g,k)$. Thus, we have
$\dpr(a,k)=\dpr(g,k)=\dpr(g,h)=m+n$.

 Since $a\in D^d$, we have $\dpr(a)\leq d$ so by sub-additivity of dp-rank we have $\dpr(k/a)\geq n+m-d>n$.

Let $\phi(x,k)$ be the formula over $a$ isolating $\tp(g/a,k)$.
Let $l:=|\phi(X,k)|$ and
$$Z=\{(x,y)\in X\times Y: \phi(x,x\cdot y)\land (\exists^{\le l} x') (x'\in X \land \phi(x',x\cdot y)\}.$$

This is an $a$-definable set containing $(g,h)$, thus its image under $(x,y)\mapsto x\cdot y$, call it $W$, is also $a$-definable and contains $k$, so $\dpr(W)\geq \dpr(k/a)>n$.
Our assumption on $\phi$ implies that the restriction of the group multiplication to $Z$ is a finite-to-one map. \end{proof}

The following statement is rather convoluted, although its proof is simple, based on what we already know. It is the main technical tool towards the definition of an (almost) $D$-group.
\begin{corollary}\label{C:g/gh - abstract}
    Let $G$ be an $A$-definable group. 

 (1) Assume that
 \begin{enumerate}
        \item[($\star$)] for every definable $X, Y\sub G$ and  a  finite-to-one surjection $f:X\to Y$,  if  $X$ is strongly internal to $D$ then there exists a definable subset $Y'\subseteq Y$ with $\dpr(Y')=\dpr(Y)$, such that $Y'$ is  strongly internal to $D$.
    \end{enumerate}

    Let $X_1\subseteq G$ be strongly internal to $D$ and  $X_2\sub G$  be $D$-critical, both over $A$.
    Then for every $A$-generic $(g,h)\in X_1\times X_2$,  we have $\dpr(g/A,g\cdot h)=\dpr(X_1)$.

 (2) Assume that
   \begin{enumerate} 
        \item[($\star_a$)]  for every definable $X, Y\sub G$ and  a  finite-to-one surjection $f:X\to Y$,  if  $X$ is almost strongly internal to $D$ there exists a definable subset $Y'\subseteq Y$ with $\dpr(Y')=\dpr(Y)$, such that $Y'$ is  almost strongly internal to $D$.
    \end{enumerate}

    Let  $X_1\subseteq G$ be almost strongly internal to $D$   and $X_2\sub G$ be  almost $D$-critical, both over $A$.
    Then for every $A$-generic $(g,h)\in X_1\times X_2$,  we have $\dpr(g/A,g\cdot h)=\dpr(X_1)$.
\end{corollary}
\begin{proof}
(1) 
Assume $X_1$ is strongly internal to $D$, $X_2$ is $D$-critical and assume towards a contradiction that $\dpr(g/A,g\cdot h)<m:=\dpr(X_1)$; then by Lemma \ref{L:general g/gh-abstract}, there exists a finite-to-one definable function $f$ from a subset $Z$ of $X_1\times X_2$ onto a definable subset $W\subseteq X_1\cdot X_2$ with $\dpr(W)>\dpr(X_2)$.

Since $X_1$ and $X_2$ are both strongly internal to $D$ so is $X_1\times X_2$, and hence also $Z$. By ($\star$), applied to $f:Z\to X_1\to W$, there exists a definable subset $W_1\subseteq W$ with $\dpr(W_1)=\dpr(W)$ and an injective map from $W_1$ into some $D^p$, namely $W_1$ is strongly internal to $D$ and $\dpr(W)>\dpr(X_2)$. This  contradicts the maximality of  $\dpr(X_2)$.

(2) This proof is almost identical, replacing everywhere ``strongly internal'' with ``almost strongly internal'' and using $(\star_a)$ instead of $(\star)$.
\end{proof}

The conclusions of Corollary \ref{C:g/gh - abstract} are important for much that follows. For the sake of clarity of exposition, we isolate this property of groups and define: 

\begin{definition}\label{Def: D-group} 
Let $D$ be a vicinic sort. An $A$-definable group $G$ is an \emph{(almost) $D$-group} if its (almost) $D$-critical rank is at least $1$ and for every $X_1\subseteq G$ (almost) strongly internal to $D$, every (almost) $D$-critical set $X_2\sub G$, both over some $B\supseteq A$, and for every $(g,h)$ generic in $X_1\times X_2$ over $B$, we have \[\dpr(g/B,g\cdot h)=\dpr(X_1).\]
\end{definition}

% Below we  show that groups interpretable in the valued fields in our settings are either $D$-groups or almost $D$-groups, for $D$ one of the distinguished sorts, by showing that they satisfy either ($\star$) or $(\star_a)$ of Corollary \ref{C:g/gh - abstract}. More precisely:

\begin{remark}\label{remark on D-groups}
    By Corollary \ref{C:g/gh - abstract}, if $D$ satisfies ($\star$) then every definable group $G$ which is locally strongly internal to $D$ is a  $D$-group. If  $D$ satisfies  ($\star_a$) then every definable group $G$ which is almost locally strongly internal to $D$ is an almost $D$-group.

    To avoid any confusion we point out that though the name may suggest it, it formally need not be the case that every $D$-group is an almost $D$-group (Example \ref{E:almost d-group not d-group}). 
\end{remark}

We can now collect our previous results and conclude: 
\begin{fact}\label{E:interp groups in dist sorts are (almost) D-groups}
    \begin{enumerate}
        \item A group of finite dp-rank which is locally (almost) strongly internal to an SW-uniformity $D$ is an (almost) $D$-group. 
        \item  Let $\CK=(K,+,\cdot,v,\dots)$ be an expansion of a valued field, $G$ an interpretable group in $\CK$ and $D$ one of the infinite distinguished sorts.
    \begin{list}{$\bullet$}{}
        \item If $\CK$ is V-minimal, $D\neq \bk$ and $G$  locally (almost)  strongly internal to $D$  then $G$ is an (almost) $D$-group.
        \item If $\CK$ is power-bounded $T$-convex and $G$ locally (almost) strongly internal to $D$ then $G$ is  an  (almost) $D$-group.
        \item If $\CK$ is $p$-adically closed and $G$ locally almost strongly internal to $D$ then $G$ is an almost $D$-group.
    \end{list}
    \end{enumerate}
   
    \end{fact}
    \begin{proof} By Fact \ref{E:dist sorts are vicinic}, all the sorts $D$ are vicinic, so we only need to check whether $D$ satisfies  ($\star$) or ($\star_a$)  (Remark \ref{remark on D-groups}).

   For  (1), When $D$ is an SW-uniformity, it satisfies both ($\star$) and ($\star_a$) by Lemma \ref{L: asi from surjective}.
     For (2), if $D$  is an SW-uniformity, then by Lemma \ref{L: asi from surjective} both  ($\star$) and ($\star_a$) hold. This covers the first two settings and the case $D=K$ in the $p$-adically closed setting. We are left with $\Gamma$ and $K/\CO$ in the $p$-adic setting.

     If $D$ is a stably embedded $\Zz$-group then by elimination of imaginaries both ($\star$) and ($\star_a$) hold, by Lemma \ref{L:asi diagaram with EI}. If $D=K/\CO$ then ($\star_a$) holds  by Lemma \ref{L:asi diagaram K/O}.
\end{proof}

We also have:
\begin{lemma}\label{L: Joint function-abstract} 
    Let $G$ be a definable (almost) $D$-group, and let $X_1\subseteq G$ be (almost) strongly internal to $D$ and $X_2\sub G $ (almost) $D$-critical, both over $A$.
    
    Assume that $(g_1,g_2)\in X_1\times X_2$ is $A$-generic and let $g=g_1g_2^{-1}$. 
    Then $X_1\cap gX_2$ is an  $Ag$-generic vicinity of $g_1$ in $X_1$. In particular, if $X_1$ is (almost) $D$-critical, then $X_1\cap gX_2$ is also (almost) $D$-critical.
\end{lemma}
\begin{proof}
    Assume for simplicity that $A=\0$. As $G$ is an (almost) $D$-group, it is easy to see that $\dpr(g_1/g_1\cdot g_2^{-1})=\dpr(X_1)$. Since $g_1\in (g_1\cdot g_2^{-1}X_2)\cap X_1$ and the intersection is $(g_1\cdot g_2^{-1})$-definable, it follows that
     $X_1\cap g_1\cdot g_2^{-1}X_2$ is a generic vicinity of $g_1$ in $X_1$.
\end{proof}

We can now deduce the main result of this section.

\begin{lemma}\label{L:existence of XX strongly-internal-abstract}
Let $G$ be a definable (almost) $D$-group, $X_1,X_2\subseteq G$  (almost) $D$-sets  over  $A$. Assume that  $(g_1,g_2)\in X_1\times X_2$ is $A$-generic. 

Then, there exists $B\supseteq A$ and $B$-definable subsets $X_i'\subseteq X_i$, such that  $X_1'\times X_2'$ is a $B$-generic vicinity of $(g_1,g_2)$ in $X_1\times X_2$  and such that  $X_1'\cdot X_2'\subseteq X_1\cdot g_2$. Moreover, $X_1'\cdot X_2'\subseteq X_1\cdot g$ for every $g\in X_2'$ and thus $X_1'\cdot X_2'$ is an (almost) $D$-set over $Bg$.
%
% The same conclusion holds if $X_1,X_2$ are $D$-sets and $D$ satisfies Axiom  (A2b) (instead of Axiom (A2)).
\end{lemma}
\begin{proof}
We prove the lemma in case the $X_i$ are almost $D$-sets (and $G$ is an almost $D$-group). The proof for the case where the $X_i$ are $D$-sets and $G$ is a $D$-group is similar. For simplicity of notation, we assume $A=\emptyset$.

Let $k=g_1\cdot g_2\in G$ and let $Y_2=\{x_2\in X_2:k\in X_1\cdot x_2\}$. Note that $Y_2$ is $k$-definable and contains $g_2$.  Since $G$ is an almost $D$-group, $\dpr(g_2/Ak)=n$, hence $\dpr(Y_2)=n$. Since  $g_1$ is inter-algebraic over $A$ with some element in $D^l$ we can apply  Lemma \ref{L:abstract general gen-os} and Remark \ref{important remark}(1) to obtain a definable subset $Y_2'\subseteq Y_2$ containing $g_2$ that is $C$-definable for some parameter set $C$ such that $\dpr(g_1,g_2/C)=\dpr(g_1,g_2)=2n$.
It follows that $\dpr(g_i,k/C)=2n$, for $i=1,2$.

We let $$Z=\bigcap_{y\in Y_2'} X_1y.$$ It is $C$-definable, containing $k$, hence $\dpr(Z)=n$. Since $g_2\in Y_2'$, we also have $Z\sub X_1g_2$. Finally, we consider
$$S=\{(x_1,x_2)\in X_1\times X_2:x_1\cdot x_2\in Z\}.$$

 It is definable over $C$, and contains $(g_1,g_2)$, thus $\dpr(S)=2n$. By Corollary \ref{C:generic inside product of min fibres}, there exists $X_1'\times X_2''\sub S$, a $B$-generic vicinity of $(g_1,g_2)$ in $X_1\times X_2$, for some $B\supseteq C$. Now let $X_2'=X_2''\cap Y_2'$; it is still a $C$-definable vicinity of $g_2$. Note that for every $g\in X_2'$, we have \[X_1'\cdot X_2'\subseteq X_1'\cdot X_2''\subseteq Z\subseteq X_1\cdot g,\] where the latter follows from the definition of $Z$ and the fact that $g\in X_2'\sub Y_2'$.
\end{proof}

\begin{remark}\label{R:variant result for X,Y, XY}
 A symmetric proof would give that we can find $X_1'$ and $X_2'$ such that $X_1'\cdot X_2'\subseteq g_1X_2$
\end{remark}

We conclude this section with a couple of examples. One corollary of Lemma \ref{L:existence of XX strongly-internal-abstract}, is that given (almost) $D$-sets $X_1,X_2\sub G$, there are $X_i'\sub X_i$, $i=1,2$, of full rank, such that $\dpr(X_1'\cdot X_2')=\dpr(X_2)$. The first example shows that this  is the best we can hope for. Namely, that even if $X_1,X_2$ are $D$-critical, the dp-rank of $X_1\cdot X_2$ may increase: 
\begin{example}
    Let $\CK$ be a real closed valued field and $G:=K\times K/\CO$. Let $X\sub G$ be the graph of the quotient map $\pi: K\to K/\CO$ and $Y\sub G$ the graph of $-\pi$. It is clear that $\dpr(X)=\dpr(Y)=1$ and that both are strongly internal to $K$, and therefore $D$-critical (see Corollary \ref{C:dp min is pure}).  It is easy to verify, using sub-additivity, that $X+Y=G$ so that $\dpr(X+Y)=2$. 
\end{example}

The next example shows that interpretable almost $D$-groups need not be $D$-groups:
\begin{example}\label{E:almost d-group not d-group}
Let $\CK$ be a sufficiently saturated elementary extension of a quadratic residual extension of $\qp$ as in Example \ref{E: crit dim is not a-crit dim} and $G=K/\CO\times G_0$, where $G_0=(K/\CO)/C_p$ the group from Example \ref{E: crit dim is not a-crit dim}. Below $D=K/\CO$.
    
    As in Example \ref{E: crit dim is not a-crit dim}, $G_0$ is locally almost  strongly internal to $D$, thus so is $G$. By Fact \ref{E:interp groups in dist sorts are (almost) D-groups}, it is an almost $D$-group, in particular,  $\dpr(g/B, g\cdot h)=\dpr(X_2)$
    for any $B$-definable almost $D$-critical $X_1,X_2$ and $B$-generic $(g,h)\in X_1\times X_2$. 

    By Example \ref{E: crit dim is not a-crit dim}, the $D$-critial rank of $G$ is $1$. To see that $G$ is not a $D$-group we consider two $D$-sets in $G$: The subgroups $X_1=K/\CO\times \{0\}$ and $X_2=\{(h,h+C_p): h\in K/\CO\}$ are in bijection with $K/\CO$, thus they are $D$-sets in $G$.
    However, the two subgroups generate $G$, so if we pick $(g,h)\in X_1\times X_2$ generic
    then $\dpr(g/g+ h)=0$.

\end{example}

\subsection{From almost strong internality to strong internality}\label{ss:from almost to strong}
In order to put the machinery for the construction of infinitesimal subgroups in gear we need to work inside a group that is locally strongly internal to a vicinic sort $D$. As we have seen in Example \ref{L: not lsi}, in the  case $D=K/\CO$ in the $p$-adic setting, it may happen that a group is locally almost strongly internal to $D$, but not locally strongly internal. In this section, we show that modding out  a finite normal subgroup resolves this problem.  The proof is inspired by a result of a similar nature due to Hrushovski and Rideau-Kikuchi, \cite[Lemma 2.25]{HrRid}.

We need (also for later use) an elementary fact from group theory:

\begin{fact}\label{F: finding groups}
Let $G$ be a group, $A,B\sub G$ arbitrary subsets, $a\in A$, $b\in B$. Assume that
\[a\cdot B=A\cdot B=A\cdot b.\]
Then there is a sub-monoid $S\le G$ (i.e. $e\in S$ and $S$ closed under multiplication) such that $A=aS$ and $B=Sb$.
In particular, if $A, B$ are finite, then $S$ is a subgroup.

\end{fact}
\begin{proof} Note that $a^{-1}A=Bb^{-1}$. We first claim that 
the set $S_1:=\{g\in G: gB\subseteq B\}$ equals $a^{-1}A$. Indeed, if $gB\subseteq B$ then $gb\in B$, so $g\in Bb^{-1}=a^{-1}A$. Conversely, since $aB=AB$ then $a^{-1}AB=B$, so the claim follows. Clearly, $e\in S_1$ and $S_1$ is closed under multiplication. By symmetry, $S_2:=\{g\in G:Ag^{-1}\sub A\}=Bb^{-1}$, is also a sub-monoid of $G$.
As we just noted, we have $S_1=S_2$, call it $S$, and $A=aS$, $B=Sb$. 

Now, if  $A$ and $B$ are finite (if one is finite then so is the other) then $S$ is a finite sub-monoid of a group, so a subgroup.
\end{proof}

\begin{lemma}\label{L:finding a group from asi}
Let $G$ be a definable group of finite dp-rank and let $D$ be any definable set. Let $X_1,X_2\subseteq G$,  $f_i:X_i\to D^k$ and $h:X_1\cdot X_2\to D^p$ be generically $m$-to-one definable functions, with $f_1,f_2, h$ defined over $A$.
Assume, moreover, that no $X_i'\sub X_i$ of the same dp-rank is $n$-internal to $D$ for any $n<m$.

Then for every  $(a,b)\in X_1\times X_2$ generic over $A$,  there is a finite subgroup $H(a,b)\subseteq G$ such that $[a]_{f_1}=a\cdot H(a,b)$ and $[b]_{f_2}=H(a,b)\cdot b$.
\end{lemma}
\begin{proof}
By the assumption on $X_1$, for any $B\supseteq A$, any $B$-definable function $F:X_1\to D^q$ (any $q$) with finite fibres is, generically, at least $m$-to-one.

Fix  $(a,b)\in X_1\times X_2$ generic over $A$. Consider the function $F:X_1\to D^{k+p}$ given by $x\mapsto (f_1(x),h(x\cdot b))$. Notice that $[x]_F=[x]_{f_1}\cap ([x\cdot b]_h\cdot b^{-1})$).

Since $[a]_F\subseteq [a]_{f_1}$, the assumption that $\dpr(a/Ab)=\dpr(X_1)$ and the minimality of $m$ forces $|[a]_F|\geq m$, so it must be that, in fact, $[a]_F= [a]_{f_1}$. It follows that $[a]_{f_1}\subseteq [a\cdot b]_h\cdot b^{-1}$. By assumption $|[a\cdot b]_h|\le m$, so necessarily $[a]_{f_1}=[a\cdot b]_h\cdot b^{-1}$,  and so $[a]_{f_1}\cdot b=[a\cdot b]_h$. If $a'\in [a]_{f_1}$ then, as $a, a'$ are interalgebraic, we have  $\dpr(a',b/A)=\dpr(a,b/A)$, so  we also get $[a']_{f_1}\cdot b=[a'\cdot b]_h$, but $[a']_{f_1}=[a]_{f_1}$ so $[a'\cdot b]_h=[a\cdot b]_h$. Since the roles of $a$ and $b$ are symmetric, we conclude also that $[a\cdot b']_h=[a\cdot b]_h$ and so $[a]_{f_1}\cdot b'=[a\cdot b]_h$ for all $b'\in [b]_{f_2}$. Therefore, $$[a]_{f_1}\cdot b=[a]_{f_1}\cdot[b]_{f_2} =a\cdot [b]_{f_2},$$ so by Fact \ref{F: finding groups} there is a finite group as claimed.
\end{proof}

For the main result of this section, we need the following facts. The first was proved in the context of SW-uniformities in  \cite[Lemma 3.14 ]{HaHaPeVF}.  It remains valid for vicinic structures as well: 
\begin{fact}\label{F:finding mutual generic}
 Let $\mathbb{U}\succ \CM$ a monster model, $b_1,\dots,b_n$ some tuples in $\mathbb{U}$. For every $\CM$-definable $X\subseteq D^r$ with finite fibres, there exists an $M$-generic $a\in X$ such that $\dpr(a,b_i/M)=\dpr(a/M)+\dpr(b_i/M)$ for all $1\leq i\leq n$.

\end{fact}
\begin{proof}
This was proved in \cite{HaHaPeVF} under the assumption that $D$ is an SW-uniformity.  However, the proof only uses the fact that, in the notation of the above statement,  there is a finite-to-one projection $\pi: X\to M^k$, with $k=\dpr(X)$, and an $M$-generic box in its image (i.e., a product of  $k$ dp-minimal subsets). In the vicinic setting, this is implied by Axiom (A1) and Corollary \ref{C:generic inside product of min fibres}.
\end{proof}

The second of the facts appearing in  \cite[Lemma 3.14]{HaHaPeVF} was proved for a single type.  The same proof works for two (or more) types:
\begin{fact}\label{F: extending to over a model}
    Let $\CM$ be a structure of finite dp-rank and $\mathbb{U}\succ \CM$ a monster model.

    For $A\sub \mathbb U$ and $a,b\in \CM^n$, there exists a small model $\CN$, $A\sub \CN\prec \CM$,   such that $\dpr(a/A)=\dpr(a/N)$ and $\dpr(b/A)=\dpr(b/N)$.
\end{fact}
\begin{proof}
     Let $\la I_t:t<k_1\ra$ be mutually indiscernible sequences over $A$ witnessing that $\dpr(a/A)\geq k_1$, i.e. each $I_t$ is not indiscernible over $Aa$, and let $\la I_t':t<k_2\ra$ be mutually indiscernible sequences over $A$ witnessing that $\dpr(b/A)\geq k_2$, i.e. each $I_t'$ is not indiscernible over $Ab$. Let $\CM'$ be some small model with $A\subseteq M'$.

     By \cite[Lemma 4.2]{SiBook}, there exists a mutually indiscernible sequence $\la J_t:t<k_1+k_2 \ra$ over $M'$ such that
     \[
     \tp(\la I_t:t<k_1\ra^\frown \la I_t':t<k_2\ra /A)=\tp(J_t:t<k_1+k_2/A).
     \]
     Let $\sigma$ be an automorphism of $\mathbb{U}$ fixing $A$ and mapping the sequence of the $J_t$ to the sequence of the $I_t$ and the $I_t'$. It follows that $\la I_t:t<k_1\ra$ are mutually indiscernible over $N:=\sigma(M')$ and each one is still not indiscernible over $Aa$ so not over $Na$ as well; likewise $\la I_t:t<k_1\ra$ are mutually indiscernible over $N$ an each one is not indiscernible over $Na$ as well.
\end{proof}

\begin{proposition}\label{P: G/H s.i.}
Let $G$ be an $A$-definable group of finite dp-rank locally almost strongly internal to a definable vicinic set $D$, and assume that $G$ is an almost $D$-group.
\begin{enumerate}
    \item If  $X\subseteq G$ is an almost  $D$-set then there exists a finite subgroup $H_X\leq G$
    and a definable subset $X'\subseteq X$, with $\dpr(X')=\dpr(X)$, such that $X'/H_X$ is strongly internal to $D$.  In particular, $G/H_X$ is locally strongly internal to $D$.
    \item The definable group $H_X$ is $A$-definable, does not depend on the choice of the almost $D$-set $X$ and is invariant under any definable group automorphism of $G$. In particular, $H$ is a normal subgroup of $G$. Set $H:=H_X$.
    \item $H$ is abelian and contained in any definable finite index subgroup of $G$. % in particular, $H$ is abelian.
    
    \item The $D$-critical rank and the almost $D$-critical rank of  $G/H$ agree and are equal to the almost $D$-critical rank of $G$.
    
    \item  $G/H$ is a $D$-group.
\end{enumerate}
\end{proposition}
\begin{proof}
(1) Let $X\subseteq G$ be an almost $D$-set, witnessed by $f:X\to D^p$ with fibres of size $m$, all definable over $B\supseteq A$. If $m=1$ then we take $H=\{e\}$, the trivial subgroup, so assume this is not the case.  %We fix a small model  $\CN$ such that $X$ and $f$ are $N$-definable.
Note that for any almost $D$-critical $X'$, any definable function $g:X'\to D^n$  (defined over arbitrary parameters) and any generic $x'\in X$ we have $|[x']_g|\geq m$, where $[x']_g=g^{-1}(g(x'))$.

To any $B$-generic  $(a,b)\in X\times X$ we associate a finite subgroup $H(a,b)\le G$, such that generic fibres of $f$ are both left and right cosets of $H(a,b)$, as follows:  By Lemma \ref{L:existence of XX strongly-internal-abstract} there is some parameter set $C\supseteq B$ and respective $C$-generic vicinities $X_1,X_2$ of $a$ and $b$ in $X$, and $c\in X_2$ such that  $X_1\cdot X_2\sub Xc$, and $\dpr(a,b/Cc)=2n$.
Let $h$ be the map on $X_1\times X_2$,  $z\mapsto f(z\cdot c^{-1})$. Then $X_1,X_2$ and $X_1\cdot X_2$  are $m$-strongly internal to $D$, witnessed by $f_1=f\restriction X_1$, $f_2=f\restriction X_2$ and $h$, respectively. By minimality of $m$, $[a]_{f_1}=[a]_f$, $[b]_{f_2}=[b]_f$ and $[a\cdot b]_{h'}=[a\cdot b]_h$, and in addition $(a,b)$ is generic in $X_1\times X_2$ over $Cc$. So Lemma \ref{L:finding a group from asi} provides us with a finite subgroup $H(a,b)\le  G$ such that $[a]_f=a\cdot H(a,b)$ and $[b]_f=H(a,b)\cdot b$.

We show that $H(a,b)$ does not depend on the choice of $(a,b)$:

\begin{claim}
For any  $(a',b')\in X^2$ generic over $B$, $H(a,b)=H(a',b')$.
\end{claim}
\begin{claimproof}
By Fact \ref{F: extending to over a model} there exists a small model $\CN\supseteq B$ such that everything we used up until now is defined over $N$, $\dpr(a,b/N)=\dpr(a,b/B)$ and $\dpr(a',b'/N)=\dpr(a',b'/B)$. Using Fact \ref{F:finding mutual generic} we can find $(c,d)\in X^2$  such that $\dpr(a,b,c,d/N)=\dpr(a',b',c,d/N)=4n$. Thus, \[H(a,b)=a^{-1}\cdot [a]_f=H(a,d)=[d]_f\cdot d^{-1}=H(a',d)=(a')^{-1}\cdot [a']_f=H(a',b'),\]
as needed.
\end{claimproof}

Let $H_X:=H(a,b)$. As, by the claim, $[a]_f=a\cdot H(a,b)=H(b,a)\cdot a$ we conclude that for all $B$-generic $a\in X$, $[a]_f=a\cdot H_X=H_X\cdot a$.  Consequently,  setting $X'=\{x\in X:[x]_f=H_X\cdot x=x\cdot H_X\}$ we have $\dpr(X')=n$.

As $X'\subseteq X \sub  G$ we may consider the image $X'/H_X$ of $X'$ under the natural quotient (viewing $G/H_X$ as a $G$-space).  Because, as we have just shown, $[x]_f=xH_X$ the function $f$ induces on $X'/H_X$ an injective function witnessing local $D$-strong internality of $G/H_X$. 

(2) Let $X$ be as in (1). An easy computation gives that for any $g\in G$,  $H_{gX}=H_X$. Moreover, note that if $\widehat X\subseteq X$ is a definable subset with $\dpr(\widehat X)=\dpr(X)$ then as it is also an almost $D$-set (Remark \ref{important remark}) and $H_X$ only depends on generic points, $H_X=H_{\widehat X}$.

Assume that $X_0\subseteq G$ is any other almost $D$-set. By  Lemma \ref{L: Joint function-abstract}, there exists $X_1\sub X_0$, $\dpr(X_1)=\dpr(X_0)=\dpr(X)$,  such that $X_1$ is contained in a translate $gX$. Thus $H_{X_0}=H_{X_1}=H_{gX}=H_X$. 
Thus $H_X$ does not depend on the almost $D$-set $X$ and on the function $f$. We denote it $H$. 

An easy computation gives that for any definable automorphism $\delta:G\to G$, $H_{\delta(X)}=\delta(H_X)$; since $H$ does depend on the almost $D$-set, we conclude that $H$ is invariant under any definable automorphism of $G$. In particular, it is normal in $G$.

To show that $H$ is $A$-definable, let $\sigma$ be an automorphism of $\CM$ fixing $A$. Then $\sigma$  fixes $G$ setwise. We now note that $\sigma(X)$ is also an almost $D$-set in $G$ and a direct application of $\sigma$ gives that $\sigma(H_X)=H_{\sigma(X)}=H_X$. We conclude that $H_X$ is $A$-definable.

(3)  Let $G_1\leq G$ be a definable finite index subgroup; obviously if $G$ is locally almost strongly internal to $D$ then so is $G_1$. Moreover, the almost $D$-rank of $G$ is equal to that of $G_1$. Thus, every almost $D$-set of $G_1$ is an almost $D$-set of $G$. As a result, the finite normal subgroup $H'\trianglelefteq G_1$ provided by (2) applied to $G_1$ is equal to $H$, the subgroup (2) provides when applied to  $G$. 

We show that $H$ is abelian. Let $G_1:=C_G(H)$; since $H$ is a  finite normal subgroup of $G$ we get that $G_1$ has finite index in $G$. So $H\subseteq C_G(G_1)\cap G_1$ and it follows that $H$ is abelian.

(4) Let $f:G\to G/H$ be the quotient map. Let $X\subseteq G$ be an almost $D$-critical set in $G$ and $X'\subseteq X$ as in (1). Let $Y\subseteq G/H$ be an almost $D$-critical set in $G/H$. As $H$ is finite, $f^{-1}(Y)$ is almost strongly internal to $D$, so $\dpr(Y)=\dpr(f^{-1}(Y))\leq \dpr(X)=\dpr(X'/H)$.

On the other hand, the definable set $X'/H$ is strongly internal, so  $\dpr(X'/H)\leq \dpr(Y)$ by the choice of $Y$.  It follows that $\dpr(X'/H)=\dpr(Y)$, so the almost $D$-critical rank of $G/H$ must equal its $D$-critical rank, and they are both equal to the almost $D$-critical rank of $G$. 

(5) To show that $G/H$ is a $D$-group  (Definition \ref{Def: D-group}): Let $X_1,X_2\sub G/H$ with $X_1$ strongly internal to $D$ and $X_2$ $D$-critical. To simplify notation, we assume that everything is definable over $\0$. Let also $(g_1,g_2)\in X_1\times X_2$ be generic. Since the quotient map $f: G\to G/H$ is a group homomorphism with finite fibres taking $g_i'\in f^{-1}(g_i)$ we get that $f^{-1}(X_i)$ is almost $D$-strongly internal (for $i=1,2$) and 
\[
    \dpr(g_1'/g_1'\cdot g_2'))=\dpr(g_1/g_1\cdot g_2). 
\]
Thus, since $G$ is an almost $D$-group, to show that $G/H$ is a $D$-group it will suffice to show that $X_2':=f^{-1}(X_2)$ is almost $D$-critical. By what we have just shown $\dpr(X_2')=\dpr(X_2)$ is the almost $D$-critical rank of $G$, and if $g: X_2\to D$ witnesses that $X_2$ is $D$-strongly internal, then $g\circ f$ witnesses that $X_2'$ is $|H|$-internal to $D$. The construction of $H$ assures that no subset of $G$ of the same rank is $m$-internal to $D$ for $m<|H|$ so that, indeed, $X_2'$ is almost $D$-critical as needed.  
\end{proof}

\begin{remark} In Fact \ref{E:interp groups in dist sorts are (almost) D-groups} we showed, in the non-$p$-adic settings, that if $G$ is locally strongly internal to a distinguished sort $D$ then it is a $D$-group. However, in the $p$-adically closed setting, even if $G$ is locally strongly internal to $K/\CO$,  we do not know whether $G$ is a $K/\CO$-group. In this case, Proposition \ref{P: G/H s.i.} guarantees the existence of a finite normal $H\sub G$ such that $G/H$ is a $K/\CO$-group whose $K/\CO$-critical rank equals the almost $K/\CO$-critical rank of $G$.
    
\end{remark}

\section{Infinitesimal groups}\label{S:infint groups}

In the present section, we develop the notion of an infinitesimal subgroup in $D$-groups 
%which are locally strongly internal to a 
for a vicinic sort $D$.  This generalises analogous results from \cite{HaHaPeVF}, proved in the topological context of SW-uniformities. In all of our settings, whenever $G$ is locally strongly internal to an unstable sort $D$, this will be  the type-definable group from Theorem \ref{intro-2}.

We first establish in Section \ref{S: Marikova} some technical results which are later used to define the infinitesimal subgroup and ensure that it has the correct properties. \\

{\bf Throughout Section \ref{S:infint groups}, $\CM$ is a $|T|^{+}$-saturated structure, $D$ a vicinic sort and $G$ an $A_0$-definable $D$-group of finite dp-rank locally strongly internal to $D$. We tacitly assume below that all parameters sets contain $A_0$.}

\subsection{Ma\v{r}\'{\i}kov\'{a}'s Method}\label{S: Marikova}

 Definable functions in the distinguished sorts in our various settings are generically well-behaved. For example, when $D$ is an SW-uniformity, definable functions are \emph{generically continuous} (i.e., continuous at all generic points), and in some contexts generically differentiable with respect to an underlying field. In the $p$-adically closed setting, when $D=K/\CO$ or  $\Gamma$, we saw that definable functions are generically given by translates of additive homomorphisms. Our aim is to make use of this generically tame behaviour in order to show that if a group $G$ is locally strongly internal to $D$ then $G$ has a (type) definable subgroup with similar properties (e.g. topological, differentiable, linear). 
 The argument which we now describe goes back to Weil's group-chunk theorem, first cast in a model theoretic setting independently by v.d. Dries \cite{vdDriesGpChunk} and Hrushovski \cite[Theorem 4.13]{PoiGroups}. In the o-minimal and $p$-adic setting (for a definable group $G\sub K^n$), this first appeared in Pillay's \cite{Pi5}, and the specific technique used below is due to Ma\v{r}\'{\i}kov\'{a}, \cite{MarikovaGps} (this was similarly  used in \cite{HaHaPeVF}).

 Below we will frequently use some earlier results, which for convenience we collect here.
 
 \begin{lemma}
     \label{R:recalling before Jana}

 Assume that $X,Y\subseteq G$ are $D$-sets over $A$ and let $(g,h)\in X\times Y$ be $A$-generic. Then
  \begin{enumerate}
    \item There exists $B'\supseteq A$ and a $D$-set $Z$ over $B'$ containing $gh$, such that  $\dpr(g,h/B')=2\dpr(X)$.
    % \item If $S\subseteq X\times Y$ is defined over $B$ and $(g,h)$ is generic in $S$ over $B$ then there is $B_1\supseteq B$ and $B_1$-generic vicinity of $(g,h)$ of the form $X_1\times Y_1 \sub X\times Y$.
    \item For $Z\subseteq G$ a D-set over $A$ and $f:X\times Y\to Z$ an $A$-definable function, if $f(g,h)$ is $A$-generic in $Z$, then for every $A$-generic vicinity  $V\subseteq Z$  of $f(g,h)$ there is $B_2\supseteq B$ and a $B_2$-generic vicinity of $(g,h)$ of the form $X_1\times Y_1\subseteq X\times Y$ such that $f(X_1\times Y_1)\subseteq V$. 
 \end{enumerate}

 \end{lemma}
 \begin{proof} By assumptions, $\dpr(X)=\dpr(Y)$, the $D$-critical rank of $G$.
 
 (1) Let $X'\sub X,Y'\sub Y$ and $B$ be as provided by Lemma \ref{L:existence of XX strongly-internal-abstract}. Namely, $X'\times Y'$ a $B$-generic vicinity of (g,h) in $X\times Y$, and $X'\cdot Y'\sub Xh'$ for all $h'\in Y'$. Choose $h'\in Y'$ with $\dpr(g,h,h'/B)=3\dpr(X)$, By sub-additivity, $\dpr(g,h/Bh')=2\dpr(X)$. The set $Z=X\cdot h'$ and the parameter set $B'=Bh'$ satisfy (1). 
 
 For (2), fix $Z$ and $f$ as above with $V\sub Z$ an $A$-generic vicinity of $f(g,h)$ in $Z$. Let $S\subseteq X\times Y$ be the $A$-definable set $\{(x,y)\in X\times Y: f(x,y)\in V\}$. Since $(g,h)\in S$, we have $\dpr(S)=2\dpr(X)$.  Corollary \ref{C:generic inside product of min fibres} yield the sets $X_1,Y_1$ as needed.
 \end{proof}

We make the following \emph{ad hoc} definition:

\begin{definition} 
Assume that $\bar a=(a_1,\ldots,a_n)\in X_1\times \cdots \times X_n$ and \[F=(F_1,\ldots,F_m):X_1\times \cdots \times X
_n \to Y_1\times \ldots \times Y_m\]  is an $A$-definable function.

We say that $\bar a$ is {\em sufficiently generic for $F$ over $A$} if each $F_i$ depends on a sub-tuple of variables $(x_{i_1},\ldots,x_{i_r})$ of $(x_1,\ldots,x_n)$ and the corresponding sub-tuple $(a_{i_1},\ldots, a_{i_r})$ of $\bar a$ is generic over $A$.
\end{definition}

For example, if $\dpr(a,b/A)=2\dpr(G)$ then the tuple $(a,b,a)$ is sufficiently generic for the map $(x,y,z)\mapsto (xy,z)$, since $F_1=xy$ depends on $(x,y)$ and $F_2=z$ depends on $z$.

The main result of this subsection is :
\begin{lemma}\label{L:Jana1}
Assume that $Y\sub G$ is a $D$-set over $A$, $d\in Y$ an $A$-generic point, and consider $F(x,y,z)=xy^{-1}z$ at $(d,d,d)$. There is $B\supseteq A$, with $\dpr(d/B)=\dpr(d/A)$,  and there are $B$-definable maps $\psi_1,\psi_2,\psi_3,\psi_4$
whose domain and range are $D$-sets over $B$, such that \[F\restriction \dom(\psi_1)=\psi_4\circ\psi_3\circ \psi_2\circ \psi_1,\] $\mathrm{Im}(\psi_i)\sub \mathrm{dom}(\psi_{i+1})$, and for every $i=0,\ldots, 3$ (with $\psi_0=id$) we have
$\psi_i\circ\cdots \circ \psi_0(d,d,d)$ sufficiently generic for $\psi_{i+1}$, over $B$.
In addition, we may choose $\dom(\psi_1)$ to be of the form $(Y_0)^3$, where $Y_0$ is a $B$-generic vicinity of $d$.
\end{lemma}
\begin{proof}
Assume that $\dpr(Y)=n$. Fix $b\in Y$ such that $\dpr(d,b/A)=2n$. We first use an auxiliary variable $w$ and write $G(w,x,y,z)=xy^{-1}z$  as a composition of the following four functions:
\[\phi_1(w,x,y,z)=(w,wx,y^{-1},z)\, ;\, \phi_2(w,x,y,z)=(w,xy,z)\] \[\phi_3(w,x,y)=(w^{-1},xy)\, ;\, \phi_4(x,y)=xy.\]
A direct computation shows that $\phi_4\circ \phi_3\circ\phi_2\circ\phi_1(w,x,y,z)=xy^{-1}z$, and we have  \[\phi_1(b,d,d,d)=(b,bd,d^{-1},d)\,;\,\phi_2\phi_1(b,d,d,d)=(b,b,d)\,;\, \phi_3\phi_2\phi_1(b,d,d,d)=(b^{-1},bd).\]

To simplify the statement, we introduce $\phi_0=id$. We now need to restrict the domain and range of the $\phi_i$ to appropriate $D$-sets so that $\mathrm{Im}(\phi_i)\sub \mathrm{dom}(\phi_{i+1})$, and for each $i=0,\ldots,3$,
$\phi_i\circ \cdots\circ \phi_0(b,a,a,a)$ is sufficiently generic for $\phi_{i+1}$, over the defining parameters.

In order to ensure that for each $i$, $\mathrm{Im}(\phi_i)\subseteq \mathrm{dom}(\phi_{i+1})$, we start from $\phi_4$ and work backwards.

\begin{itemize}
\item By Lemma \ref{R:recalling before Jana} (1), there is $B_1\supseteq A$ and a  $D$-set $Z$ over $B_1$, containing $bd$, such that $\dpr(b,d/B_1)=2n$.
 \item By Lemma \ref{R:recalling before Jana} (2), there is $B_2\supseteq B_1$ and a $B_2$-generic vicinity $Y_1^{-1}\times Z_1\sub Y^{-1}\times Z$,  of $(b^{-1},bd)$
 such that $Y_1^{-1}\cdot Z_1\sub Y$.
 \item Again, by Lemma \ref{R:recalling before Jana} (2), there is $B_3\supseteq B_2$ and $B_3$-generic vicinity $Y_2\times Y_3\sub Y\times Y$ of $(b,d)$ such that
 $Y_2\cdot Y_3 \sub Z_1$.
 \item Similarly, there is $B_4\supseteq B_3$ and a $B_4$-generic vicinity $Z_1\times Y_5^{-1}\sub Z\times Y^{-1}$ of $(bd,d^{-1})$ such that $Z_1\cdot Y_5^{-1}\sub Y_2$.
 \item Finally, there is $B\supseteq B_4$ and a $B$-generic vicinity $Y_6\times Y_7\sub Y\times Y$ of $(b,d)$ such that $Y_6\cdot Y_7\sub Z_1$.
 Furthermore, we may assume that $Y_6\sub Y_1$.

\end{itemize}

We now restrict the $\phi_i$ to the appropriate domains and obtain:
\[\phi_1:Y_6\times Y_7\times Y_5\times Y_3\to Y_6\times Z_1\times Y_5^{-1}\times Y_3\,; \, \phi_2:Y_6\times Z_1\times Y_5^{-1}\times Y_3\to Y_6\times Y_2\times Y_3,\]
\[\phi_3:Y_6\times Y_2\times Y_3 \to Y_6^{-1}\times Z_1\,;\, \phi_4: Y_6^{-1}\times Z_1\to Y.\] (for $\phi_4$, we used the fact that $Y_6\sub Y_1$).

 The appropriate tuples are sufficiently generic for the $\phi_i$, since all coordinate functions are $\emptyset$-definable, and we chose the $Z_i$ and $Y_j$, so that the points remain generic in them.

We can now write $xy^{-1}z$ as $\psi_4\circ \psi_3\circ \psi_2\circ \psi_1$, where
\[\psi_1(x,y,z)=\phi_1(b,x,y,z)\,;\, \psi_2(x,y,z)=\phi_2(b,x,y,z)\,;\, \psi_3(x,y)=\phi_3(b,x,y)\,;\, \psi_4(x)=\phi_4(b^{-1},x).\]
It is easy to verify that $(d,d,d)$ satisfies the requirements.

To obtain $\mathrm{dom}(\phi_1)$ of the form $Y_0^3$,  we may take $Y_0=Y_7\cap Y_5\cap Y_3$.    
\end{proof}

Since in the present context we do not have an underlying topology, we have to start by developing the notion of \emph{an infinitesimal vicinity}. The notation and terminology are intended to maintain the topological intuition.

\subsection{Infinitesimal vicinities}

\begin{definition} 
Let $Z\sub G$ be a $D$-set over $A$ and $d\in Z$ an $A$-generic point.
The {\em infinitesimal vicinity of $d$ in $Z$},  denoted $\nu_Z(d)$, is the partial type consisting of all $B$-generic vicinities of $d$ in $Z$, as $B$ varies over all small parameter subsets of $\CM$.
\end{definition}

We think of $\nu_Z(d)$ both as a collection of formulas and as a set of realization of the partial type in some monster model.

\begin{remark}\label{R:generic sets; enough to take containing A}
   In the definition of $\nu_Z(d)$ there is no harm in restricting to $B$-generic vicinities for $B\supseteq A$. Indeed, if $X$ is any $B$-generic vicinity of $d$ in $Z$ then by Lemma \ref{L:abstract general gen-os} there exists $B'\supseteq A$ and a $B'$-generic vicinity $X_1\subseteq X$ of $d$ such that $\dpr(d/A)=\dpr(d/B')$.
\end{remark}

By Lemma \ref{filter base} and Remark \ref{R:generic sets; enough to take containing A}, we have the following crucial property:
\begin{lemma}\label{L:generic sets form filter base} 
The collection of definable sets $\nu_Z(d)$ is a filter base, namely if $X,Y\in \nu_Z(d)$ then there exists $W\sub X\cap Y$ in $\nu_Z(d)$.
\end{lemma} 

For the following recall that by Remark \ref{important remark}(2), a cartesian product of $D$-sets is a D-set.

\begin{lemma}\label{L:product of nu}
If $Z_1,Z_2\subseteq G$ are $D$-sets over $A$ and $(d_1,d_2)\in Z_1\times Z_2$ is $A$-generic, then $\nu_{Z_1\times Z_2}(d_1,d_2)=\nu_{Z_1}(d_1)\times \nu_{Z_2}(d_2)$.
\end{lemma}
\begin{proof}
The fact that  $\nu_{Z_1}(d_1)\times \nu_{Z_2}(d_2)\vdash \nu_{Z_1\times Z_2}(d_1,d_2)$ follows from  Corollary \ref{C:generic inside product of min fibres}.

For the other direction, let $X_i\subseteq Z_i$ be $B_i$-generic vicinities of $d_i$ in $Z_i$, for $i=1,2$. By Lemma \ref{L:abstract general gen-os}, twice, we first find $B\supseteq A$ and a $B$-generic vicinity $X_1'\subseteq X_1$ of $d_1$ with $\dpr(d_1,d_2/B)=\dpr(d_1,d_2/A)$, and then $C\supseteq B$ and a $C$-generic vicinity $X_2'\subseteq X_2$ of $d_2$ with $\dpr(d_1,d_2/C)=\dpr(d_1,d_2/B)$. This $X_1'\times X_2'$ is a $C$-generic vicinity of $(d_1,d_2)$ in $Z_1\times Z_2$, so $\nu_{Z_1\times Z_2}(d_1,d_2)\vdash X_1\times X_2$, as needed.
\end{proof}

The next lemma provides a substitute for generic continuity of definable functions: 

\begin{lemma}\label{L:definable functions to generics preserve infinit types}
     Assume that $G_1,G_2$ are $D$-groups over $A$, $Z_i\subseteq G_i$ are $D$-sets for $G_i$ over $A$ ($i=1,2$),  and  $f:Z_1\to Z_2$ is an $A$-definable function. If $c$ is $A$-generic in $Z_1$ and  $f(c)$ is $A$-generic in $Z_2$ then $f(\nu_{Z_1}(c))\vdash  \nu_{Z_2}(f(c)).$
\end{lemma}
\begin{proof}
   Let $Y\subseteq Z_2$ be a $B$-generic vicinity of $f(a)$ in $Z_2$. Because $c$ is $A$-interalgebraic with some element of $D^n$ (some $n$) we apply Lemma \ref{L:abstract general gen-os} to $f(c)$ to conclude that there is $C\supseteq A$ and a $C$-generic vicinity $Y'\subseteq Y$ of $f(c)$, such that $\dpr(c,f(c)/A)=\dpr(c,f(c)/C)$. Since $f(c)\in \dcl(c)$, it follows from sub-additivity of dp-rank that:
    $\dpr(c/A)\leq \dpr(f(c)/C,c)+\dpr(c/C)\leq \dpr(c/A)$, hence $f^{-1}(Y')$ is a $C$-generic vicinity of $c$ in $Z_1$.

    It follows that $f(\nu_{Z_1}(c))\vdash  \nu_{Z_2}(f(c)).$
\end{proof}

As a final result, we show that the above definition of $\nu_D(G)$ agrees with the definition given in \cite{HaHaPeVF} when the definable set is strongly internal to an SW-uniformity.

\begin{proposition}\label{P: generic neighbourhoods in SW}
    Assume that $D$ is an SW-uniformity. Let $Z\subseteq G$ be a $D$-set over $A$ and let $g:Z\to D^m$ be a definable injection witnessing it. For any $A$-generic $d\in Z$ the partial type $\nu_Z(d)$ is logically equivalent to 
    \[
    \{g^{-1}(U):U\subseteq D^m \text{ open $M$-definable containing $g(d)$}\}.
    \]
\end{proposition}
\begin{proof}
    Let $U\subseteq D^m$ be open and definable over some $B\supseteq A$, with $g(d)\in U$. By \cite[Proposition 3.12]{HaHaPeVF} there is some open $V\sub U$, $g(d)\in V$, that is definable over $C\supseteq A$ with $\dpr(g(d)/C)=\dpr(g(d)/A)$; thus $d\in g^{-1}(V)$.
    
    For the other inclusion, let $Y\in \nu_Z(d)$ be $B$-definable for some $B\supseteq A$.  Let $X:=g(Y)\subseteq g(Z)$ be $B\supseteq A$ definable with $g(d)\in X$ and $\dpr(d/B)=\dpr(Z)$. By \cite[Corollary 4.4]{HaHaPeVF}, $g(d)$ is in the relative interior of $X$ in $g(Z)$ so there is some open $U\subseteq D^m$ definable over $B$ such that $g(d)\in U\cap X\subseteq g(Z)$. I.e., $g^{-1}(U\cap X)\in \nu_Z(d)$ and $g^{-1}(U\cap X)\sub Y$, as required.
\end{proof}

\subsection{Groups of infinitesimals vicinities}
We are finally ready to introduce infinitesimal subgroups, associated with $D$-subsets of definable groups. Recall that throughout $G$ is a $D$-group. First, we show that the infinitesimal vicinities constructed in the previous section are cosets of a type-definable subgroup:

\begin{lemma}\label{L:generic vicinities are cosets} 
If $X,Y\subseteq G$ are $D$-sets over $A$ and $(c,d)\in X\times Y$ is $A$-generic then, as partial types, 
\[c\cdot \nu_Y(d)=\nu_X(c)\cdot \nu_Y(d)=\nu_X(c)\cdot d.\]

Moreover, $\nu_X(c)$ and $\nu_Y(d)$ are left and right cosets, respectively, of the same type-definable group: $c^{-1}\nu_X(c)=\nu_Y(d)d^{-1}$.
\end{lemma}
\begin{proof}
    Let $n=\dpr(X)=\dpr(Y)$. By Lemma \ref{R:recalling before Jana},  there are $B\supseteq A$, and $X_1\times Y_1\sub X\times Y$, a $B$-generic  vicinity of $(c,d)$,
and $Z\sub G$, a  $D$-set over $B$ containing  $c\cdot d$ such that $X_1\cdot Y_1\sub Z$ and $\dpr(c,d/B)=2n$.
We fix $Z$, and by Lemma \ref{L:generic sets form filter base} we may assume that $X=X_1$, $Y=Y_1$.  For simplicity of notation, assume that $A=B=\0$.

Let us see that, as types, $c\cdot\nu_Y(d)=\nu_Z(cd)$:

The function $y\mapsto c\cdot y$ takes $Y$ into $Z$, thus, since by sub-additivity, $d$ is generic in $Y$ over $c$ and $cd$ is generic in $Z$ over $c$, it follows from Lemma \ref{L:definable functions to generics preserve infinit types} that, as types, $c\cdot \nu_Y(d)\vdash  \nu_Z(cd)$. The map $z\mapsto c^{-1}z$ takes $cY\sub Z$ into $Y$. It follows from Lemma \ref{L:definable functions to generics preserve infinit types} again that $c^{-1}\cdot \nu_{cY}(cd)\vdash \nu_Y(c)$.
However, $cY\subseteq Z$ is a $c$-generic vicinity of $cd$ in $Z$ and we have $\nu_{cY}(cd)=\nu_Z(cd)$, so $c^{-1}\nu_Z(cd)\vdash \nu_Y(d)$,
and we have equality of types $c\cdot \nu_Y(d)=\nu_Z(cd)$. Similarly, we have $\nu_X(c)\cdot d=\nu_Z(cd)$, hence we can conclude
$$c\cdot \nu_Y(d)=\nu_Z(cd)=\nu_X(c)\cdot d.$$

Now consider the definable function $(x,y)\mapsto x\cdot y$. By Lemma \ref{L:definable functions to generics preserve infinit types}, it maps $\nu_X(c)\times \nu_Y(d)$ into $\nu_Z(cd)$, i.e. $\nu_X(c)\cdot \nu_Y(d)\vdash \nu_z(cd)$. By the above we conclude that $\nu_X(c)\cdot \nu_Y(d)=\nu_Z(cd)$,  and thus
\[c\cdot \nu_Y(d)=\nu_X(c)\cdot \nu_Y(d)=\nu_X(c)\cdot d.\]

If we now realize the above two partial types in an $|\CM|^+$-elementary extension and apply  Fact \ref{F: finding groups} then we conclude that the set $S:=c^{-1}\nu_X(c)=\nu_Y(d)d^{-1}$ contains $e$ and closed under group multiplication. To see that $S$ is a subgroup, it is sufficient to show that $S$ is closed under group inverse.

For that, we apply what we already proved,  for the elements $c^{-1}$ in the set $X^{-1}$ and (again) to $d\in Y$ (noting that $(c^{-1},d)$ is generic in $X^{-1}\times Y$). We conclude that 
$c^{-1}\nu_Y(d)=\nu_{X^{-1}}(c^{-1})d$, so $S=\nu_Y(d)d^{-1}=c\nu_{X^{-1}}(c^{-1})$. It is easy to see, straight from the definition of $\nu_X(c)$, that 
\[
S^{-1}=(\nu_X(c)c^{-1})^{-1}=c\nu_{X^{-1}}(c^{-1})=S,
\] 
so $S$ is closed under group inverse and hence it is a subgroup of $G$, as needed.
\end{proof}

We can finally prove the main results concerning infinitesimal groups in $D$-sets.
\begin{proposition}\label{P: nu}
Let $G$ be a $D$-group.
% locally strongly internal to $D$.
\begin{enumerate}
\item Assume that $X\subseteq G$ is a $D$-set over $A$, then for every $A$-generic $a,b\in X$, we have
$\nu_X(a)a^{-1}=\nu_X(b)b^{-1}=a^{-1}\nu_X(a)$. Call this group $\nu_X$.

\item If $X,Y\sub G$ are $D$-sets over $A$ then $\nu_X=\nu_Y$, and we can call it $\nu=\nu_D(G)$.
\item The type $\nu$ is invariant under every automorphism of $\CM$ over $A_0$ (the parameters defining $G$) and also under every definable group automorphism of $G$. In particular, for every $g\in G(\CM)$, $g^{-1}\nu g=\nu$.

\item The dp-rank of $\nu$ equals the $D$-critical rank of $G$.
\item $\nu$ is strongly internal to $D$. 
\end{enumerate}
\end{proposition}
\begin{proof}
Let $n$ be the $D$-critical rank of $G$.

(1) By Fact \ref{F:finding mutual generic}, we find $c\in X$ such that $\dpr(a,c/A)=\dpr(b,c/A)=2n$. Applying Lemma \ref{L:generic vicinities are cosets} to $(a,b)$ and $(a,c)$ we have
$\nu_X(a)a^{-1}=c^{-1}\nu_X(c)=\nu_X(b)b^{-1}$, so the group $\nu_X(a)a^{-1}$ does not depend on the choice of an $A$-generic point in $X$. Similarly, the group $a^{-1}\nu_X(a)$ does not depend on the choice of $a$.

Applying this now to $a$ and $c$ with $\dpr(a,c/A)=2n$.  we see that  $\nu_X(a)a^{-1}=c^{-1}\nu_X(c)=a^{-1}\nu_X(a)$,
and we may now denote this group by $\nu_X$.

(2) It follows from Lemma \ref{L:generic vicinities are cosets} that if $X,Y$ are two $D$-sets then $\nu_X=\nu_Y$.

(3) Let $\sigma $ be either an automorphism of $\CM$ over $A_0$ or a definable group automorphism of $G$. Then for every $D$-set $X$, the set $\sigma(X)$ is also a $D$-set, so by (2), $\nu=\nu_X=\nu_{\sigma(X)}$. Thus, 
$$\sigma(\nu)=\sigma(\nu_X)=\nu_{\sigma(X)}=\nu,$$ where the middle equality follows from the fact that $\sigma$ is a group automorphism.

%Let $g\in G(M)$ and $X$ be any $D$-set; say over some $A$ with $A\ni g$. Let $d\in X$ be such that $\dpr(d/A)=n$, so $\dpr(g\cdot d/A)=n$ as well. The function $x\mapsto gx$ sends $X$ to $gX$, that is also a $D$-set, and by Lemma \ref{L:definable functions to generics preserve infinit types}, it sends $\nu_{X}(d)$ to $\nu_{gX}(gd)$. Hence,
%\[g\nu g^{-1}=g\nu_X(d)d^{-1}g^{-1}=\nu_{gX}(gd)(gd)^{-1}=\nu,\]
%as needed.

(4) For $X$ and $c$ as in (1), it follows by definition that $\nu_X(c)$ is a partial type consisting only of formulas of dp-rank $n$. Since $D$ is dp-minimal and dp-rank is preserved under definable finite-to-finite correspondences, it follows from \cite[Corollary 3.5]{Simdp} that $n=\dpr(\nu_X(c))=\dpr(\nu)$. 

(5) For any generic $g\in X$, $\nu_X(g)$ is strongly internal to $D$ as a subset of a $D$-set. Since $\nu$ is in definable bijection with $\nu_X(g)$, the result follows. 
\end{proof}

Note, that while Proposition \ref{P: nu}(3) shows that the partial type $\nu$ is invariant under conjugation by $g\in G(\CM)$, it does not imply that its realization in an elementary extension $\widehat \CM$ is a normal subgroup of $G(\widehat \CM)$.

\begin{corollary} \label{c: minimal}
    Let $G$ be a $D$-group. Let $\nu$ be the type definable subgroup  provided by Proposition \ref{P: nu}, $\mu$  any type definable group strongly internal to $D$ with $\dpr(\nu)=\dpr(\mu)$. Then $\nu\vdash \mu$. 
\end{corollary}
\begin{proof}
    If  $\mu\vdash Y$  then, since $\mu$ is a type-definable group, there exists a definable set $X$, $\mu \vdash X$, such that $X^{-1}X\sub Y$. By the assumption on $\mu$, we may assume that $X$ is strongly internal to $D$ and $D$-critical.  By Remark \ref{r; D-subset}, $X$ contains a definable $D$-set $X_1$.
    
Now, for every generic $c\in X_1$, we have $\nu=c^{-1}\nu_{X_1}(c)\vdash X_1^{-1}X_1\sub Y$. Thus, $\nu \vdash Y$, which implies that $\nu\vdash \mu$. \end{proof}

\subsection{The case of $D$ locally linear}

In the present section we show that if $D$ is a locally linear vicinic sort expanding an abelian group (as defined below) then the type definable subgroup of $G$ provided by Proposition \ref{P: nu} can be identified with a type definable subgroup of $\la D^r,+\ra$, for some $r$. Let us define, therefore: 

\begin{definition} 
A vicinic definable set $D$ is \emph{locally linear} if it expands an abelian group $\la D,+\ra$ and for any parameter set $A$, $A$-definable partial function $f:D^n\to D$ and $A$-generic  $a\in \dom(f)$, there exists $B\supseteq A$, a $B$-generic vicinity $X$ of $a$ in $\dom(f)$ and  a definable homomorphism $\lambda:D^n\to D$ such that $f\restriction X=(\lambda(x-a) +f(a))\restriction X$.
\end{definition}

Here are the examples of locally linear sorts which we need.
    
\begin{fact}\label{E:examples of loc linear}
    \begin{enumerate}
        \item The sort $K/\CO$, where $K$ is either a V-minimal or power bounded $T$-convex valued field, is locally linear-see \cite[Section 6.3, Proposition 6.16]{HaHaPeVF}. Recall that, in this setting, $K/\CO$ is an SW-uniformity.
        \item If $D$ is  a stably embedded $\Zz$-group, or a pure ordered vectors space over an ordered field, then it is locally linear (see \cite[Corollary 1.10]{CluHal},  and  \cite[Chapter 1, Corollary 7.6]{vdDries}). In particular, the value group $\Gamma$ is locally linear in all of the above valued fields. 
        \item The sort $K/\CO$, where $K$ is a $p$-adically closed field, is locally linear (by Corollary \ref{C: KO affine}).
    \end{enumerate}
\end{fact}

\begin{proposition}\label{P:infini_vicinity- def iso to subgroup}
    Let $D$ be a definable locally linear vicinic set and  $G$  a definable $D$-group, locally strongly internal to $D$. Then there exists a definable group isomorphism between  $\nu_D(G)\subseteq G$ and a type-definable subgroup of $(D^r,+)$ for some natural number $r$.
\end{proposition}
\begin{proof}
    Fix $X\sub G$ a $D$-set over some parameter set $A$ and $d\in X$ generic over $A$. We identify $X$ with its image under some $A$-definable injection (provided, by the assumption that $X$ is a $D$-set) and assume,  without loss of generality, that $X\sub D^r$ has minimal fibres. Write  $F(x,y,z)=xy^{-1}z$ around $(d,d,d)$ as a composition of four maps, as provided by Lemma \ref{L:Jana1}, i.e. we find  $B\supseteq A$, $Y_0\subseteq D$ a $B$-generic vicinity of $d$ and  $B$-definable maps $\psi_1,\dots, \psi_4$ such that $F\restriction (Y_0)^3=\psi_4\circ \psi_3\circ \psi_2 \circ \psi_1$ and  $\psi_i\circ\dots\circ \psi_1(d,d,d)$ is sufficiently generic for $\psi_{i+1}$, $i=0,\ldots, 3$. Local linearity of $D$ implies that $\psi_{i+1}$ coincides with a translate of a definable homomorphism, as provided by the definition, in a vicinity of $\psi_i\circ\dots\circ \psi_1(d,d,d)$.
    We can thus find a $C$-generic vicinity $Y_1\subseteq Y_0$  of $d$ such that $F$ takes values in $X$ and equals on $Y_1^3$ to $\lambda(x-d,y-d,z-d)+d$ for some definable homomorphism $\lambda: D^{3r}\to D^r$. 
    
    Since $d=F(x,d,d)=\lambda(x-d,0,0)+d$, we have $\lambda(x-d,0,0)=x-d$ for any $x\in Y_1$. Similarly, $\lambda(0,0,z-d)=z-d$ for any $z\in Y_1$. Since $F(x,x,d)=d$ it follows that $\lambda(0,y-d,0)=d-y$ for any $y\in Y_1$. As $\lambda$ is additive we conclude that $\lambda(x-d,y-d,z-d)=x-y+z-d$ on $Y_1$; thus
    \[F(x,y,z)=xy^{-1}z=(x-d)-(y-d)+(z-d)+d=x-y+z.\]
    
    Since, by Proposition \ref{P: nu}, $\nu_X(d)$ is a coset of a subgroup of $G$, then it is closed under the function $xy^{-1}z$; therefore it is also closed under $x-y+z$. It follows that $\nu_X(d)$ is also a coset of a subgroup of $(D^r,+)$. The function $x\mapsto (x\cdot_G d)-_D d$. is an isomorphism of the groups $\nu_X(d)\cdot d^{-1}$ and $\nu_X(d)-d$. Hence, $\nu_D(G)$ is isomorphic to a type definable subgroup of $\la D^r,+\ra$.
\end{proof}

A similar proof gives the following stronger result: 
%The assumptions of it holds in $K/\CO$ for $(K,v)$ the various valued fields we will encounter. 
\begin{proposition}\label{P:def group into K/O}
    Let $\CK=(K,v,\dots)$ be an expansion of valued fields that is either V-minimal, power-bounded T-convex or $p$-adically closed  and set $D=K/\CO$. Let $G$ be a definable $D$-group.
    %locally strongly internal to $D$, say witnessed by a $D$-set $Y\subseteq G$. 
    Then there is a definable subgroup $G_1\le G$, $\nu_{K/\CO}\vdash G_1$, that is  $D$-critical and 
    definably isomorphic  to a definable subgroup of $\la (K/\CO)^r,+\ra$, for some natural number $r$.     
\end{proposition}
\begin{proof}
    By Fact \ref{E:examples of loc linear}, $D$ is a locally linear vicinic sort. 
    We start as in Proposition \ref{P:infini_vicinity- def iso to subgroup}: fix some $B$-generic $d$ is some $B$-definable $D$-set $Y_0\subseteq G$ such that after identifying $Y_0$ with a subset of $D^r$, $xy^{-1}z$ coincides with $x-y+z-d$ on $(Y_0)^3$. We identify $Y_0$ with a subset of $D^r$

    We  claim that $Y_0$ may be taken to be  of the form $d+H$ for some definable subgroup $H$ of $(K/\CO)^r$. If $K$ is $p$-adically closed, this is Lemma \ref{L:sets are locally cosets in K/O}. When $D$ is an SW-uniformity (i.e. when $\CK$ is V-minimal or power-bounded $T$-convex) we proceed as follows: 
    
    By Fact \ref{F: extending to over a model}, there it exists $\CK_0\prec\CK$, containing $C$, with $\dpr(d/K_0)=\dpr(d/C)$. By \cite[Proposition 4.6]{SimWal} (and using Lemma \ref{L:Jana1} again) we may further assume that $r=\dpr(d/K_0)$.  By \cite[Proposition 3.12]{HaHaPeVF}, by passing to a definable subset, we may assume that $Y_0$ is a $C'$-generic vicinity of $d$ in $D^r$ for some $C'\supseteq \CK_0$. Since $d$ is in the interior of $Y_0$ there exists a ball around $0$, $B_0\subseteq D^r$ such that $d+B_0\subseteq Y_0$; so take $H=B_0$ (and thus $Y_0=d+H$).
    
    As a result $(Y_0)^3$ is closed under the function $x-y+z$ hence also under the function $xy^{-1}z$. It follows that $Y_0d^{-1}$ is a subgroup, $G_1\le G$, and hence $Y_0=G_1d=H+d$, and the function $x\mapsto (x\cdot_G d)-_D d$ is a group isomorphism between $G_1$ and $H$.

    By Corollary \ref{c: minimal}, we have $\nu_{K/\CO}\vdash G_1$
\end{proof}

\subsection{The case of $D$ an SW-uniformity}    
The aim of the present section is to show that when $D$ is an SW-uniformity and $G$ is a $D$-group we can endow $\nu_D(G)$ with a (definable) topology, which -- in turn -- allows us to topologise $G$. 

Although some of the results in this section can probably be proved in a higher level of generality, we restrict to the case where $D$ is an SW-uniformity. Recall that  every SW-uniformity is vicinic and that every definable group  locally (almost) strongly internal to an SW-uniformity $D$ is an (almost) $D$-group; consequently, all results on $D$-groups proved up to this point may be applied in the present setting. 

Let $D$ be an SW-uniformity, definable in $\CM$. By Proposition \ref{P: generic neighbourhoods in SW}, we may use the definition of the infinitesimal group $\nu$ as given in \cite{HaHaPeVF}, i.e., the intersection of all definable open neighbourhoods of $0$. 

\begin{proposition}\label{P:nu-topological}
Let $G$ be a definable group, $Y\sub G$ a $D$-critical subset, witnessed by a function $f: Y\to D^n$, all definable over some $A$.  Fix some $A$-generic $c\in Y$ and $\widehat \CM \succ\CM$, an $|M|^+$-saturated model.
    \begin{enumerate}
        \item The weak topology induced by $x\mapsto f(x\cdot c)\in D^n$ on $\nu(\widehat \CM)$ turns it into a topological group.
        \item The topology on $\nu$ obtained in the previous clause does not depend on the choice of the $D$-critical set $Y$, the function $f$, or the choice of the point $c$. 
        % For any $D$-critical set $Y\subseteq G$ over $\CN$, definable injection $f:Y\to D^n$ and $c\in Y$ with $\dpr(c/N)=\dpr(Y)$, the topology induced by $x\mapsto f(x+c)$ on $\nu(\widehat \CM)=(\nu_Y(c)c^{-1})(\widehat \CM)$ is the same.
         \item 
         \begin{enumerate}
             \item[(i)] For every $g\in G(\widehat \CM)$ there is an open $V\sub \nu(\widehat \CM)$ such that $V\sub \nu(\widehat \CM)^g\cap \nu(\widehat \CM)$.
            \item[(ii)] For every $g\in G(\widehat \CM)$, the function $x\mapsto x^g=gxg^{-1}$, from $\nu(\widehat \CM)$ into $\nu(\widehat \CM)^g$, is continuous at $e$ with respect to topology on $\nu(\widehat \CM)$.    
        \end{enumerate}
    \end{enumerate}
\end{proposition}

\begin{proof}

    (1) By Fact \ref{F: extending to over a model}, we may assume that $A$ is a model so by \cite[Proposition 4.6]{SimWal} we may further assume that $\dpr(Y)=n$ and that $Y\subseteq D^n$ is open. Since in SW-uniformities every definable function is generically continuous (\cite[Proposition 3.7]{SimWal}), and continuity is preserved under composition. We now use Lemma \ref{P:infini_vicinity- def iso to subgroup}, as in Proposition \ref{P:infini_vicinity- def iso to subgroup}, to find $C\supseteq A$ with $\dpr(c/C)=\dpr(c/A)$ and a $C$-definable open subset $Y_0\subseteq Y$ containing $c$ such that $F(x,y,z)=xy^{-1}z$ takes values in $Y$ and is continuous on $(Y_0)^3$. We have $\nu=\nu_Y(c)c^{-1}$, and thus the pullback, under the map $x\mapsto x\cdot c$, of the topology on $Y$ endows $\nu$ with a topology. It is a group topology, since $(x,y)\mapsto xc^{-1}yc^{-1}c=xc^{-1}y$ is continuous and so is $x\mapsto (xc^{-1})^{-1}c=cx^{-1}c$.

(2) The injection $f:Y\to D^n$ endows $Y$ with a definable topology and by \cite[Lemma 4.6]{HaHaPeVF}, this topology at every $A$-generic point of $Y$  does not depend on $f$, call it $\tau_Y$.
Thus, we may assume that $Y\sub D^n$.

Now, given an $A$-generic $c$ in $Y$, the above construction endows $\nu(\widehat \CM)$ with a group topology, call it  $\tau_{Y,c}$, for which, by definition,  the map $x\mapsto xc$ is a homeomorphism between $(\nu(\widehat \CM),\tau_{Y,c})$ and $(\nu_Y(c)(\widehat \CM),\tau_Y)$. 

Similarly, if $X$ is any other $D$-critical set, with $d$ an $A$-generic in $X$ then the map $x\mapsto xd$ endows $\nu(\widehat \CM)$ with a group topology $\tau_{X,d}$ which is homeomorphic, via $x\mapsto xd$, to $(\nu_X(d)(\widehat \CM),\tau_X)$. By replacing $A$ by a small model $\CN$ containing it (see \cite[Lemma A.1]{HaHaPeVF}), and applying  Fact \ref{F:finding mutual generic},  we may assume that 
$\dpr(c,d/N)=\dpr(X\times Y)$.

Thus, the  map $x\mapsto x\cdot (c^{-1}d)$ is a bijection of $\nu_Y(d)$ and $\nu_X(c)$, defined over $Nc^{-1}d$. Since $c$ is generic in $X$ over $Nc^{-1}d$ (by Lemma \ref{C:g/gh - abstract}), the map is  a homeomorphism of $(\nu_Y(c)(\widehat \CM),\tau_Y)$ and $(\nu_X(d)(\widehat \CM),\tau_X)$.
This shows that $\tau_{Y,c}=\tau_{X,d}$.

    (3) Let $Y$ be a $D$-critical set over $A$, witnessed by some $f:Y\to D^n$ and assume $e\in Y$. Recall that \[\nu=\{V: V\in \tau_{Y,f} \text{ with } e\in V\}.\]  By Proposition \ref{P: nu}(3), $\nu$ is invariant under conjugation by $g\in G(\CM)$, thus  by compactness, for every $U\in \tau_{Y,f}$ with $e\in U$ there exists $V\in \tau_{Y,f}$ with $e\in V$ such that $gVg^{-1}\subseteq U$. Since this statement is first order, it  also holds  of  any $g\in G(\widehat M)$ and $U$ an $\widehat{M}$-definable open basic open set containing $0$. This gives (i) and (ii). 
\end{proof}

Using these results, we may endow $G$ with a  definable group topology.

\begin{corollary}\label{C: G locally SW impplies SW}
Let $G$ be a definable group locally strongly internal to $D$.
\begin{enumerate}
    \item The group $G$ has a definable basis for a topology, making $G$ a non-discrete Hausdorff topological group.
    \item If $X\sub G$ is a $D$-critical set then $Int(X)\neq \emptyset$.
    \item If $G$ is dp-minimal then it is an SW-uniformity. 
\end{enumerate}
\end{corollary}
\begin{proof} 

    (1) Apply Proposition \ref{P:nu-topological}(3) with Lemma \ref{L: top from H to G} to conclude that $G(\widehat \CM)$ has a uniformly definable basis of neighbourhoods of the identity $e\in G(\widehat \CM)$, for $\widehat \CM\succ \CM$ an $|M|^+$-saturated extension. Since this is first order and $G$ is definable, the same holds for $G=G(M)$. Since $D$ is an SW-uniformity it has no infinite definable discrete sets, so -- in particular -- the topology on $\nu$, and therefore also on $G$ is non-discrete.

(2) If $X$ is $D$-critical and $c$ is generic in $X$, then by construction $\nu_X(c)$ is open with respect to the group topology, so in particular $X$ has non-empty interior.
    
    (3) The group topology gives rise to a uniformity. Since the topology is a non-discrete group topology, it has no isolated point. Finally, every definable infinite subset of $G$ is necessarily $D$-critical,
    so by (2) has non-empty interior.

   % Let $S\subseteq  G$ be a definable infinite set and let $Y\subseteq G$ be $D$-critical set. By Lemma \ref{L: Joint function-abstract}, there exists $g\in G$ with $S\cap gY$ infinite. Since $gY$ is strongly internal to $D$ and $D$ is an SW-uniformity, there exists a definable open subset of $S\cap gY\subseteq S$, as required. 
  \end{proof}

\section{The case of a strongly minimal D}\label{S: lsi to k}

 We consider in this section the case where some infinite definable subset of a definable group $G$ is almost strongly internal to a strongly minimal set. This will be the case in the V-minimal setting, when $G$ is locally almost strongly internal to $\bk$. In such situations, techniques from the theory of $\omega$-stable groups can be applied, and the results can be formulated in greater generality. \\  

\textbf{For this section, we fix a sufficiently saturated (possibly multi-sorted) structure $\CM$ and a definable strongly minimal set $F$. We assume further that $F$ is stably embedded and eliminates imaginaries.}
\\

We observe:
\begin{lemma}\label{C:cor for asi to k}

      Assume that $G$ is a definable group in $\CM$ of finite dp-rank. %For any definable sets $X_1,X_2\subseteq G$, (almost) strongly internal to $F$, so is $X_1\cdot X_2$.
 If $X_1, X_2 \sub G$ are (almost) $F$-critical, then so is $X_1\cdot X_2$.
%  \end{enumerate}
\end{lemma}
\begin{proof}

Since $X_1$ and $X_2$ are both (almost) strongly internal to $F$, so is $X_1\times X_2$, and because  $F$ eliminates imaginaries, it follows from Lemma \ref{L:asi diagaram with EI}, that $X_1\cdot X_2$ is also (almost) $F$-strongly internal. Since $\dpr(X_1\cdot X_2)\ge \dpr(X_1)$, if $X_1, X_2$ are (almost) $F$-critical, then equality of the ranks follows.
\end{proof}

Recall that in a any strongly minimal definable set, the dp-rank and Morley rank coincide. Furthermore, in the following we will repeatedly use the fact that if $f:X\to Y$ is a finite-to-one definable functions then $\mr(f(X))=\mr(X)$. A definable group is called {\em definably connected} if it has no definable subgroups of finite index.

\begin{proposition}\label{P: si to sm}
    Let $G$ be a definable group of finite dp-rank in $\CM$, locally almost strongly internal to $F$. Then there exist a definable, definably connected normal subgroup $H\trianglelefteq G$ and a definable finite normal $H_0\trianglelefteq G$, with $H_0\subseteq H$, such that
    \begin{enumerate}
                \item $H$ is almost $F$-critical.
        \item $H/H_0$ is strongly internal to $F$.
    \end{enumerate}
    Moreover, if $G$ is locally strongly internal to $F$ then one can take $H_0=\{e\}$.
\end{proposition}
\begin{proof}
    Let $Y\subseteq G$ be almost $F$-critical with $\dpr(Y)=n$. Because $F$ is strongly minimal, this implies that $\mr(Y)=n$ and by replacing $Y$ with a subset, we may assume that its Morley degree is $1$.
Note that the Morley rank of subsets of $Y$ is definable in parameters, since $Y$ is almost strongly internal to a strongly minimal definable set and the Morley Rank is definable in strongly minimal sets.

The construction of the group $H$ and the verification of its desired properties is now an adaptation of well known arguments from the theory of $\omega$-stable groups to the present context : Define a relation $E$ on $G$:
\[g\mathrel{E} h\Longleftrightarrow \mr(Yg\cap Yh)=n.\]
Notice that $\mr(Yg\cap Yh)=\mr(Y\cap Yhg^{-1})$, and hence it follows from the definability of $\mr$ that  $E$ is definable (one can also use the fact that the generic type of $Y$ is definable).
Since $\dM(Y)=1$, $E$ is an equivalence relation on $G$, which is moreover right-invariant under $G$. Furthermore, the $E$-class of the identity element $e$, call it $H$, is closed under group inverse. Thus, $H$ is a definable subgroup of $G$. In model theoretic terminology, $H$ is the generic stabiliser of $Y$ in $G$. Notice that if $Yg\cap Y\neq \0$ then $g\in Y^{-1}Y$, hence $H\sub Y^{-1}Y$.

\begin{claim} 
$H$ is almost strongly internal to $F$, $\mr(H)=n$ and $H$ is connected and normal in $G$.
\end{claim}
\begin{claimproof}
    By Lemma \ref{C:cor for asi to k}, $Y^{-1}Y$ is almost strongly internal to $F$ so has finite Morley rank. By the maximality assumption on $\mr(Y)$, we have $\mr(Y^{-1}Y)=n$ (it is clear that $n=\mr(Y)\leq \mr(Y^{-1}Y)$), hence $\mr(H)\leq n$. We similarly have $\mr(Y Y)=n$.

To see that $\mr(H)=n$, we let $k=\dM(YY)$. Notice first that $g_1 E  g_2$ if and only if $Hg_1=Hg_2$. Now,
for $g_1, g_2\in Y$, we have $\mr(Yg_1)=\mr(Yg_2)=n$.  If $\neg(g_1\mathrel{E} g_2)$, then $\mr(Yg_1\cap Yg_2)<n$, but since
the Morley degree of $YY$ is $k$, there can be at most $k$ different $E$-classes intersecting $Y$, so $Y$ is covered by at most $k$-many right cosets of $H$. It follows that there is some $g_0\in Y$ such that $\mr(Hg_0\cap Y)=n$. In particular, $\mr(H)=n$. The group $H$ is connected because $\dM(Y)=1$. Indeed, if $H_1$ were a subgroup of $H$ finite index then some cosets of $H_1$ would have contained $Yg_0^{-1}$ (up to a set of smaller rank), contradicting the definition of $H$.

It is left to see that $H$ is normal in $G$. Indeed, if $H^g\neq H$ for some $g\in G$ then, since $H$ is connected, it follows that $[H:H^g\cap H]$ is infinite, and therefore the set $H^gH$, which by Lemma \ref{C:cor for asi to k} is almost strongly internal to $F$,  contains infinitely many pairwise disjoint right cosets of $H^g$. It follows that $\mr(H^gH)>n$, contradicting the maximality of $n$.
\end{claimproof}

%(note that $H^gH$ is again almost strongly internal to $F$ and thus of finite Morley rank). \qed

It is left to find the desired finite $H_0$.
We have so far shown that $H$ is  almost strongly internal to a stably embedded set $F$, thus by the work of Hrushovski and Rideau-Kikuchi, \cite[Lemma 2.25]{HrRid}, there exists a finite normal subgroup $H_0\leq H$ such that $H/H_0$ is internal to $F$. As $F$ eliminates imaginaries, $H/H_0$ is strongly internal to $F$. \\

We conclude with the following:

\begin{claim}\label{c:finite normal}
    There exists a definable group $H_0$ {\bf normal  in $G$} such that $H_0\sub H$ and
    $H/H_0$ is strongly internal to $F$.
\end{claim}
\begin{claimproof}
    Amongst all finite definable normal subgroups of $H$ with $H/H_0$ strongly internal to $F$, choose $H_0$  of minimal cardinality and let $f:H\to H/H_0$ be the quotient map. To see that $H_0$ is normal, let $g\in G$ and assume that $H_0^g\neq H_0$. So, for $x\in H$, $x\mapsto (xH_0,xH_0^g)\in H/H_0\times H/H_0^g$ is a group homomorphism with kernel $H_0\cap H_0^g$. Thus, $H/(H_0\cap H_0^g)$ is strongly internal to $F$ contradicting the minimality assumption. 
\end{claimproof}

For the last part, if $G$ is locally strongly internal to $F$ then so is $Y$ and consequently so is $H$ and no $H_0$ is needed.
\end{proof}

\section{The main theorems}\label{S: final}
In this section, we apply the results obtained in the previous sections to study groups interpretable in some dp-minimal valued fields. 
%\subsection{Definable Groups in dp-minimal fields of characteristic $0$}
We start with a lemma on definable groups in a slightly more general context: we show that groups \emph{definable} in a dp-minimal field $\CK$ of characteristic 0 have unbounded exponent, under the additional assumption that definable functions are generically differentiable (e.g., if $\CK$ is 1-h-minimal). \\

The idea for the proof of the next lemma is due to S. Starchenko. Recall that a type-definable group is torsion free if its set of realisations in any saturated enough model is torsion free. 
\begin{lemma}\label{L: no tor}
     Let $\CK$ be a sufficiently saturated  SW-uniform structure expanding a field of characteristic $0$ with generic differentiability. 
     \begin{enumerate}
         \item If $G$ is an infinite interpretable group locally  strongly internal to $D=\CK$ then $G$ is a $D$-group and the associated type-definable subgroup $\nu_D(G)$ is torsion-free.
         \item If $G$ is locally almost strongly internal to $K$ then $G$ has unbounded exponent.
     \end{enumerate}  
     
   %  In particular, if $G$ is dp-minimal then it is abelian-by-finite.
\end{lemma}
\begin{proof}
    Since $K$ is an SW-uniformity it is vicinic (Fact \ref{E:dist sorts are vicinic}), and thus if $G$ is almost locally strongly internal to $D=K$ it is an almost $D$-group (Lemma \ref{L: asi from surjective} and Corollary \ref{C:g/gh - abstract}). By Proposition \ref{P: G/H s.i.} there is a finite normal subgroup $H\leq G$ such that $G/H$ is locally strongly internal to $K$. If $G/H$ has unbounded exponent, then so does $G$. So it suffices to prove the former. Hence, replacing $G$ with,  $G/H$ Clause (2) follows from Clause (1). 
    
    For (1), since $K$ is an SW-uniformity and $G$ is locally strongly internal to $K$, it is a $D$-group (Fact \ref{E:interp groups in dist sorts are (almost) D-groups}(1)).  Let $\nu:=\nu_K(G)$ as provided by Proposition \ref{P: nu},   $\widehat\CK\succ \CK$  some $|\CK|^+$-saturated extension. As in \cite[Proposition 4.19]{HaHaPeVF}, we endow $\nu(\widehat \CK)$ with a differential group structure that we identify with a type-definable group in $(\widehat K)^m$, for some integer $m$.

    Let $\delta(x,y)$ be the multiplication map on $\nu:=\nu(\widehat \CK)$. It follows easily from the chain rule that  the  differential of $\delta$ at $(e,e)$, as a map
    from the tangent space at $(e,e)$ of $\nu\times \nu$ into $T_e(\nu)$, is $x+y$. Thus, for  any  $n\in \mathbb N$, the $K$-differential of the $n$-fold multiplication $\delta_n(x)$ mapping $x$ to $\delta(x,\delta(x,\delta(\dots\delta(x,x)\dots)$ equals $n \id$, where $\id$ is the identity (matrix or differential). Since  $\mathrm{char}(K)=0$, $n \id$ is an invertible linear map. 
    By the very definition of the derivative, this means that there exists an open neighbourhood $U$ of $e$, such that
    $\delta_n(x)=x^n\neq e$ for all $x\in U$.
    Since it is a first-order property, we can take $U$ to be $\CK$-definable.
     As $\nu(G)$ is the intersection of all $\CK$-definable open neighbourhoods of $e$,  $x^n\neq e$ for all $n$ and any $x\in \nu(\widehat \CK)$, as claimed. 
\end{proof}

We shall need the following results.

\begin{fact}\label{F:dp-min unbounded is abelian-by-finite}
Assume that $\CM$ is a structure such that every dp-minimal group interpretable in $\CM$ has unbounded exponent, then every such group is abelian by finite. 
\end{fact}
\begin{proof}
By \cite[Proposition 3.1]{SimDPMinOrd} a dp-minimal group $G$ has a normal abelian subgroup $H$ such that $G/H$ has bounded exponent. Our assumption implies that $G/H$ is finite, with the desired conclusion.  
\end{proof}

\subsection{Power bounded T-convex valued fields}
Let $\CK$ be a T-convex power bounded structure. By naming one constant $c\notin \CO$, we may assume that $\CK$ has definable Skolem functions. We assume that $\CK$ is  sufficiently saturated.\\

The next Lemma was proved in \cite{HaHaPeVF} under the additional assumption that the infinite set below carried a field structure. 
\begin{lemma}\label{L:l.s.i to one dist sorts in tconvex}
     Every infinite  set interpretable in $\CK$ is locally strongly internal to one of the distinguished sorts.
\end{lemma}
\begin{proof}
    Let $Y=X/E$ be an infinite interpretable set, where $X\subseteq K^n$.  By \cite[Lemma 5.10 and Proposition 5.5]{HaHaPeVF} there exists an infinite definable subset $S\subseteq Y$ in definable bijection with an infinite definable subset of $K/E'$, and a definable finite-to-finite correspondence between $S$ and one of $K$, $K/\CO$, $\bk$ or $\Gamma$.  Since $K$ is  weakly o-minimal, each $E'$-equivalence class is a finite union of convex sets. By replacing $E'$ with the equivalence relation $E''$, choosing the first component in each $E'$-class, we may assume that $E'$ is a convex equivalence relation. Thus $S$ is linearly ordered (and weakly o-minimal), and in particular it eliminates finite imaginaries in the sense of \cite[Section 4.7]{HaHaPeVF}. As each of the sorts $K$, $K/\CO$, $\bk$ or $\Gamma$ is an SW-uniformity, the result follows by \cite[Lemma 4.28]{HaHaPeVF}.
\end{proof}
We can now state the main theorem in the $T$-convex power bounded case.

\begin{theorem}\label{T:groups in tconvex}
    Let $G$ be an infinite group interpretable in $\CK$. Then $G$  is locally strongly internal to at least one distinguished sort $D$, and  for every  such $D$, there exists a type-definable subgroup $\nu_D$ of $G$, with the following properties:   
    \begin{enumerate}
         \item If $D=K$ or $D=\bk$ then $\nu_D$ is definably isomorphic to a type-definable group $\nu'\vdash K^r$ or $\nu'\vdash \bk^r$, respectively, for some $r$. 
        \item If $D=\Gamma$ then $\nu_D$ is definably isomorphic to  a type-definable subgroup $\nu'$ of $\la \Gamma^r,+\ra$, for some $r$.
        \item If $G=K/\CO$ then there exists an infinite definable subgroup $H$ of $G$,  such that $\nu_D \vdash H\le G$, and $H$ is definably isomorphic to a subgroup of $\la (K/\CO)^r,+\ra$ for some $r$.
\item \begin{enumerate}
    \item $\dpr(\nu_D)=n$ the $D$-critical rank of $G$, and the group is minimal among all type-definable subgroups of $G$ of dp-rank $n$ that are
    %$\nu_D$ is smallest among all type-definable subgroups of $G$ of dp-rank $n$ and 
    strongly internal to $D$. 
        \item $\nu_D$ is torsion free, so in particular $G$ has unbounded exponent. 
    
\end{enumerate}
        \end{enumerate}    
    Moreover, if $G$ is dp-minimal then $G$ is abelian-by-finite.
\end{theorem}

 \begin{proof}
    Let $G$ be in an infinite interpretable group. By Lemma \ref{L:l.s.i to one dist sorts in tconvex}, $G$ is locally strongly internal to one of the distinguished sorts $D$. By Fact \ref{E:dist sorts are vicinic}, every such $D$ is vicinic and by Fact \ref{E:interp groups in dist sorts are (almost) D-groups}, $G$ is  a $D$-group. By Proposition \ref{P: nu}(5), $\nu_D(G)$ is strongly internal to $D$.
    
    If $D=K$ or $D=\bk$ then let $\nu:=\nu_K(G)$ (or $\nu_\bk (G)$) and (1) follows. By Lemma \ref{L: no tor}(1), $\nu_D(G)$ is torsion free. 
    
     If $D=\Gamma$ then $D$ is locally linear, by Fact \ref{E:examples of loc linear} (2) and \cite[Theorem B]{vdDries-Tconvex}. By Proposition \ref{P:infini_vicinity- def iso to subgroup}, $\nu_\Gamma(G)$ is definably isomorphic to a type-definable subgroup of $\la \Gamma^r,+\ra$, for some $r$. In this case $\nu_\Gamma$ is torsion free since $\la \Gamma^r,+\ra$ is.

    If $D=K/\CO$ then it is locally linear by  Fact \ref{E:examples of loc linear} (1).  Proposition \ref{P:def group into K/O} provides us with a definable subgroup $H$ satisfying the requirements. In this case $\nu_{K/\CO}(G)$ is torsion free since $\la (K/\CO)^r,+\ra$ is. This ends the proof of (1)-(3).

    4(a).  By Proposition \ref{P: nu}(4), $\dpr(\nu_D)=n$. The minimality of $\nu_D$  is Corollary \ref{c: minimal}.

   4(b). We showed in all cases that $\nu_D$ is torsion-free. It follows that $G$ must have unbounded exponent, for if not then for some $m\in \mathbb N$, $x^m=e$ for all $x\in G$, and the same is true in every elementary extension $\widehat \CK$. This contradicts the fact $\nu_D(\widehat \CK)$ is infinite and torsion-free. 
   
   Since $G$ was an arbitrary interpretable group,  we may apply Fact \ref{F:dp-min unbounded is abelian-by-finite} to conclude that  every dp-minimal interpretable group is abelian-by-finite. 
    %That exactly one of the possibilities (1)-(4) must hold if $G$ is dp-minimal follows from  Corollary \ref{C:dp min is pure}. \textcolor{red}{A: We don't claim uniqueness at this stage.}
\end{proof}

\begin{remark}
    In \cite[Lemma 2.16]{MelRCVFEOI}, Mellor proves that for weakly o-minimal
    structures  $\dcl^{eq}$ agrees with $\acl^{eq}$. It follows that if $f:X\to Y$ is a definable finite-to-one function, then there exists a definable $X_1\sub X$ with $f(X_1)=f(X)$ and $f\restriction X_1$ is injective. Thus, if $\CK$ is power bounded $T$-convex and $D$ a distinguished sort then the $D$-critical rank of $G$ equals its almost $D$-critical rank.   Consequently, Clause 4(a) of the above theorem could be restated with the almost $D$-critical rank of $G$. 
\end{remark}

\subsection{V-minimal valued fields}
Let $\CK=(K,v,\dots)$ be a sufficiently saturated V-minimal valued field.
We need the following result, appearing implicitly in \cite{HaHaPeVF}:

\begin{lemma}\label{L:from correspondece to lasi in vminimal}
     Assume that $X$ and $Y$ are infinite sets definable in some $|T|^
+$-saturated (multi-sorted) structure
$\CM$ and $C\sub X\times Y$  a definable  finite-to-finite definable correspondence.
%projecting onto both  $X$ and $Y$.
If $Y$ is either a field or supports an SW-uniform structure, then $X$ is locally almost strongly internal to $Y$. %(in fact into $Y^r$ with $r=1$).
\end{lemma}
\begin{proof}
    It is not hard to show (see the first part of the proof of \cite[Lemma 4.28]{HaHaPeVF} for the details) that if $Y$ eliminates finite imaginaries (in the sense of \cite[Section 4.7]{HaHaPeVF}) then any finite-to-finite correspondence $C\sub X\times Y$ gives rise to a definable finite-to-one function $f:X'\to Y$ for some infinite definable $X'\sub X$. This gives the desired conclusion if $Y$ is a field. 
    
    Assume now that $Y$ supports an SW-uniform structure. 
    %Restricting $X$ and $Y$ we may assume that there is  $C\sub X\times Y$, a definable finite-to-finite correspondence between $X$ and $Y$.  
    For  $x\in X$ and $y\in Y$ denote
\[C_x=\{y\in Y:(x,y)\in C\},\, C^y=\{x\in X:(x,y)\in C\}\] and note that by $\aleph_0$-saturation, $|C_x|$  is uniformly bounded.

Assume everything is definable over some parameters set $A$ and let $d\in Y$ with $\dpr(d/A)=\dpr(Y)=1$. Since $C^d$ is finite, so is $(C^d)^{-1}$, where $(C^y)^{-1}:=\{y\in Y: y\in C_x,\, x\in C^y\}$.

Since the topology on the SW-uniformity is Hausdorff, we can find a relatively open subset of $Y$, $U\ni d$ with $(C^d)^{-1}\cap U=\{d\}$. By \cite[Proposition 3.12]{HaHaPeVF} there is some $B\supseteq A$ and a $B$-definable open neighbourhood $U_0\subseteq U$ with $d\in U_0$ and $\dpr(d/B)=\dpr(d/A)=1$.

Let $Y'=\{y\in U_0: |(C^y)^{-1}\cap U_0|=\{y\}$\}; it is $C$-definable and as $d\in Y'$, it has dp-rank $1$.

Consider $C'=\{(x,y)\in X\times Y':(x,y)\in C\}$. It is still a finite-to-finite correspondence and we claim that for any $x\in X$, $|C'_x|=1$. Indeed, if $y\in C'_x$ then by definition $y\in (C^y)^{-1}$ so it follows that $|C'_x|=1$ by the definition of $Y'$.

It follows that $C'$ gives a finite-to-one definable map between $X$ and $Y'$.
\end{proof}

\begin{corollary}\label{C:l.a.s.i to one dist sorts in vminimal}
     Every infinite  set interpretable in $\CK$ is locally almost strongly internal to one of the distinguished sorts.
\end{corollary}
\begin{proof}
    Let $G=X/E$ be an infinite interpretable set, where $X\subseteq K^n$. By \cite[Proposition 5.6 and Proposition 5.5]{HaHaPeVF} there exists an infinite definable subset $S\subseteq G$, in definable bijection with an infinite definable subset of $K/E'$, and a definable finite-to-finite correspondence between $S$ and either $K$, $K/\CO$, $\bk$ or $\Gamma$. Since each of these sorts is either a field or an SW-uniformity, we can conclude by Lemma \ref{L:from correspondece to lasi in vminimal}.
\end{proof}

Here is the main theorem in the V-minimal setting.

\begin{theorem}\label{T: groups in vminimal}
    Let $G$ be an infinite group interpretable in $\CK$. Then there exists a finite normal subgroup $H\sub G$, such that $G/H$
     is locally strongly internal to at least one distinguished sort $D$, and  for every  such $D$,  there exists a type-definable subgroup $\nu_D$ of $G/H$, with the following properties:   
    \begin{enumerate}

      \item If $D=\bk$ then $\nu_D=N$ is a definable normal subgroup of $G/H$ of Morley Rank $n$, such that $N$   is definably isomorphic to a $\bk$-algebraic group.
      
         \item If $D=K$  then $\nu_D$ is definably isomorphic to a type-definable group $\nu'\vdash K^r$, for some $r$.

        \item If $D=\Gamma$ then $\nu_D$ is definably isomorphic to  a type-definable subgroup $\nu'$ of $\la \Gamma^r,+\ra$, for some $r$.
        \item If $D=K/\CO$ then there exists an infinite definable subgroup $H$ of $G$,  such that $\nu_D \vdash H\le G$, and $H$ is definably isomorphic to a subgroup of $\la (K/\CO)^r,+\ra$ for some $r$.
\item \begin{enumerate}
    \item $\dpr(\nu_D)=n$ the almost $D$-critical rank of $G$,  and the group is minimal among all type-definable subgroups of $G/H$ of dp-rank $n$ that are
    %$\nu_D$ is smallest among all type-definable subgroups of $G$ of dp-rank $n$ and 
    strongly internal to $D$. 
        \item $\nu_D$ is torsion free, so in particular $G$ has unbounded exponent. 
    
\end{enumerate}
        \end{enumerate}    
    Moreover, if $G$ is dp-minimal then $G$ is abelian-by-finite.
\end{theorem}

 \begin{proof}
    Let $G$ be in an infinite interpretable group. By Lemma \ref{C:l.a.s.i to one dist sorts in vminimal}, $G$ is locally almost strongly internal to one of the distinguished sorts $D$. By Fact \ref{E:dist sorts are vicinic}, every such $D$ is vicinic and by Fact \ref{E:interp groups in dist sorts are (almost) D-groups}, $G$ is  an almost $D$-group. 
    
   Assume first that $D=\bk$. By Proposition \ref{P: si to sm}, there exist  definable normal subgroups, $H, H_1\trianglelefteq G$, with $H\sub H_1$ finite and $N:=H_1/H$ strongly internal to $\bk$. Since $\bk$ is a stably embedded algebraically closed field, $N$ is $\bk$-algebraic (\cite{PoiFields}). 
  
    It is well known that every algebraic group over an algebraically closed field of characteristic $0$ has unbounded exponent.

Assume that $D=\Gamma,K$ or $K/\CO$. By Fact \ref{E:dist sorts are vicinic}, each of those is vicinic and by Fact \ref{E:interp groups in dist sorts are (almost) D-groups}, $G$ is an almost $D$-group.  Thus, there are $H\trianglelefteq G$,  a finite normal subgroup as provided by Proposition \ref{P: G/H s.i.} and $\nu_D$ a type-definable subgroup of $G/H$, as in Proposition \ref{P: nu}.
    
    From this point on the proof of clauses 2-4 of Theorem \ref{T:groups in tconvex} goes through verbatim for the above three sorts, using the fact that Since $\Gamma$ is a pure ordered vector space, and therefore, locally linear. This ends the proof of 1-4.

The proof of 5(a),(b) is identical to the  proof of 4(a),(b) in Theorem \ref{T:groups in tconvex}, noting that the almost $D$-critical rank of $G$ is the $D$-critical rank of $G/H$. The proof that a dp-minimal group is abelian-by-finite is identical to the one in Theorem \ref{T:groups in tconvex}.
    
\end{proof}

\begin{remark}
 It may be worth pointing out that the analogue of Theorem \ref{T: groups in vminimal} is wrong in $\mathrm{ACVF}_{p,p}$: interpretable groups  (and even definable ones) need not have unbounded exponent, and as shown by Simonetta \cite{Simonetta} dp-minimal such groups need not be abelian-by-finite. Our methods do carry us a long way in $\mathrm{ACVF}_{p,p}$ (as well as in $\mathrm{ACVF}_{0,p}$), and the failure of our main results in this case seems local. More explicitly,  Lemma \ref{L:from correspondece to lasi in vminimal} holds in any C-minimal valued field, and so do \cite[Proposition 5.5, Proposition 5.6]{HaHaPeVF}. Since all the distinguished sorts in the $C$-minimal case are either fields or SW-uniformities, Corollary \ref{C:l.a.s.i to one dist sorts in vminimal} goes through to assure that and group interpretable in $\mathrm{ACVF}_{p,p}$ or $\mathrm{ACVF}_{0,p}$ is almost locally strongly internal to one of the distinguished sorts. Thus, interpretable groups are almost $D$-groups for $D$ a vicinic sort -- and our construction of the infinitesimal group $\nu_D(G)$ goes through unaltered. Thus, it seems that clauses (1)-(4) of Theorem \ref{T: groups in vminimal} could still be true, if -- in case $D=K/\CO$ -- we require only that $\nu$ be type-definable in $K/\CO$ (not necessarily a subgroup). The reason for this change is that we do not know that $K/\CO$ is locally linear in this setting. 
 %Corresponding results still hold in $\mathrm{ACVF}_{0,p}$  as well.
\end{remark}
 
 An interesting question we leave open is, therefore: 
 \begin{question}
   Let $\mathcal K$ be a saturated enough algebraically closed valued field or, more generally, a pure dp-minimal valued field.   Let $G$ be a dp-minimal group interpretable in $\CK$. If $G$ is locally almost strongly internal to $K$, is $G$ abelian-by-finite? 
 \end{question}

\subsection{$p$-adically closed valued fields}
Let $\CK=(K,v)$ be a sufficiently saturated $p$-adically closed valued field.
% (or a model of $\mathbb Q_{p}^{an}$ or,  more generally, a $P$-minimal 1-h-minimal field with definable Skolem functions).  
% As we noted before, it has definable Skolem functions.

\begin{lemma}\label{L:l.s.i to one dist sorts in $p$-adically closed}
     Every infinite  set interpretable in $\CK$ is locally almost strongly internal to one of the sorts $K,K/\CO$ and $\Gamma$.
\end{lemma}
\begin{proof} Assume that $X/E$ is a definable quotient in $\CK$. We first apply \cite[Proposition 5.5]{HaHaPeVF} (noting that $\CK$ satisfies the necessary assumptions, see \cite[Proposition 5.8]{HaHaPeVF}), and conclude that there are an infinite  $T\sub X/E$, a distinguished sort $D$, an infinite $D'\sub D$ and a definable finite-to-finite correspondence $C\sub T \times D'$.

When $D=K$,  it is an SW uniformity and the result follows from Lemma \ref{L:from correspondece to lasi in vminimal}.

When $D=\Gamma$, then, since $\Gamma$ is linearly ordered, the correspondence gives rise to a finite-to-one function from $T$ into $\Gamma$, as needed.

When $D=K/\CO$, the second projection map $\pi_2:C\to K/\CO$ proves that $C$ is almost strongly internal to $K/\CO$.   We now consider the first projection $\pi_1:C\to T\sub X/E$. By Lemma \ref{L:asi diagaram K/O} (2), there exists an infinite subset of $T$ which is almost strongly internal to $K/\CO$.
\end{proof}

\begin{theorem}\label{T: groups in padics}
    Let $G$ be an infinite group interpretable in $\CK$. Then there is a finite normal   $H\trianglelefteq G$ such that $G/H$ is locally strongly internal to at least one distinguished sort $D$, and for each such $D$, there exists a type-definable subgroup $\nu_D$ of $G$ with the following properties:    
    \begin{enumerate}
        \item If $D=K$  then $\nu_D$ definably isomorphic to a type-definable $\nu'\vdash K^r$, for some $r$. 
        \item If $D=\Gamma$ then $\nu_D$ is definably isomorphic to a type-definable subgroup of $\la \Gamma^r,+\ra$ for some integer $r$.
        \item If $D=K/\CO$ then there exists an infinite definable subgroup $H$ of $G$,  such that $\nu_D \vdash H\le G$, and $H$ is definably isomorphic to a subgroup of $\la (K/\CO)^r,+\ra$ for some $r$.
\item \begin{enumerate}
    \item $\dpr(\nu_D)=n$ the almost $D$-critical rank of $G$, and $\nu_D$  is minimal among all type-definable subgroups of $G/H$ of dp-rank $n$ that are
    %$\nu_D$ is smallest among all type-definable subgroups of $G$ of dp-rank $n$ and 
    strongly internal to $D$. 
        \item $\nu_D$ has unbounded exponent, so in particular so is $G$. 
    
\end{enumerate}
    \end{enumerate}

    Moreover, if $G$ is dp-minimal then $G$ is abelian-by-finite.
\end{theorem}
\begin{proof}
    Let $G$ be in an infinite interpretable group. By Lemma \ref{L:l.s.i to one dist sorts in $p$-adically closed}, $G$ is locally almost strongly internal to one of the distinguished sorts $D$. By Fact \ref{E:dist sorts are vicinic}, every such $D$ is vicinic and by Fact\ref{E:interp groups in dist sorts are (almost) D-groups} $G$ is  an almost $D$-group.
    Let $H\trianglelefteq G$ be a finite normal subgroup as provided by Proposition \ref{P: G/H s.i.} and let $\nu_D$ be the type definable subgroup of $G/H$ as in Proposition \ref{P: nu}.
    
    If $D=K$ then,  as in the proof of Theorem \ref{T:groups in tconvex}(1), $\nu_D$ is definably isomorphic to a type-definable group $\nu'\vdash K^r$, for some $r$,
    and it is moreover torsion-free.
    
    If $D=\Gamma$ then it  is a $\mathbb Z$-group so, by Fact \ref{E:examples of loc linear},  it  is locally linear. By Proposition \ref{P:infini_vicinity- def iso to subgroup}, $\nu_D$ is definably isomorphic to a type-definable subgroup of $\la \Gamma^r,+\ra $, for some $r$, so in particular it is torsion-free, so has  unbounded exponent.  
    
     If $D=K/\CO$ then it is locally linear by Fact \ref{E:examples of loc linear}, so by Proposition \ref{P:def group into K/O}, there exists a definable subgroup $H\sub G$, $\nu_D\vdash H$, such that $H$ is definably isomorphic to a subgroup of $\la (K/\CO)^r,+\ra$. It remains to see that $\nu_{K/\CO}$ has unbounded exponent. 
     By Fact \ref{F:F is dense in Z-radius balls}(3), for every $n$ there are only finitely many elements in $(K/\CO)^r$   of order $n$, thus every infinite subgroup of $(K/\CO)^r$ has unbounded exponent. 

      The proof of (4)(a),(b) is identical to the previous cases, as is the fact that a dp-minimal group is abelian-by-finite.

    %and by Corollary \ref{C:dp min is pure} exactly one of (1)-(3) can occur. 
\end{proof}

\subsection{Concluding remarks}
The combination of Theorem \ref{T:groups in tconvex}, Theorem \ref{T: groups in vminimal} and Theorem \ref{T: groups in padics} implies Theorem \ref{intro-1} and Theorem \ref{intro-2}, as stated in the introduction. In fact, they give stronger results than appear in the more concisely stated Theorem \ref{intro-2}: Let us points out some of the differences.

\begin{itemize}

\item Canonicity of $\nu_D$. This  can be seen by the minimality of $\nu_D$ among all type-definable subgroups of the same rank that are strongly internal to $D$, as well as its invariance under definable automorphisms of $G$ and under automorphisms of $\CK$ (Proposition \ref{P: nu}).

\item  In all cases except the case $D=\bk$ in the V-minimal setting and $D=K/\CO$ in the $p$-adic setting, we show that $\nu_D$ is in fact torsion-free and not only of unbounded exponent.

\item  In the power bounded $T$-convex setting, for all distinguished sorts $D$ to which $G$ is locally almost strongly internal, the group $\nu_D$ is a type-definable subgroup of $G$ and not of $G/H$, for some finite $H$.
%In the V-minimal and $p$-adically closed cases, each sort $D$ to which $G$ is locall almost strongly internal,   gives rise to a finite normal subgroup, call it $H_D$, such that $G/H_D$ is locally strongly internal to $D$. It is not hard to show that we can, in fact, find a single group $H$ such that for any distinguished sort $D$ that $G$ is locally almost strongly internal to $D$, $G/H$ is locally strongly internal to $D$.

\end{itemize}

\section{Examples}\label{S: examples}
We study some examples of interpretable groups in the valued fields we considered here and see how our results are reflected in those   examples. The examples are, mostly, common to all contexts, but their nature may vary between the different settings.  \\%\footnote{Obviously, these examples appear in more general valued fields}.

Let $\CK=(K,v,\dots)$ be some expansion of a valued field.
 
\begin{example}[$K/\m$] 
    If $\CK$ is $V$-minimal or power bounded $T$-convex then as $\m$ is an additive subgroup of $K$, $K/\m$ is an infinite interpretable group. Since $\CO\subseteq K$, $(\bk,+)$ is a subgroup of $K/\m$. I.e. $K/\m$ is locally strongly internal to $\bk$. If $\CK$ is $p$-adically closed, then $K/\m$ and $K/\CO$ are definably isomorphic.
\end{example}

\begin{example}[$K/\CO\rtimes \CO^\times$]
    Multiplication defines an action of  $\CO^\times$ on $K/\CO$ by automorphisms and  $G=K/\CO\rtimes \CO^\times$ is a solvable group of class $2$ since $G'\leq K/\CO$ (so, in particular, it is not abelian-by-finite). Its dp-rank is  $2$ (the universe of the group being $K/\CO\times \CO^\times$). It is locally strongly internal to both $K$ and $K/\CO$: the subgroup $0\times \CO^\times$ witnesses the former and $K/\CO\times \{1\}$ the latter.   The  geometric dimension of $G$, as induced on $\CK^{eq}$ from the $acl$-dimension (see \cite{Gagelman}) is $1$, since this is the geometric dimension of $K/\CO\times \CO^\times$. The example thus shows that Theorem \ref{intro-1}  does not extend to definable groups of geometric dimension 1. 
\end{example}

\begin{example}[RV$_\gamma$]
    (1) Assume first that $(K,v)$ is $p$-adically closed and identify the standard part of $\Gamma$ with $\mathbb{Z}$. 
    %For $\gamma\in \Gamma_{\geq 0}$, let RV$_\gamma$ be $K^\times/(1+\m_\gamma)$, where $\m_\gamma=\{x\in K: v(x)>\gamma\}$. (One may also consider $K^\times/(1+\m^n)$, for natural number $n$, but this equals     RV$_{n-1}$ ) 
    Since in $p$-adically closed fields there are definable angular component maps $\mathrm{ac}_n: K^\times\to \CO/\m_{n-1}$, for any (standard) natural $n$, we can identify $\Gamma$ with the definable subset of RV$_{n-1}$ defined by $\mathrm{ac}_n(x)=1$. Thus, RV$_n$ is locally strongly internal to $\Gamma$.
    
    For $\gamma\in \Gamma$ non-standard (i.e., $\gamma >n$ for all $n\in \Nn$) the picture is different. Since $\gamma>\Zz$ we have $|\CO^\times/\m_{n-1}|\leq |\CO^\times/\m_\gamma|$ for all standard $n\in \Nn$ and as the left-hand side is unbounded with $n$ it follows that  $\mathcal{O}^\times/\m_\gamma$ is infinite. The map $a+\m_\gamma\mapsto a(1+\m_\gamma)$ is a definable injection from $\CO^\times/\m_\gamma$ into RV$_\gamma$. Since $\Gamma$ is discrete $\m_\gamma$ is a closed ball (of valuative radius $\gamma+1$), so  $\CO^\times/\m_\gamma$ is in definable bijection with a subset of $K/\CO$ and thus locally strongly internal to $K/\CO$. So, in this case $RV_\gamma$ is locally strongly internal to $K/\CO$.
    
    In fact, if $\gamma$ is non-standard, and $\delta\in \Gamma$ is such that $2\delta >\gamma$ and  such that  $\gamma -\delta>\mathbb{Z}$.  Then  $(1+\m_\delta)/(1+\m_\gamma)$ is an infinite subgroup of $\mathrm{RV}_\gamma$ definably isomorphic to the additive group $\m_\delta/\m_\gamma$. 
    Indeed, the definable map $a+\m_\gamma\mapsto (1+a)(1+\m_\gamma)$ for $a\in \m_\delta$ from $\m_\delta/\m_\gamma$ to $(1+\m_\delta)/(1+\m_\gamma)$ is a bijective group homomorphism.
    
    (2) Now assume that $(K,v)$ is power-bounded T-convex  or V-minimal (actually $\Gamma$ dense and $\bk$ infinite is sufficient). 
    %We let RV$_\gamma=K^\times/(1+\m_\gamma)$ be as above (note that for $n\in \mathbb N$, $\m^n=\m$ in this setting).
    
     Consider the definable subset $\CO^\times/(1+\m_\gamma)$ of RV$_\gamma$. We claim that there is a definable injection from $\bk$ into $\CO^\times/(1+\m_\gamma)$; thus RV$_\gamma$  is locally strongly internal to $\bk$. Indeed, pick some $t\in K$ with $v(t)=\gamma$, thus $a+\m_\gamma \mapsto at^{-1}(1+\m_\gamma)$ (with $v(a)=\gamma$) is an injection of $\CO_\gamma/\m_\gamma$ (where $\CO_\gamma=\{x:v(x)\geq \gamma\}$) into $\CO^\times/(1+\m_\gamma)$; finally note that $\CO_\gamma/\m_\gamma\cong \bk$. In particular, $RV_\lambda$ is locally strongly internal to $\bk$.
\end{example}

\begin{example}
The next example shows, in the power bounded $T$-convex case, the necessity for the group $\nu\sub G$ to be type-definable (rather than definable). It is  similar to examples from \cite{Strz}.

We consider the two o-minimal structures, $\Gamma$ and $\bk$ and fix a positive $\gamma_0\in \Gamma$. We start with $H=\Gamma\times (\bk,+)$ and in it consider the group $\Lambda$ generated  by $\la \gamma_0,1\ra$.

We let $S= \Gamma \times [0,1)\sub H$, and on $S$ define the operation
\[(x_1,y_1)\oplus (x_2,y_2)=
\begin{cases}
          (x_1+x_2,y_1+y_2) & y_1+y_2<1\\
                   (x_1+x_1-\gamma_0,y_1+y_2-1) & y_1+y_2\geq 1          
\end{cases}   \]

This is a group operation which makes $G=(S,\oplus) $ an interpretable group isomorphic to the quotient of the  group $\la S\ra \sub H$ by the subgroup $\Lambda$.

 The group $G$ has one definable subgroup $\Gamma\times 0$ isomorphic to $\Gamma$.
However, since $\Gamma$ and $\bk$ are foreign, the only definable subsets of $G$ which are (almost) strongly internal to $\bk$ are subsets of $\{\gamma\}\times [0,1)$, $\gamma\in \Gamma$, or finite unions of such. It is not hard to see that no finite union of such sets gives rise to a {\bf definable} subgroup of $G$. On the other hand, $G$ has a type definable subgroup $\nu$ which is definably isomorphic to the infinitesimals in $(\bk,+)$.
\end{example}

\section{The distinguished sorts are foreign}\label{S: foreig}

 Using the techniques developed in \cite{HaHaPeVF} and in the present work,  we can prove a certain orthogonality result for the  distinguished sorts. It shows, essentially, that in our setting, a definable set can be almost strongly internal to at most one distinguished sort. More precisely:

\begin{definition}
    Let $\CM$ be any structure. Two $\CM$-definable sets $D_1,D_2$ are \emph{foreign} if there is no definable finite-to-finite correspondence $C\subseteq X\times Y$, where $X\subseteq D_1^n$ and $Y\subseteq D_2^m$ are definable infinite subsets. 
\end{definition}

We leave to the reader the  proof of the following easy observation:
\begin{lemma}\label{L:reduce correspondence to one variable}
     If there exists a definable finite-to-finite correspondence between infinite definable subsets of $D^n_1$ and $D_2^m$ then there also exists one between infinite definable subsets of $D_1$ and $D_2$.
\end{lemma}

{\em We assume now that $\CK$ is either power bounded $T$-convex, V-minimal or $p$-adically closed.}
\begin{proposition}
    \label{F:any two dist sorts are foreign}
    Any two distinct distinguished sorts in $\CK$ are  foreign.
\end{proposition}
\begin{proof}
Most of the cases could have been proved earlier using more elementary methods, but we find it to be a good application of our results here.
 
    First, assume that $D_1,D_2$ are sets that are not foreign, namely (applying Lemma \ref{L:reduce correspondence to one variable}) there exists  a definable finite-to-finite correspondence between respective infinite subsets thereof. We repeatedly use  \cite[Lemma 4.28]{HaHaPeVF}, stating that if $D_1$  eliminates finite imaginaries (EfI for short) and $D_2$ either has (EfI) or supports an SW-uniformity then $D_1$ is locally strongly internal to $D_2$. We shall use the fact that expansions of fields have  (EfI).
    
    \vspace{.1cm}
    
    \underline{$K$ and $\bk$ are foreign:} Note first that if $\CK$ is $p$-adically closed then there is nothing to prove ($\bk$ is finite), so we assume we are not in this case.
    
    By the above, $K$ is locally strongly internal to $\bk$. By \cite[Theorem 4.21]{HaHaPeVF}, in the power-bounded T-convex case, or \cite[Theorem 4.24]{HaHaPeVF}, in the V-minimal case, $K$ is definably isomorphic to $\bk$.  But $K$ is a valued field and $\bk$ is not: it is strongly minimal in the $V$-minimal case and $o$-minimal in the $T$-convex case. 
    
    \vspace{.1cm}
    \underline{$K$ and $\bk$ are foreign to $\Gamma$:} 
    Since $\Gamma$ has definable choice in all settings, it eliminates imaginaries so by the above, if the sorts were not foreign  we would get $K$ (respectively,  $\bk$) locally strongly internal to $\Gamma$, contradicting \cite[Proposition 6.29]{HaHaPeVF}.

    \vspace{.1cm}
    \underline{$K$ and $\bk$ are foreign to  $K/\CO$:}
    In the $p$-adcially closed, we only need to check that $K$ and $K/\CO$ are foreign, which follows from \cite[Proposition 6.29]{HaHaPeVF}.
    
    In the remaining cases, $K/\CO$ is an SW-uniformity and $K$, $\bk$ are fields, so satisfy (Efi), hence by \cite[Proposition 4.28]{HaHaPeVF} any correspondence between $K/\CO$ and $K$ or $\bk$  implies that  $K$, $\bk$ are locally strongly internal to $K/\CO$. This contradicts \cite[Proposition 6.24]{HaHaPeVF}.
    
    \vspace{.1cm}
    \underline{$\Gamma$ and $K/\CO$ are foreign:}
    Assume towards contradiction that this is not the case.  We first claim that $\Gamma$ is locally strongly internal to $K/\CO$.
    
    Indeed, $\Gamma$ eliminates (finite) imaginaries and  in the V-minimal and $T$-convex cases,  $K/\CO$ is an SW-uniformity, so by \cite[Lemma 4.28]{HaHaPeVF}, $\Gamma$ is locally strongly internal to $K/\CO$ in these cases. In the $p$-adically closed setting, we  apply Lemma \ref{L:asi diagaram K/O} (with $X=K/\CO$ and $T=\Gamma$), and conclude that $\Gamma$ is locally almost strongly internal to $K/\CO$. However, since $\Gamma$ is ordered, it follows that $\Gamma$ is in fact locally strongly internal to $K/\CO$.

    Since $\Gamma$ is an interpretable group, we may apply Proposition \ref{P:def group into K/O} and conclude that there exists a definable infinite subgroup $G_1\le \Gamma$, which is definably isomorphic to a definable subgroup of $(K/\CO)^r$, for some integer $r$. 
    
     In $(K/\CO)^r$, every infinite definable subgroup has many infinite definable proper subgroups (intersection with balls).
    However, in the V-minimal and $T$-convex power bounded cases, $\Gamma$ is an ordered vector space thus has no infinite definable subgroups other than itself, contradiction.
   
   In the $p$-adically closed setting,\footnote{A direct proof, which does not make use of our work here, and is based on \cite[Proposition 3.1]{ChCuLe}, was suggested to us by P. Cubides Kovacsics.} $\Gamma$ is torsion-free while every definable subgroup of $(K/\CO)^r$ has torsion (Lemma \ref{L:Tor of ball}(2)), leading also to a contradiction.
\end{proof}

\begin{question}
 Note that  a definable quotient of a distinguished sort $D$ by a definable equivalence relation with infinitely many infinite classes can be foreign to $D$ itself, and thus it is locally (almost) strongly internal to one of the other sorts. In fact, the sorts $K/\CO$, $\Gamma$, and $\bk$ are all quotients of $K$, or some subset of $K$, by such an equivalence relation.

    By repeatedly taking appropriate quotients, one can alternate between local strong internality to two different sorts:
     Consider $G=K/\m$ (locally strongly internal to $\bk$), and the definable subgroups $r\CO$ and $s\CO$ for $r,s\in K$ such that  $v(r)<v(s)<0$.  Then $(r\CO/\m)/(s\CO/\m)\cong r\CO/s\CO\cong (r/s)\CO/\CO$ (all isomorphisms definable), with the latter definably isomorphic to a ball in  $K/\CO$ (so obviously locally strongly internal to $K/\CO$). Every ball in $K/\CO$ has a quotient definably isomorphic to a subgroup of  $K/\m$ thus this quotient is locally strongly internal to $\bk$, and we can choose the subgroups along the way so this process will go on indefinitely. 
     
     It would be interesting to known which of the distinguished sorts may appear in such a sequence of quotients.
\end{question}

Proposition \ref{F:any two dist sorts are foreign} allows us to show that dp-minimal groups interpretable in $\CK$ are pure in the following sense: 

\begin{corollary}\label{C:dp min is pure}
    If $G$ is interpretable in $\CK$ and $X_1,X_2\sub G$ are almost strongly internal to foreign sorts $D_1, D_2$, respectively, then $\dpr(G)\geq \dpr(X_1)+\dpr(X_2)$. In particular, if $\dpr(G)=1$ then $G$ can be locally almost strongly internal to at most one distinguished sort.
\end{corollary}
\begin{proof} 
Assume toward contradiction that $\dpr(G)<\dpr(X_1)+\dpr(X_2)$, and 
    consider the function $f:X_1\times X_2\to G$, defined by $f(x_1,x_2)=x_1\cdot x_2$.
    
    We have $\dpr(X_1\times X_2)=\dpr(X_1)+\dpr(X_2)$, thus 
    by the dp-rank assumption, $f$ cannot be everywhere finite-to-one. Hence there is some $g\in G$ such that $f^{-1}(g)$ is infinite.
    But, by its definition, $f^{-1}(g)\sub X_1\times X_2$ is the graph of a bijection between (infinite) subsets of $X_1$ and $X_2$.
    This gives rise to a finite-to-finite correspondence between infinite subsets of $D_1$ and $D_2$, contradiction.
\end{proof}

Fact \ref{F:any two dist sorts are foreign} and Corollary \ref{C:dp min is pure} raise interesting questions about the possible dp-rank of subsets of $G$ that are almost strongly internal to the distinguished sorts. For example:
\begin{question}
    For a distinguished sort $D$ let $AlC_D(G)$ denote the almost $D$-critical rank of $G$ (with $AlC_D(G)=0$ if $G$ is not locally almost strongly internal to $D$). Is it true that 
    \[\dpr(G)=AlC_K(G)+AlC_\bk(G)+AlC_\Gamma(D)+AlC_{K/\CO}(G)?\]
    
  % Are there always definable  $X_1,\ldots, X_4\sub G$ (some possibly empty), with $X_i$ almost strongly internal to $K, \bk ,\Gamma, K/\CO$, respectively,   such that $\dpr(G)=\sum_{i=1}^4 \dpr(X_i)$? 
\end{question}

We end with an  example of a dp-minimal valued field where the distinguished sorts are not foreign.

\begin{example}
    Let $\CR$ a sufficiently saturated extension of $\Rr_{\exp}$ and let $\CK$ be $\CR$ expanded by a predicate for the convex hull of $\mathbb{Z}$, which we denote by $\CO$. So $\CK$ an exponential $T$-convex valued field and, therefore, dp-minimal ($\CR$ is o-minimal and $\CO$ is externally definable). We claim that $K/\CO$ is strongly internal to $\Gamma$.
    
    We first note that $\exp(\CO)=\CO_{>0}\setminus \m$. Indeed, for the right-to-left, since $\log$ is a $\emptyset$-definable continuous function, if $x\in \CO_{>0}$ then $\log(x)\in \CO$. For the other direction, assume for a contradiction that $a=\exp(b)\in \m$ for some $b\in \CO$. Then $a^{-1}=\exp(-b)\notin \CO$, contradicting $T$-convexity.
    
    Thus (and as $\CO^\times=\CO\setminus \m$), $\exp$ induces map $E: K/\CO\to K^\times/\CO^\times$ given by $E(x+\CO):=\exp(x)+\CO^\times$. It is easy to check that $E$ is a homomorphism of (ordered) groups. It is injective because $\exp$ is. 
\end{example}

\appendix

\section{Endowing an infinitesimal group in RCVF with a linear order}
Here we show that if $\CK$ is power bounded $T$-convex, $G$ is a dp-minimal group interpretable in $\CK$ then  $\nu(G)$,  the  associated infinitesimal subgroup,  is an ordered group with respect to the induced ordering:

\begin{proposition}
	Let $G$ be a dp-minimal interpretable group in a power bounded $T$-convex structure, $\CK$ and assume that it is locally strongly internal to one of the distinguished sorts $D$. Then the group $\nu$ provided by Theorem \ref{T: groups in padics} is ordered with respect to the order induced from its embedding into $D$.
\end{proposition}
\begin{proof} 
We shall be brief. Assume that $\CK$ is sufficiently saturated. 
Since $K/\CO$ and $\Gamma$ are ordered groups and $\nu$ is a subgroup,  we have nothing to prove in these cases. We prove the result in case $D$ is $K$. The proof translates verbatim to the case where $D=\textbf{k}$ since it uses  only weak o-minimality of $K$ and \cite[Corollary  2.8]{vdDries-Tconvex} asserting that any definable function $f: K\to K$ in a $T$-convex structure is piece-wise  monotone. By o-minimality, this is also true in $\textbf{k}$. 

	The following is a simple corollary of the piece-wise monotonicity of definable functions.
	\begin{claim}
		 Assume that $f: K\to K$ is a definable, continuous, open and injective partial function with open domain. Then $f$ is locally strictly monotone at every point.
	\end{claim}
Our goal is to prove: If $\nu\sub K$ is the type-definable infinitesimal  neighbourhood of $e\in K$, endowed with the $K$-ordering and  a $\CK$-definable topological group operation then left multiplication is order preserving (by symmetry, the same is true for right multiplication).

	Let $e\in \nu$ be the identity element and $\lambda(x,y)$ a $\0$-definable function whose restriction to $\nu$  is the group multiplication. We may assume, by compactness, that $\lambda$ is defined and continuous on $U\times U$ for some $K$-definable open $U$, $\nu\vdash U$,  which satisfies:
\begin{list}{$\bullet$}{}
	\item For all $x\in U$ the function $\lambda_x(y):= \lambda(x,y)$ is an injective open map. In addition, $\lambda(x,e)=\lambda(e,x)=x$ for all $x\in U$.

We may further find a definable open $V\sub U$, $\nu \vdash V$, such  that:
	\item For all $x\in V$ there exists (a unique) $y\in U$ such that $\lambda(x,y)=\lambda(y,x)=e$.  By abuse of notation, we let $x^{-1}$ denote this $y$.
	
	\item $V=V^{-1}$, $V\cdot V\cdot V\sub U$ and $\lambda$ is associative on $V$, whenever defined.
\end{list}
Absorbing parameters into the language, assume that $U$ and  $V$ above are $\emptyset$-definable.

By the above claim, for every $g\in U$ the function $\lambda_g$ is locally (strictly) monotone at every point. 	Let $\widehat \CK\succ \CK$ be an $|K|^+$-saturated extension. We first show that for every $g\in \nu(\widehat \CK)$, the function $\lambda_g$ is  {\em locally} strictly increasing at $e$.

Let $W^+$ (resp. $W^-$) be the set of $g\in U(\widehat \CK)$ such that $\lambda_g$ is locally strictly increasing (reps. decreasing) at $e$. Both sets are  $\0$-definable, hence  $\nu^+:=\nu\cap (e,\infty)$, which -- by weak o-minimality -- is a complete type over $\emptyset$, is concentrated on one of  $W^+$ and $W^-$, and the same for $\nu^-:=\nu\cap (-\infty,e)$. We claim that both are concentrated on $W^+$, and hence $\nu\vdash W^+$ (clearly,  $e\in W^+$).

Indeed, assume towards a contradiction that, say, $\nu^-\vdash W^-$, and fix any $a\in \nu^-(\widehat \CK)$ (in particular, $\dpr(a/M)=1$). Since $a\in W^-$, the function $\lambda_a$ is strictly decreasing on some open interval  $J\ni e$.  As $ae=a$ belongs to the open set $\nu^-$, it follows by continuity that there exists $b\models \nu^-\cap J$ sufficiently close to $e$, such that $ab\in \nu^-$.
 But then, $\lambda_{ab}=\lambda_a\circ\lambda_b$, $\lambda_b$ is decreasing at $e$ and $\lambda_a$ is decreasing at $b$ so $\lambda_{ab}$ is increasing at $e$, contradicting the fact that $ab\in \nu^-(\widehat \CK)\sub  W^-$.
 
 The cases where $\nu^+\vdash W^-$ leads to a contradiction in the same way.
 We may therefore conclude that for every $g\in \nu(\widehat \CK)$, $\lambda_g$ is locally increasing at $e$. By compactness, we may assume that $U\sub W^+$.

 Fix $g\in V(\CK)$ (where associativity holds) such that $\dpr(g)=1$.  The function $\lambda_g$ is continuous at $e$, $\lambda_g(e)=g$, hence $\lambda_g(\nu)=\nu(g)$. Furthermore, since $g\in K$, by  our assumptions, $\lambda_g$ is strictly increasing on some $\CK$-definable open interval $I\ni e$. By \ref{Gen-Os in SW} we may choose $I$ to be $A$-definable, $A\sub M$, such  that $\dpr(g/A)=1$. It follows that for all $c\models \nu(g)$, $\lambda_c$ is strictly increasing on $I$, so in particular on $\nu(\widehat \CK)$.

%We need to see that for all $h\in \nu(\widehat \CK)$, $\lambda_h$ is strictly increasing on $\nu$.
Fix any $h\in \nu$, we need to see that $x\mapsto hx$ is increasing on $\nu$. We write $h=\lambda_g^{-1}(c)$, for some $c\in \nu(g)$. Now, for $x\in \nu(\widehat \CK)$, we have $$h\cdot x=\lambda_g^{-1}(c)\cdot x=\lambda_{g^{-1}}(c)\cdot x=\lambda_{g^{-1}}(\lambda_c(x)),$$ where the right equality follows from the associativity on $V$. Now, if $x<y\in \nu(\widehat \CK)$ then
$\lambda_c(x)<\lambda_c(y)$ and hence (because $\lambda_g$ and its inverse  are increasing on $\nu$), we have
$$h\cdot x=\lambda_g^{-1}(\lambda_c(x))<\lambda_g^{-1}(\lambda_c(y))=h\cdot y.$$ It follows that 
multiplication is order preserving.
\end{proof}

\bibliographystyle{plain}
\bibliography{harvard}

\end{document}